\documentclass[12pt,leqno]{amsart}
\usepackage{a4wide}
\usepackage{amssymb}
\usepackage{amsmath}
\usepackage{amsthm}
\usepackage{verbatim}
\usepackage{txfonts}
\usepackage{fullpage}
\usepackage[all,cmtip]{xy} 

\numberwithin{equation}{section}

\theoremstyle{plain}

\newcommand{\K}{\ensuremath{{\mathbb{K}}}}
\newcommand{\A}{\ensuremath{{\mathbb{A}}}}
\newcommand{\C}{\ensuremath{{\mathbb{C}}}}
\newcommand{\Z}{\ensuremath{{\mathbb{Z}}}}
\renewcommand{\P}{\ensuremath{{\mathbb{P}}}}
\newcommand{\Q}{\ensuremath{{\mathbb{Q}}}}

\newcommand{\F}{\ensuremath{{\mathbb{F}}}}

\newcommand{\E}{\ensuremath{{\mathbb{E}}}}
\newcommand{\I}{\ensuremath{{\mathbb{I}}}}
\newcommand{\B}{\ensuremath{{\mathbb{B}}}}
\newcommand{\charf}{\textbf{1}}
\newcommand{\Vol}{\text{Vol}}
\newcommand{\GL}{\ensuremath{{\text{GL}}}}
\newtheorem{theo}{Theorem}[section]
\newtheorem{lem}[theo]{Lemma}
\newtheorem{prop}[theo]{Proposition}
\newtheorem{cor}[theo]{Corollary}

\theoremstyle{remark}
\newtheorem{rem}[theo]{Remark}
\newtheorem{example}[theo]{Example}
\theoremstyle{definition}
\newtheorem{defn}[theo]{Definition}

\newtheorem*{cor*}{Corollary}

\newcommand{\zxz}[4]{\begin{pmatrix} #1 & #2 \\ #3 & #4 \end{pmatrix}}
\renewcommand{\Re}{\operatorname{Re}}
\newcommand{\Hom}{\operatorname{Hom}}

\title{Cuspidal part of an Eisenstein series restricted to an index 2 subfield}

\begin{document}
\author{Yueke Hu}

\address{Department of Mathematics, University of Wisconsin Madison, Van Vleck Hall, Madison, WI 53706, USA}
\email{yhu@math.wisc.edu}

\begin{abstract}
 Let $\E$ be a quadratic extension of a number field $\F$. Let $E(g, s)$ be an Eisenstein series on $GL_2(\E)$, and let $F$ be a cuspidal automorphic form on $GL_2(\F)$. We will consider in this paper the following automorphic integral:
$$\int\limits_{Z_{\A}\GL_{2}(\F)\backslash \GL_{2}(\A_{\F})} F(g)E(g,s) dg.$$
This is in some sense the complementary case to the well-known Rankin-Selberg integral and the triple 
product formula. We will approach this integral by Waldspurger's formula. We will discuss when the integral is automatically zero, and otherwise the L-function it represents. 
We will calculate local integrals at some ramified places, where the level of the ramification can be arbitrarily large.
\end{abstract}
\maketitle

\section*{Acknowledgements}

I would like to thank my advisor, Prof. Tonghai Yang. He suggested this problem to me and has given me guidance throughout this paper. I would like to thank Lei Zhang for his advices during his visit to Madison, Michael Woodbury for answering my question in email correspondence, and Robert Harron for his suggestions on the introduction part. 
I would like to thank the organizers of the FRG workshop held at Stanford University in June 2013. I discussed with many people there and got helpful inspirations. I would also like to thank my wife who supports me during this paper. 
This paper is partially supported by Graduate School Grant and NSF grant of Prof. Tonghai Yang.

\section{Introduction}

In this paper we are interested in the cuspidal part of an Eisenstein series restricted to an index 2 subfield. More specifically, 
let $\E$ be a quadratic algebra over a number field $\F$. Let $F$ be a cusp form for $\GL_{2}(\A_{\F})$ corresponding to a cuspidal automorphic representation $\pi$. Let $B$ be the standard Borel subgroup of $\GL_2$ and $Z$ be its center.
Let $\chi_1$ and $\chi_2$ be two Hecke characters on $\E^*\backslash\E_\A^*$, and define $\chi=\frac{\chi_1}{\chi_2}$. Let $\Phi_s$ be a section of the induced representation  $Ind_B^{\GL_2}(\chi_1,\chi_2,s)$. So $\Phi_s$  satisfies
$$\Phi_s(\zxz{a_1}{n}{0}{a_2} g)=\chi_1(a_1)\chi_2(a_2)|\frac{a_1}{a_2}|_{\E_\A}^{s+1/2}\Phi_s(g)$$
for all $\zxz{a_1}{n}{0}{a_2}\in B(\E_\A)$ and $g\in \GL_2(\E_\A)$.
Let $$E(g,s)=\sum\limits_{\gamma \in B(\E)\backslash \GL_{2}(\E)}\Phi_s(\gamma g)$$ 
be the associated Eisenstein series. It is well-known that such Eisenstein series is in the continuous spectrum for $L^2(\GL_2(\E)\backslash \GL_2(\E_\A))$. Its integral against a cusp form on $\GL_2(\E_\A)$ will simply be zero.

But we are interested in the spectral decomposition of $E(g,s)$ when we restrict it to $\GL_2(\A_\F)$. 
In particular, for a cusp form $F$ of a cuspidal automorphic representation $\pi$ on $\GL_2(\A_\F)$, we consider the following integral:
\begin{equation}\label{eq1}
\I(E,F,s)=\int\limits_{Z_{\A}\GL_{2}(\F)\backslash \GL_{2}(\A_{\F})} F(g)E(g,s) dg.
\end{equation}
This integral is not necessarily zero. We would like to see when this integral is automatically zero and otherwise how $\I(E,F,s)$ depends on $s$. 

In addition to its own interest, this automorphic integral is in some sense the complementary case to the well-known Rankin-Selberg integral and triple product integral. 
It's also a special case of an automorphic integral which is actually related to arithmetic height pairing on certain Shimura varieties according to the main theorem in the work of Bruinier, Kudla and Yang in \cite{BKY}. The work in this paper may help us to understand that general integral better.


Let $w_\pi$ denote the central character of $\pi$. If $w_\pi\cdot(\chi_1\chi_2)|_{\A_\F^*}\neq 1$, then $\I(E,F,s)=0$ for a trivial reason. In the following we will assume that 
\begin{equation}\label{centercondition}
w_\pi\cdot(\chi_1\chi_2)|_{\A_\F^*}=1.
\end{equation}
Under this assumption, we will relate $\I(E,F,s)$ to certain $L-$functions and special values of $L-$functions. This is not surprising as we have already seen many examples about relations between automophic integrals and $L-$functions.

\subsection{Automorphic integrals and $L-$functions}

The earliest example for the connection between the integral and $L-$function is the integral representation for the Riemann zeta function. This connection can be used to show, for example, the functional equation and the analytic continuation of the zeta function. 
Tate in his thesis gave the first adelic version of the story (see \cite{Bump} as a reference). Let $\mu$ be a Hecke character on $\A_\F^*$ and $f\in S(\A_\F)$ be a Schwartz function. Tate showed that the integral 
\begin{equation}
\int\limits_{\A_\F^*}f(x)\mu(x)|x|^sd^*x
\end{equation}
represents the $L-$function of the Hecke character $L(\mu,s)$. Usually we say an automorphic integral represents some L-function, if the integral is equal to that L-function up to some constants, easy L-factors and local integrals at ramified places.

His work provided the basic idea to relate the automorphic integrals with the $L-$functions in general: write the automorphic integral as a product of local integrals, then identify the local integral with the corresponding local $L-$factors for unramified places. 
The local integral at ramified places could be different from expectation. It depends, for example, on the choice of the Schwartz functions. Thus the global integral could differ from the $L-$function by factors at the set of ramified places, which is finite.

Another 
example is the so called Rankin-Selberg method (we again refer the readers to \cite{Bump}). Let $F_i$ be cusp forms from automorphic cupidal representations $\pi_i$ for $i=1,2$. Let $E(g,s)$ be the Eisenstein series on $\GL_2(\A_\F)$ (not on $\GL_2(\E_\A)$) associated to two Hecke characters $\chi_1$ and $\chi_2$ of $\A_\F^*$. 
Then the integral
\begin{equation}\label{rankinselbergint}
\int\limits_{Z_{\A}\GL_{2}(\F)\backslash \GL_{2}(\A_{\F})} F_1(g)F_2(g)E(g,s) dg
\end{equation}
represents $$L(\pi_1\times\pi_2,\chi_1,s).$$ 
If we specify $\chi_1$ to be the trivial character, then we get the standard Rankin-Selberg $L-$function $L(\pi_1\times\pi_2,s)$. 
The Rankin-Selberg method can be applied to more general reductive groups. For a survey on this subject, see for example \cite{Bump04}.

Two other developments in this flavor are Waldspurger's formula and the triple product formula. Let $\B$ denote a quaternion algebra. Suppose that  $\E$ is a quadratic algebra over $\F$ embedded in $\B$. $\E^*$ can be identified with a maximal torus of $\B^*$. 
Let  $\pi'$ be an automorphic representation of $\B^*$ and $\hat{\pi'}$  its contragredient representation. Let $\sigma$ be the cuspidal representation of $\GL_2$ such that $\sigma =JL(\pi')$, where $JL$ denotes the Jacquet-Langlands correspondence. 
When $\B\simeq M_2(\F)$, which is our primary case of interest, $\sigma\simeq \pi'$. Denote by  
$w_{\pi'}$ the central character of $\pi'$. Let $\Omega$ be a Hecke character of $\E^*_{\A}$ such that $\Omega|_{\A^*_\F}=w_{\pi'}$. For a cusp form $F_1\in \pi'$, Waldspurger in \cite{Walds} considered the following period integral:
\begin{equation}\label{periodint}
\int\limits_{Z_\A\E^*\backslash \E_\A^*}F_1(t)\Omega^{-1}(t)dt.
\end{equation}
This gives an element in $\Hom_{\E_\A^*}(\pi'\otimes\Omega^{-1},\C)$ which is at most one dimensional. The work of Tunnell (\cite{Tu}) and Saito (\cite{Sa}) gave a criterion for the local component $\Hom_{\E_v}(\pi_v'\otimes\Omega_v^{-1},\C)$ to be nonzero. 
This criterion actually implies that there is a unique local quaternion algebra, either the division algebra or a matrix algebra, such that $\Hom_{\E_v}(\pi_v'\otimes\Omega_v^{-1},\C)$ is nonzero (replacing $\pi_v'$ by $JL(\pi_v')$ if necessary).

When the global quaternion algebra $\B$ avoids such local obstructions, Waldspurger showed that a pairing of such period integrals indeed represents special values of $L-$functions. For a cusp form $F_2\in \hat{\pi'}$, one can similiarly consider 
\begin{equation}
\int\limits_{Z_\A\E^*\backslash \E_\A^*}F_2(t)\Omega(t)dt.
\end{equation}
Then the product of the two period integrals
\begin{equation}
\int\limits_{Z_\A\E^*\backslash \E_\A^*}F_1(t)\Omega^{-1}(t)dt\int\limits_{Z_\A\E^*\backslash \E_\A^*}F_2(t)\Omega(t)dt
\end{equation}
can be related to special values of L-functions
$$\frac{L(\Pi_\sigma\otimes\Omega^{-1},1/2)}{L(\eta,1)}.$$
Here $\Pi_\sigma$ is the base change of $\sigma$ to $\E$ and $\eta$ is the quadratic Hecke character associated to $\E/\F$.

The triple product formula is in some sense similar. Now we have three irreducible unitary cuspidal automorphic representations $\pi_i$ for $\B^*$.  Let $F_i\in\pi_i$ be the cusp forms for $i=1,2,3$. Let $\Pi$ denote $\pi_1\otimes\pi_2\otimes\pi_3$ in this case. Consider the integral
\begin{equation}\label{tripleproductint}
\int\limits_{Z_{\A}\B^*(\F)\backslash \B^*(\A)} F_1(g)F_2(g)F_3(g) dg.
\end{equation}
This integral gives an element of $\Hom_{\B^*(\A)}(\Pi,\C)$, which is at most one dimensional. Prasad in his thesis (\cite{Prasad}) gave a criterion for its local component to be nonzero. Jacquet then conjectured that the central value 
\begin{equation}\label{tripleLcentralvalue}
L(\pi_1\otimes\pi_2\otimes\pi_3,1/2)
\end{equation}
of the triple product L-function does not vanish if and only if there exists a quaternion algebra $\B$ and the corresponding $F_i$'s such that (\ref{tripleproductint}) does not vanish. This conjecture was first proved by Harris and Kudla in \cite{H&K91} \cite{H&K04} using an integral representation of triple product L-function (see \cite{Garrett} \cite{ps}) and the regularized Siegel-Weil formula (see \cite{K&R}). 
Later on, more explicit formulae relating (\ref{tripleproductint}) and (\ref{tripleLcentralvalue}) were given in \cite{GK} \cite{BS} \cite{Watson} for some special cases. Ichino then generalized the above results in \cite{Ichino}, where he considered $\Pi$ as an irreducible unitary cuspidal automorphic representations over an \'{e}tale cubic algebra $\K$ (this in particular includes the case $\Pi=\pi_1\otimes\pi_2\otimes\pi_3$ when $\K$ is just $\F\oplus\F\oplus\F$). 
He showed that a pairing of integral (\ref{tripleproductint}) represents 
$$\frac{L(\Pi,1/2)}{L(\Pi,Ad,1)}.$$

Note the integrals  (\ref{eq1}), (\ref{rankinselbergint}) and (\ref{tripleproductint}) are very similar. Especially, if we take $\E=\F\oplus\F$ for (\ref{eq1}), then the Eisenstein series there is a product of two Eisenstein series over $\F$. 
Take $\B$ just to be a matrix algebra for (\ref{tripleproductint}). Then (\ref{eq1}), (\ref{rankinselbergint}) and (\ref{tripleproductint}) give a complete list of integrals of possible products of three automorphic forms, either cusp form or Eisenstein series, over $Z_{\A}\GL_{2}(\F)\backslash \GL_{2}(\A_{\F})$. 

In general for the triple product formula, we can start with a cusp form defined over an \'{e}tale cubic algebra $\K$, and integrate it over the diagonal $Z_{\A}\GL_{2}(\F)\backslash \GL_{2}(\A_{\F})$.  Similarly for the Rankin-Selberg integral, we can start with a cusp form defined over a quadratic algebra $\E$, restrict it to the base field and integrate it against an Eisenstein series over $\F$. 
When $\E$ is a quadratic field extension, the integral represents Asai L-function(\cite{Kable}). So we have the following table:
\newline

\begin{tabular}{|p{4cm}|p{4cm}|c|}
\hline
Degree of the algebra that the cusp form is defined over	&Degree of the algebra that the Eisenstein series is defined over	&L-functions represented\\ \hline
3	&No Eisenstein series	&Triple product L-function\\ \hline
2	&1	&Rankin-Selberg L-function or Asai L-function\\ \hline
1	&2	&?\\
\hline

\end{tabular}
\newline
\newline

Note that we need at least one cusp form to guarantee convergence. So our work on (\ref{eq1}) is a complementary to the Rankin-Selberg integral and the triple product integral.

Despite their similarity, we won't follow, for example, Ichino's method directly, as cusp forms and Eisenstein series are somewhat different in nature.  Also a pairing of our global integrals may involve another variable $s$, which is not so nice to deal with. 
We shall see in this paper that we actually apply Waldspurger's work to avoid such potential problems.

\subsection{Main results and organization}

The first goal of this paper is to prove the following formula (see Theorem \ref{maintheo})

\begin{equation}\label{introeq1}
C\cdot \I(E,F,s)=L(\pi\otimes\chi_{1}|_{\A_\F^*},2s+1/2)\frac{L(\Pi\otimes\Omega,1/2)}{ L(\eta,1)L^{\E}(\chi,2s+1)}\prod\limits_{v\in S}\P_v^0(s,1/2,f_v,\Phi_{s,v}).
\end{equation}
Here $C$ is a fixed period integral of Waldspurger type, $\chi_{1}|_{\A_\F^*}$ is the restriction of $\chi_1$ to $\A_\F^*$ and $\Pi$ is the base change of $\pi$ to $\E$. $L^{\E}(\chi,2s+1)$ is an $L-$function over $\E$, since $\chi=\frac{\chi_1}{\chi_2}$ is a Hecke character defined over $\E$. $S$ is the set of ramified places (which is finite). For $v$ a ramified place, $\P_v^0(s,1/2,f_v,\Phi_{s,v})$ is a normalized local integral defined by:
\begin{align}\label{introP0}
\P_v^0(s,1/2,f_v,\Phi_{s,v})=&\frac{L_v(\eta_v,1)L^{\E}_v(\chi_v,2s+1)
}{L_v(\Pi_v\otimes\Omega_v,1/2)L_v(\pi_v\otimes\chi_{1,v}|_{\F_v^*},2s+1/2)}\\
&\cdot\int\limits_{ZN\backslash \GL_{2}(\F_v)}\int\limits_{\GL_{2}(\F_v)}W_{\varphi,v}^-(\sigma)\Delta_v(\sigma)^{w-1/2}r'(\sigma)f_v(g,det(g)^{-1})\Phi_{s,v}(\gamma_0 g)d\sigma dg.\notag
\end{align}
We will see in Propositions \ref{propinert} and \ref{propsplit} that $\P_v^0(s,1/2,f_v,\Phi_{s,v})=1$ for all unramified places.

The most important part of (\ref{introeq1}) is  $$L(\pi\otimes\chi_{1}|_{\A_\F^*},2s+1/2).$$ 
The part $$\frac{L(\Pi\otimes\Omega,1/2)}{ L(\eta,1)}$$ 
is as expected from Waldspurger's formula, which is not surprising. 

The second goal is to work out the local integrals $\P_v^0$ for the ramified places. While we assume we know enough about Waldspurger's formula, the local calculations done here is more general than what have been done for Waldspurger's formula. In particular the methods used here should also be applicable to the local computations of Waldspurger's formula. I hope the techniques and methods used here can also be helpful to other kinds of local integrals, for example the triple product formula.

The arrangement of this paper is as follows: we will briefly review in Section 2 the Weil representations (following \cite{Walds}) and the Shimizu lifting (see \cite{Shimi}) as a special case of the theta correspondence. We will also discuss some special Schwartz functions of the Weil representation in Subsection 2.2, part of which may be new results. 

In Section 3, we discuss more about the period integral as in Waldspurger's formula. Then we review Tunnell-Saito's theorem (\cite{Tu}\cite{Sa}), which provides a criterion about whether the local space $\Hom_{\E_v}(\pi_v'\otimes\Omega_v^{-1},\C)$ is zero. Gross and Prasad's work (\cite{Gr}\cite{GP}) gives a local test vector when it's not zero. Then Waldspurger's work (\cite{Walds}) provides a method to study the pairing of period integrals globally. We will state it in terms of the Shimizu lifting. Another formulation of  Waldspurger's formula is given via matrix coefficients, which we will also make use of.

By analyzing $\GL_2(\F)$ orbits of $B(\E)\backslash \GL_2(\E)$ in Section 4, we will see that 
\begin{equation}\label{introeq2}
\I(E,F,s)=\int\limits_{\E^{*}_{\A}\backslash \GL_{2}(\A)} \Phi_s(\gamma_0 g)\int\limits_{Z_{\A}\E^{*}\backslash \E^{*}_{\A}}F(tg)\Omega(t)dtdg
\end{equation}
This is actullay a weighted integral of Waldspurger's period integral with $\B$ being a matrix algebra. As a corollary of this (see Corollary \ref{Corofautomatic0}), 
$$\left\lbrace \begin{array}{c}
   \Hom_{\E_\A^*}(\pi\otimes\Omega,\C)=0\\
    \text{or\ \ } L(\Pi\otimes\Omega,1/2)=0
  \end{array}\right\rbrace
\Longrightarrow
\I(E,F,s)=0.
$$
When this doesn't happen, we can pair (\ref{introeq2}) with a fixed period integral and apply Waldspurger's formula, rewriting our main integral (\ref{eq1}) to be the form as in Proposition \ref{mainprop}:
\begin{equation}\label{introlocalproduct}
C\cdot\I(E,F,s)=\int\limits_{Z_{\A}N_{\A}\backslash \GL_{2}(\A)}\int\limits_{\GL_{2}(\A)} W_\varphi^-(\sigma)\Delta(\sigma)^{w-1/2}r'(\sigma)f(g,\det (g)^{-1})\Phi_s(\gamma_0 g)dgd\sigma|_{w=1/2}.
\end{equation}
This can be directly written as a product of local integrals as appeared in (\ref{introP0}). We will also formulate the local integrals by matrix coefficients:
\begin{equation}\label{intromatrixcoeff}
 C\cdot\I(E,F,s)=C_0\prod\limits_{v}\int\limits_{\F_v^*\backslash \GL_2(\F_v)}\Phi_{s,v}(\gamma_0 g)< F_{1,v},\pi_v(g) F_{v}>dg.
\end{equation}

We will do most local calculations in terms of the Shimizu lifting, that is, by (\ref{introlocalproduct}). In Section 5, we will compute the local integrals for unramified places, i.e. when locally $\pi_v$ is unramified, $\E_v/\F_v$ is either inert or split, and $\Phi_s$ is unramified (which in turn means $\chi_{1,v}$ and $\chi_{2,v}$ are unramified). Then Proposition \ref{propinert} and Proposition \ref{propsplit} will suggest the L-factors as appeared in (\ref{introeq1}) and (\ref{introP0}).

In Section 6, we will do local computations for other non-archimedean places. A table of expected local L-factors is listed at the beginning of Section 6.
We assume, for most cases, that $\pi_v$, $\Phi_{s,v}$ and $\E_v/\F_v$ have disjoint ramifications. But the level of the ramification of $\pi_v$ and $\Phi_{s,v}$ can be arbitrary. 
We will also do a case when $\pi_v$ and $\Phi_{s,v}$ have joint ramifications. For all cases, the denominator of the local integral is just as expected. The results are listed in the following table:
\newline

\begin{tabular}{|c|p{2.1cm}|c|c|c|c|}
\hline
Case	&$\pi_v$	&$\chi_{1,v}$ and $\chi_{2,v}$	&$\E_v/\F_v$	
&$\P_v^0(s,1/2,f,\Phi_s)$\\ \hline
1	&unramified	&unramified	&ramified	&$1$	\\ \hline
2	&unramified special 	&unramified	&split	&$\frac{1}{(q+1)^2(1-\chi_v^{(2)}q^{-(2s+1)})}$\\ \hline
3	&supercupidal or ramified principal	&unramified	&split	&$\frac{1}{(q+1)^2q^{2c-2}(1-\chi_v^{(2)}q^{-(2s+1)})}$\\ \hline
4	&unramified	&$\chi_{1,v}$ level c	&inert	&$\frac{\P_0}
{1+q^{-1}}$ for $\P_0$ given in (\ref{ramPhiP0})\\ \hline
5	&$\mu_{2,v}$ level c	&$\chi_{1,v}$ level c	&inert	& $\frac{1}{(q-1)^3(q+1)^2q^{4c-5}\chi_{1,v}(\sqrt{D})}$\\ \hline
\end{tabular}
\newline

Here for cases 2 and 3, there are two places of $\E$ above the place $v$, and $\chi_v^{(2)}=\chi_v^{(2)}(\varpi)$ means the component of $\chi$ at the second place above $v$.

One reason to do the calculations in terms of the Shimizu lifting is that we don't need to explicitly give the Gross-Prasad test vector in the representation. The properties of the test vectors can be translated into the properties of the Schwartz functions $f_v$ in (\ref{introP0}). These Schwartz functions can be chosen just to be those given in Subsection 2.2. On the other hand, the Whittaker functions $W_{\varphi,v}^-$ that appeared in (\ref{introP0}) will be chosen to be the newforms, which makes calculations easier.

We will discuss the approach by matrix coefficients for the local computations in Section 7. In particular we will compute the local integral in (\ref{intromatrixcoeff}) for the case when $\pi_v$ is supercuspidal, $\Phi_{s,v}$ is unramified and $\E_v/\F_v$ is inert. 
Tunnell-Saito's theorem implies that the integral should be zero if the level of $\pi_v$ is odd; when the level is even the calculation turns out to be surprisingly easier than the approach by the Shimizu Lifting. We expect similar results when $\pi_v$ is a ramified principal series representation.

In Appendix A we will discuss how to integrate over $\GL_2$ by using right invariance under specific compact subgroups, though this is not the only method we will be using. In Appendix B we shall derive Proposition \ref{powerofmonomials} for the Kirillov model of a supercuspidal representation.  Simply put, it describes the change of the support of an element in the Kirillov model under the group actions purely by its level. 
According to \cite{Yo77}, this result should be equivalent to the following (See Corollary \ref{cortoPropB3}):
\begin{cor*}
 Suppose a supercuspidal representation $\pi$ is of level $c$ with central character being unramified or level 1. Let $\lambda$ be a character of $\F_v^*$ of level $i$. Then the level $c(\pi\otimes\lambda)$ is $\text{max}\{c,2i\}$.
\end{cor*}
This result seems elementary and may have been proven somewhere else. So far I've only seen Lemma 1 of \cite{Yo77} which claimed that $c(\pi\otimes\lambda)=2i$ if $i>c/2$ but $c(\pi\otimes\lambda)\leq c$ if $i\leq c/2$.

This proposition is used in the local calculations in Section 6 and Section 7. We expect it also to be helpful in the local calculation of the triple product formula.

\section{the Weil Representation and the Shimizu Lifting}
\subsection{The Weil Representation}
The Weil representation can be defined for more general reductive group pairs, but we will follow \cite{Walds} with the following setting:

Fix $\psi$ an additive character of $\F$. Let $G=\GL_2(\F)$, and let $\B$ be a quaternion algebra over $\F$ with reduced norm $Q$. We will be particularly interested in $\B$ being matrix algebra $M_2(\F)$. Denote by $GO(\B)$ the orthogonal similitude group for $\B$. Let $\B^*\times \B^*$ acts on $\B$ via $(h_1,h_2)\cdot b=h_1b\ h_2^{-1}$. This actually give us a short exact sequence
\begin{equation}\label{quatexseq}
 1\rightarrow\F^*\rightarrow (\B^*\times \B^*)\rtimes\{1,\iota\} \rightarrow GO(\B)\rightarrow 1.
\end{equation}
Here $\iota: x\mapsto \bar{x}$ is the main involution on $\B$. On $\B^*\times \B^*$ it acts by $(h_1,h_2)\mapsto (\bar{h}_2^{-1},\bar{h}_1^{-1})$. $\F^*$ is embedded into the group in the middle by $x\mapsto (x,x)\rtimes 1$. We will simply write $(h_1,h_2)$ for $(h_1,h_2)\rtimes 1$ when considered as an element of $GO(\B)$.

\begin{defn}\label{Weildef}
The Weil representation for the similitude group pair $G\times GO(\B)$ on the space of Schwartz functions $S(\B\times \F^*)$ is defined as follows: for $f(x,u)\in S(\B\times \F^*)$, $\alpha,\delta\in\F^*$, $\beta\in\F$, $h\in GO(\B)$,
\begin{enumerate}
 \item[(i)]$r'(\zxz{1}{\beta}{0}{1})f(x,u)=\psi_u(\beta q(x))f(x,u)$,
 \item[(ii)]$r'(\zxz{0}{1}{-1}{0})f(x,u)=\gamma[\psi_u,q]\int\limits_{\B} f(y,u)\psi_u((x,y))dy$,
 \item[(iii)]$r'(\zxz{\alpha}{0}{0}{\alpha^{-1}})f(x,u)=|\alpha|^2f(\alpha x,u)$,
 \item[(iv)]$r'(\zxz{1}{0}{0}{\delta})f(x,u)=|\delta|^{-1}f(x,\delta^{-1}u)$,
 \item[(v)]$r''(h)f(x,u)=f(h^{-1}\cdot x,u\nu(h))$.
\end{enumerate}


Here $\gamma[\psi_u,q]$ equal to $1$ if $\B$ is the matrix algebra and $-1$ is $\B$ is a division algebra. $\psi_u(x)=\psi(ux)$. Furthermore $\nu$ is the similitude character for the group $GO(\B)$, and $(x,y)=Q(x+y)-Q(x)-Q(y)$ in (ii). 
\end{defn}
\begin{rem}\label{quatremark}
For $(h_1,h_2)\in GO(\B)$, 
we have  $\nu(h_1,h_2)=Q(h_1)Q(h_2)^{-1}$, and 
\begin{equation}\label{quatcal}
 r''(h_1,h_2)f(x,u)=f(h_1^{-1}xh_2,uQ(h_1)Q(h_2)^{-1}).
\end{equation}
Also by combining (iii) and (iv), we can get
\begin{equation}\label{Weilsilim}
r'(\zxz{\alpha}{0}{0}{1})f(x,u)=|\alpha|f(\alpha x,\alpha^{-1}u).
\end{equation}
We will use these simple facts later. 
\end{rem}
\subsection{Special elements in the Weil Representation}\label{sectionWeil}

In this subsection let $\F_v$ be the local field  at a finite place $v$ of $\F$ and  $\B_v=M_2(\F_v)$.  We will discuss about properties of some special elements in the Weil representation $S(M_2(\F_v)\times \F_v^*)$. Part of the results here may be new as I haven't seen it in literature.

First we clarify what we mean by the right action of $g\in \GL_2$ on $S(M_2(\F_v)\times \F_v^*)$. It follows from the action of $(1,g)\in GO(\B)$ as in (v) of Definition \ref{Weildef}, so a Schwartz function $f(x,u)$ is mapped to $f(xg,\frac{u}{\det g})$. We will say a Schwartz function is invariant under the right action (or just right-invariant) of some compact subgroups of $\GL_2$ according to this sense.
Let $x=\zxz{x_1}{x_2}{x_3}{x_4}, y=\zxz{y_1}{y_2}{y_3}{y_4}\in M_2(\F_v)$. By definition, $$Q(x)=\det x=x_1x_4-x_2x_3,$$ and $$(x,y)=x_1y_4+x_4y_1-x_2y_3-x_3y_2.$$

Throughout this paper, let $O_F$ be the ring of integers of the local field $\F_v$ and $O_F^*$ be its group of units. Fix a uniformizer $\varpi$ for $\F_v$. Denote by $v(x)$ the valuation of $x$. Denote by $\omega$ the matrix $\zxz{0}{1}{-1}{0}$. Denote by $K$ the compact subgroup $\GL_2(O_F)$. Denote by $K_1(\varpi^c)$ the subgroup of $K$ whose elements are congruent to $\zxz{*}{*}{0}{1} \mod{(\varpi^c)}$ for an integer $c>0$. Similarly denote by $K_0(\varpi^c)$ for $\zxz{*}{*}{0}{*}\mod{(\varpi^c)}$ and $K_1^1(\varpi^c)$ for $\zxz{1}{*}{0}{1}\mod{(\varpi^c)}$. Let $q=|\varpi|^{-1}$. Assume that the local additive character $\psi=\psi_v$ is unramified.

\begin{lem}
Let $f=char(\zxz{O_F}{O_F}{O_F}{O_F})(x)\times char(O_F^*)(u)\in S(M_2(\F_v)\times \F_v^*)$. It is invariant by $K$ under both the right action and the Weil representation $r'$ defined above.
\end{lem}
\begin{proof}
The invariance under the right action is obvious. For the Weil representation, it's clear from definition (i)(iii) that $f$ is invariant under the action of the Borel part $B(O_F)$. One can easily see that
\begin{align*}
\int\limits_{M_2(\F_v)}f(y,u)\psi_u(q(x,y))dy&=\int\limits_{M_2(O_F)}\psi(u(x_1y_4+x_4y_1-x_2y_3-x_3y_2))dy\\
&=char(M_2(O_F))(x)\times char(O_F^*)(u).
\end{align*}
The last equality is true because in general we have
\begin{equation}\label{Weilveccal1}
\int\limits_{\varpi^k O_F}\psi(x_iy_j)dy_j=q^{-k}char(\varpi^{-k}O_F)(x_i)
\end{equation}
for any integer $k$. Then one can get the conclusion because $K$ is generated by $B(O_F)$ and $\omega$.

\end{proof}

\begin{lem}\label{levelweilvec}
Let $f=char(\zxz{O_F}{O_F}{\varpi^cO_F}{O_F})\times char(O_F^*)$, for integer $c>0$.
\begin{enumerate}
\item[(i)] It is invariant by $K_1(\varpi^c)$ under both the right action and the Weil representation.
\item[(ii)] For $n\in \F_v^*$, $v(n)=j$ for $0\leq j\leq c$, 
\begin{equation}
r'(\zxz{1}{0}{n}{1})f(x,u)=q^{j-c}char(\zxz{O_F}{\varpi^{j-c}O_F}{\varpi^{j}O_F}{O_F})\psi(- ux_2x_3n^{-1})\times char(O_F^*).
\end{equation}
This function is still  right $K_1(\varpi^c)-$invariant.
\end{enumerate}
\end{lem}
\begin{proof}
The invariance under the right action is easy to check. So is the Weil representation of $B(O_F)$. We claim that $K_1(\varpi^c)$ is generated by $B(O_F)$ and $\zxz{1}{0}{n}{1}$ where $n\equiv 0$ mod ($\varpi^c$). Indeed for $x_3\equiv 0 \mod{(\varpi^c)}$, $x_1,x_4\in O_F^*$ , $x_2 \in O_F$, we have
\begin{equation}
\zxz{x_1}{x_2}{x_3}{x_4}=\zxz{\frac{x_1x_4-x_2x_3}{x_4}}{x_2}{0}{x_4}\zxz{1}{0}{\frac{x_3}{x_4}}{1},
\end{equation}
where $\frac{x_1x_4-x_2x_3}{x_4}$ is still an unit and $\frac{x_3}{x_4}\equiv 0 \mod{(\varpi^c)}$. So it remains to check that $f$ is invariant by all such $\zxz{1}{0}{n}{1}$.
\newline
Note that $\zxz{1}{0}{n}{1}=-\omega\zxz{1}{-n}{0}{1}\omega$. By formula (\ref{Weilveccal1}) we know 
\begin{equation*}
r'(\omega)f(x,u)=q^{-c}char(\zxz{O_F}{\varpi^{-c}O_F}{O_F}{O_F})\times char(O_F^*).
\end{equation*} 
This however is invariant under $\zxz{1}{-n}{0}{1}$, since by definition it will give a factor $\psi(-un\det y)$ which is trivial for $y\in \zxz{O_F}{\varpi^{-c}O_F}{O_F}{O_F}$ and $n\in \varpi^cO_F$. Then the action of $-\omega$ will change it back.
\newline 
For part (ii), we also use $\zxz{1}{0}{n}{1}=-\omega\zxz{1}{-n}{0}{1}\omega$. From the calculation above we get
\begin{equation*}
 r'(\zxz{1}{-n}{0}{1}\omega)f(y,u)=q^{-c}char(\zxz{O_F}{\varpi^{-c}O_F}{O_F}{O_F})\psi(- un\det y)\times char(O_F^*).
\end{equation*}
Note $\psi(- un\det y)=\psi(uny_2y_3)$ for $y\in \zxz{O_F}{\varpi^{-c}O_F}{O_F}{O_F}$. For another action of $\omega$, we consider the integral in $y_1y_4$ and the integral in $y_2y_3$ seperately. The integral in $y_1y_4$ is very easy, as in the previous lemma. Now we focus on the following integral:
\begin{equation*}
 q^{-c}\int\limits_{y_2\in \varpi^{-c}O_F}\int\limits_{y_3\in O_F}\psi(uy_3(ny_2-x_2))\psi(-ux_3y_2)dy_3dy_2.
\end{equation*}
Assume that $x_2$ is fixed. For the integral in $y_3$ to be non-zero, we need $y_2$ to satisfy $ny_2-x_2\in O_F$, which is equivalent to say $y_2\in n^{-1} x_2+\varpi^{-j}O_F$ as $v(n)=j$. Then the integral becomes 
\begin{equation*}
 q^{-c}\int\limits_{y_2\in \varpi^{-c}O_F\cap n^{-1} x_2+\varpi^{-j}O_F}\psi(-ux_3y_2)dy_2.
\end{equation*}
Note $\varpi^{-j}O_F\subseteq \varpi^{-c}O_F$. The domain of the integral is not empty iff $x_2\in \varpi^{j-c}O_F$. In that case, the integral becomes
\begin{equation*}
 q^{-c}\int\limits_{y_2\in n^{-1} x_2+\varpi^{-j}O_F}\psi(-ux_3y_2)dy_2=q^{j-c}\psi(- ux_2x_3n^{-1}) \text{\ if \ }x_3\in \varpi^j O_F.
\end{equation*}
So we get $r'(\omega\zxz{1}{-n}{0}{1}\omega)=q^{j-c}char(\zxz{O_F}{\varpi^{j-c}O_F}{\varpi^{j}O_F}{O_F})\psi(- ux_2x_3n^{-1})\times char(O_F^*)$. Then just note that the action of $-1$ will not change this function.
\end{proof}
%

\begin{rem}
Roughly speaking, the invariance under the Weil representation $r'$ depends on the configuration on the diagonals. The invariance under the right action depends on configurations on the rows.
Actually one can construct in this way Schwartz functions of any prescribed levels of invariance. Suppose that we need a function which is invariant by the Weil representation under $K_1(\varpi^{c_1})$, and invariant by the right action under $K_1(\varpi^{c_2})$.
Then we can consider, for example, the following functions(far from unique):
\begin{equation}
f=char(\zxz{\varpi^{-j}O_F}{\varpi^{-j-c_2}O_F}{\varpi^{j+c_1+c_2}O_F}{\varpi^{j+c_1}O_F})\times char(O_F^*), j\in \Z .
\end{equation} 
\end{rem}

Now we discuss about a slightly different type of Schwartz functions.
\begin{lem}\label{lemofexoticSchwartz}
Let $b_1,b_2\in O_F$ and $c$ is an integer. Define  $f=char(\zxz{b_1+\varpi^cO_F}{O_F}{b_2+\varpi^cO_F}{O_F})\times char(O_F^*)$. 
\begin{enumerate}
\item[(i)]$f$ is $K_1^1(\varpi^c)-$invariant under the Weil representation $r'$ and right action. 
\item[(ii)]For $n\in \F_v^*$, $0\leq v(n)=j\leq c$,
\begin{equation}
r'(\zxz{1}{0}{n}{1})f=q^{2(j-c)}char(\zxz{b_1+\varpi^jO_F}{\varpi^{j-c}O_F}{b_2+\varpi^jO_F}{\varpi^{j-c}O_F})\psi(un^{-1}[(x_1-b_1)x_4-x_2(x_3-b_2)])char(O_F^*).
\end{equation}
This function is still right $K_1^1(\varpi^c)-$ invariant.
\end{enumerate}

\end{lem}
\begin{proof}
We will focus on the computations. The rest are easy to check. First of all,
\begin{equation}
r'(\omega)f(y,u)=q^{-2c}\psi(u(y_4b_1-y_2b_2))char(\zxz{O_F}{\varpi^{-c}O_F}{O_F}{\varpi^{-c}O_F})(y)char(O_F^*)(u).
\end{equation}
From this it's clear that $f$ is invariant under $\zxz{1}{0}{n}{1}$ for $n\in \varpi^c O_F$. For those $n$ with $0\leq v(n)=j< c$, the action of $\zxz{1}{-n}{0}{1}$ will give a factor $\psi(-un(y_1y_4-y_2y_3))$. For another action of $\omega$, we will only do the integral in $y_1y_4$. The integral in $y_2y_3$ is very similar. So we want to do the following integral:
\begin{equation}
q^{-c}\int\limits_{y_4\in \varpi^{-c}O_F}\int\limits_{y_1\in O_F}\psi(uy_4b_1-uny_1y_4+uy_1x_4+uy_4x_1)dy_1dy_4.
\end{equation}
The way to do this integral is similar to what we did in the previous lemma. For the integral in $y_1$ to be non-zero, we need $x_4-ny_4\in O_F$, that is, $y_4\in n^{-1}x_4+\varpi^{-j}O_F$. For the domain of the integral in $y_4$ to be non-empty, we need $x_4\in \varpi^{j-c}O_F$. Then the domain for $y_4$ is $n^{-1}x_4+\varpi^{-j}O_F\subset \varpi^{-c}O_F$, and the integral in $y_4$ is 
\begin{equation}
q^{-c}\int\limits_{y_4\in n^{-1}x_4+\varpi^{-j}O_F}\psi(uy_4b_1+uy_4x_1)dy_4=q^{j-c}\psi(u(x_1+b_1)n^{-1}x_4)
\end{equation}
when $x_1\in -b_1+\varpi^jO_F$. Note that the final action of $-1$ should change $x_i$ to $-x_i$ and also their corresponding domains. Then the statement in the lemma is clear.
\end{proof}
\begin{rem}
For $f=char(\zxz{b_1+\varpi^cO_F}{O_F}{b_2+\varpi^cO_F}{O_F})\times char(\beta+\varpi^c O_F)$ with $\beta\in (O_F/\varpi^cO_F)^*$, one can have a similar lemma.
\end{rem}

\subsection{the Shimizu Lifting}
Now we review briefly Shimizu's lifting (refer to \cite{Shimi} for more details). For the dual group pair $\GL_2\times GO(\B)$, we can use the theta lifting to give an automorphic representation of $GO(\B)$ corresponding to a given automorphic representation $\sigma$ of $\GL_2$.
One can lift this representation further by the exact sequence (\ref{quatexseq}) to an automorphic representation $\Theta(\sigma)$ for $\B^*\times \B^*$. 
\newline
In particular, let $f\in S(\B_\A\times \A_F^*)$, $h_1,h_2\in \B^*_{\A}$, $\varphi\in\sigma$, $g\in \GL_2(\A_{\F})$, the theta kernel is
\begin{equation}
 \theta(f,g,h_1,h_2)=\sum\limits_{x\in \B(\F),u\in \F^*}r'(g)r''(h_1,h_2)f(x,u).
\end{equation}
The global theta lifting is
\begin{equation}
 \theta(f,\varphi,h_1,h_2)=\int\limits_{\GL_2(\F)\backslash \GL_2(\A_{\F})}\varphi(g)\theta(f,g,h_1,h_2)dg.
\end{equation}
Then $\Theta(\sigma)$ is just the collection of all such $\theta(f,\varphi,g_1,g_2)$ for all possible $\varphi\in\sigma$ and $f\in S(\B_\A\times \A_\F^*)$.
\begin{theo}\label{shimizu}
 (\textbf{Shimizu's lifting}) Let $\sigma$ be an automorphic representation of $\GL_2$.
\begin{enumerate}
 \item [(i)]If $\sigma$ doesn't appear in the image of Jacquet-Langlands correspondence, then $\Theta(\sigma)=0$.
 \item[(ii)]If $\sigma$ appears in the image of Jacquet-Langlands correspondence, let $JL(\pi ')=\sigma$. Then $\Theta(\sigma)=\pi '\otimes \hat{\pi }'$.
\end{enumerate}
\end{theo}
\begin{rem}
In particular this theorem applies to the case when $\B$ is the matrix algebra. In this case, $\B^*\simeq \GL_2$ and $\sigma\simeq\pi'$.
\end{rem}

\section{Period integral, test vector and Waldspurger's formula}
\subsection{Period integral}

Recall in the introduction we mentioned the following period integral studied by Waldspurger
\begin{equation}
\int\limits_{Z_\A\E^*\backslash \E_\A^*}F_1(t)\Omega^{-1}(t)dt.
\end{equation}
Here $\E$ is a quadratic algebra over $\F$ embedded in $\B$. $F_1$ is an element of $\pi'$ which is an automorphic representation of $\B^*$ with
 the central characters $w_{\pi'}$. $\Omega$ is a Hecke character of $\E^*_{\A}$ such that $\Omega|_{\A^*_\F}=w_{\pi'}$.

This period integral actually defines an element in $\Hom_{\E_\A^*}(\pi'\otimes\Omega^{-1},\C)$. But it's not necessary that this space is non-zero. 

Now we discuss the local obstruction for this integral to be nonzero.
We put a subscript $v$ for any notation to mean its local component at place $v$. The Hasse invariant $\epsilon(\B_v)$ of a local quaternion algebra $\B_v$ is defined to be $1$ if $\B_v\simeq M_2(\F_v)$, and $-1$ if it's a division algebra. Let $\eta_v:\F_v^*\rightarrow \C^*$ be the local component of the quadratic Hecke character $\eta$ associated to the quadratic extension $\E/\F$. Consider the local root number 
$\epsilon(\frac{1}{2},\Pi_{\sigma,v}\otimes\Omega_v^{-1})$
where $\Pi_{\sigma,v}$ is the base change of $\sigma_v$ to $\E_v$. In general the local root number would depend on the local component of a chosen additive character $\psi$. In this case by $\Omega|_{\A_\F^*}=w_{\pi'}$, this local root number is independent of $\psi_v$ and only takes values $\pm 1$. See \cite{Tu}.

The following theorem is due to Tunnell and Saito (\cite{Tu} \cite{Sa}).
\begin{theo}\label{Tunnell}
The space $\Hom_{\E_v^*}(\pi'_v\otimes \Omega_v^{-1},\C)$ is at most one-dimensional. It is nonzero if and only if 
\begin{equation}
\epsilon(\frac{1}{2},\Pi_{\sigma,v}\otimes\Omega_v^{-1})=\Omega_v^{-1}(-1)\eta_v(-1)\epsilon(\B_v).
\end{equation}
\end{theo}
\begin{example}\label{ApplyTunnell}
Suppose that  $\Omega_v$ is unramified and $\B_v\simeq M_2(\F_v)$. So $\sigma_v=JL(\pi'_v)\simeq \pi'_v$. Let $n(\sigma_v)$ denote the level of $\sigma_v$. If $\E_v$ is split over $\F_v$, then $\epsilon(\frac{1}{2},\Pi_{\sigma,v}\otimes\Omega_v^{-1})$ is always 1, and $\Hom_{\E_v^*}(\pi'_v\otimes \Omega_v^{-1},\C)$ is non-zero. If $\E_v$ is inert over $\F_v$, $\epsilon(\frac{1}{2},\Pi_{\sigma,v}\otimes\Omega_v^{-1})=1$ if and only if $n(\sigma_v)$ is even. As a result,
$\Hom_{\E_v^*}(\pi'_v\otimes \Omega_v^{-1},\C)$ is non-zero if and only if $n(\sigma_v)$ is even. (See \cite{Gr} Proposition 6.3.)
\end{example}

\subsection{Gross and Prasad's test vector}
When $\Hom_{\E_v^*}(\pi'_v\otimes \Omega_v^{-1},\C)$ is non-zero for a local non-archimedean place, pick a non-zero element $l$ of it. Gross and Prasad in \cite{GP} gave a choice of test vector $F_{1,v}\in \pi'_v$ such that $l(F_{1,v})\neq 0$, under the hypothesis that either $\pi_v'$ or $\Omega_v$ is unramified. This hypothesis implies that the central character is always unramified.

We first assume that $\Omega_v$ is unramified. Let $O_F$ be the ring of integers in $\F_v$ and $\varpi$ a fixed uniformizer of it.  On $\B_v$ we have a Trace map defined to be $Tr(\alpha)=\alpha+\bar{\alpha}$, where $\bar{\alpha}$ is the image of $\alpha$ under the non-trivial involution on $\B_v$. An order $R$ of $\B_v$ is defined to be a subring of $\B_v$ containing $O_F$ which is a free $O_F-$module of rank 4 (equivalently, $R\otimes_{O_F}\F_v=\B_v$). Its dual is defined to be
\begin{equation*}
R^\perp=\{\beta\in \B_v|Tr(\alpha\beta)\in O_F \text{\ for all\ }\alpha \in R\}.
\end{equation*}
Let $q=|\varpi|^{-1}$ as before. Define the reduced discriminant $d(R)$ of $R$ to be the integer such that $\sharp(R^\perp/ R)=q^{2d(R)}$. See \cite{Gr} for more details.

Let $R_c$ be an order of reduced discriminant $c=n(\pi')$ which contains $O_E$ under the embedding $\E_v\hookrightarrow\B_v$. It is unique up to conjugacy by $\E_v^*$. Let $R_c^*$ denote its units.

\begin{prop}\label{Grosstest}
Assume $\Omega_v$ is unramified. If $n(\pi')\geq 2$, further assume $\E_v/\F_v$ is unramified.

When $\Hom_{\E_v^*}(\pi'_v\otimes \Omega_v^{-1},\C)\neq 0$, pick $l$ to be a non-trivial element of it. Let $F_{1,v}\in \pi'$ be the unique (up to constant) element fixed by $R_c^*$. Then $l(F_{1,v})\neq 0$.
\end{prop}
\begin{rem}
Both Theorem \ref{Tunnell} and Proposition \ref{Grosstest} have statements on the other side of the Jacquet-Langlands correspondence when $\epsilon(\frac{1}{2},\Pi_{\sigma,v}\otimes\Omega_v^{-1})=-\Omega_v^{-1}(-1)\eta_v(-1)\epsilon(\B_v)$.
But we won't record them here as we don't need them.
\end{rem}
\begin{example}\label{GrTest1}
Suppose  $\B_v\simeq M_2(\F_v)$. Suppose that $\E_v/ \F_v$ is inert and can be written as $\F_v(\sqrt{D})$. Embed $\E_v$ into $\GL_2(\F_v)$ via $a+b\sqrt{D}\mapsto \zxz{a}{b}{bD}{a}$. By Example \ref{ApplyTunnell}, $\pi'_v\simeq \sigma_v$ should be of even level $c=2k$. Then we can choose
\begin{equation}
R_c=\{\zxz{a+\varpi^k O_F}{b+\varpi^k O_F}{bD+\varpi^k O_F}{a+\varpi^k O_F}|a+b\sqrt{D}\in O_E \}.
\end{equation}
\end{example}
\begin{example}\label{Grtest2}
When $\E_v/\F_v$ is split, $\B_v$ must be the matrix algebra and $\pi'_v\simeq \sigma_v$. Suppose 2 is a unit for the local field. If $\E_v$ is embedded into $\GL_2(\F_v)$ diagonally, we can pick $R_c=\{\zxz{O_F}{O_F}{\varpi^c O_F}{O_F}\}$. The fixed element $F_{1,v}$ by $R_c^*$ is then the standard new form studied by Casselman. But in this paper we always consider the global embedding being
\begin{equation}
\E\hookrightarrow \GL_2(\F): a+b\sqrt{D}\mapsto \zxz{a}{b}{bD}{a}.
\end{equation}
For a split place, fix an element $\sqrt{D}\in \F_v$ such that $\sqrt{D}^2=D$. One can easily check that
\begin{equation}
\zxz{1}{-\frac{1}{\sqrt{D}}}{\sqrt{D}}{1}^{-1}\zxz{a}{b}{bD}{a}\zxz{1}{\frac{1}{-\sqrt{D}}}{\sqrt{D}}{1}=\zxz{a+b{\sqrt{D}}}{0}{0}{a-b\sqrt{D}}.
\end{equation}
So we can pick 
\begin{equation}
R_c=\zxz{1}{-\frac{1}{\sqrt{D}}}{\sqrt{D}}{1}\zxz{O_F}{O_F}{\varpi^c O_F}{O_F}\zxz{1}{-\frac{1}{\sqrt{D}}}{\sqrt{D}}{1}^{-1},
\end{equation}
and the element fixed by $R_c^*$ is just the image of the new form under the action of $\pi'(\zxz{1}{-\frac{1}{\sqrt{D}}}{\sqrt{D}}{1})$.
\end{example}

Now we assume that $\pi'_v$ is unramified and $\Omega_v$ ramified of level $c$. This already implies that $\B_v$ splits and $\pi'_v$ is an unramified principal series. Let $O_c=O_F+\varpi^c O_E$. Let $R$ be a maximal order in $M_2(\F_v)$ which optimally contains the order $O_c$. This just means $R$ is maximal and $R\cap \E_v=O_c$. Such maximal order is unique up to conjugacy by $\E_v^*$. Similarly we have the following result:
\begin{prop}
Assume that  $\pi'_v$ is unramified and $\Omega_c$ is ramified of level c. 

When $\Hom_{\E_v^*}(\pi'_v\otimes \Omega_v^{-1},\C)\neq 0$, pick $l$ to be a non-trivial element of it. Let $F_{1,v}\in \pi'$ be the unique (up to constant) element fixed by $R^*$. Then $l(F_{1,v})\neq 0$.
\end{prop}
\begin{example}\label{Grtest3}
Suppose that  $\E_v/\F_v$ is inert and $\B_v\simeq M_2(\F_v)$. $\E_v$ is embedded into $\GL_2(\F_v)$ just as in Example \ref{GrTest1}. Then we can pick $R=\{\zxz{O_F}{\varpi^cO_F}{\varpi^{-c}O_F}{O_F}\}$.
\end{example}

\subsection{Waldspurger's formula}
Shimizu's lifting and Waldspurger's result on the period integral were originally based on the convention that $\B^*\times \B^*$ acts on $\B$ via $(h_1,h_2)b=h_1b\bar{h_2}$. But later on people have reformulated their work with the convention that $\B^*\times \B^*$ acts on $\B$ via $(h_1,h_2)b=h_1bh_2^{-1}$, which is actually our setting. So we will state a variation of Waldspurger's original work according to our setting. 

Denote by $\Delta$ the modulus function for $\GL_2$ such that $\Delta(\zxz{a_1}{m}{0}{a_2} k)=|\frac{a_1}{a_2}|^{1/2}$.

\begin{theo}\label{Walds2}
(\textbf{Waldspurger})
Let $F_{1}\in\pi'$, $F_{2}\in\hat{\pi}'$. Let $\varphi\in\sigma$ such that $\theta(f,\varphi,h_{1},h_{2})=F_{1}(h_{1})F_{2}(h_{2})$ under the Shimizu lifting. Then
\begin{align}
&\text{\ \ \ \ }\int\limits_{Z_{\A}\E^{*}\backslash \E^{*}_{\A}} F_{1}(t_{1}h_{1}) \Omega^{-1}(t_{1})dt_{1}\int\limits_{Z_{\A}\E^{*}\backslash \E^{*}_{\A}} F_{2}(t_{2}h_{2})\Omega(t_{2}) dt_{2}\\
&=\int\limits_{N_{\A}Z_{\A}\backslash  \GL_{2}(\A)}\int\limits_{\E^{*}_{\A}} W_{\varphi}^-(\sigma)\Delta(\sigma)^{w-1/2}r'(\sigma)r''(h_{1},h_{2})f(t,Q(t)^{-1})\Omega(t)dtd\sigma|_{w=1/2} \notag\\
&=\frac{L(\Pi_\sigma\otimes\Omega^{-1},1/2)}{L(\eta,1)} \prod\limits_{v\in S}P_0(f_v,\Omega_v,1/2),\notag
\end{align}
where $W_{\varphi}^-$ is the Whittaker function corresponding to $\varphi$ with respect to $\psi^-(x)=\psi(-x)$. $\eta$ is the quadratic Hecke character associated to $\E/\F$. $S$ is the finite set of ramified places. $P_0(f_v,\Omega_v,w)$ is defined as
\begin{equation}
P_0(f_v,\Omega_v,w)=\frac{L_v(\eta_v,1)}{L_v(\Pi_{\sigma,v}\otimes\Omega_v^{-1},1/2)} \int\limits_{NZ\backslash  \GL_{2}(\F_v)}\int\limits_{\E_v^{*}} W_{\varphi,v}^-(\sigma)\Delta_v(\sigma)^{w-1/2}r'(\sigma)r''(h_1,h_2)f_v(t,Q(t)^{-1})\Omega_v(t)dtd\sigma.
\end{equation}
\end{theo}

The fact that the Whittaker function corresponds to $\psi^-$ is important for some specific local computations. 

Another way to formulate Waldspurger's result is as follows:
Theroem \ref{Tunnell} says $\Hom_{\E_v^*}(\pi'_v\otimes \Omega_v^{-1},\C)$ is 
one dimensional. So is $\Hom_{\E_v^*}(\pi'_v\otimes \Omega_v^{-1},\C)\otimes\Hom_{\E_v^*}(\hat{\pi'_v}\otimes \Omega_v,\C)$ and its tensor product over all places. One can explicitly write an element of it locally by matrix coefficients. Suppose that  $<\cdot,\cdot>:\pi_v'\times\hat{\pi_v'}\rightarrow \C$ is a bilinear and $\B_v^*-$ invariant pairing. 

Then 
\begin{align}\label{Waldsmatrixcoef}
\text{\ \ }\int\limits_{\F_v^*\backslash \E_v^*}\Omega_v^{-1}(e)<\pi_v'(e) F_{1,v},F_{2,v}>de=\int\limits_{\F_v^*\backslash \E_v^*}\Omega_v(e)< F_{1,v},\hat{\pi_v'}(e) F_{2,v}>de
\end{align}
gives an element of $\Hom_{\E_v^*}(\pi'_v\otimes \Omega_v^{-1},\C)\otimes\Hom_{\E_v^*}(\hat{\pi'_v}\otimes \Omega_v,\C)$.

Globally, 
\begin{equation*}
\int\limits_{Z_{\A}\E^{*}\backslash \E^{*}_{\A}} F_{1}(t_{1}) \Omega^{-1}(t_{1})dt_{1}\int\limits_{Z_{\A}\E^{*}\backslash \E^{*}_{\A}} F_{2}(t_{2})\Omega(t_{2}) dt_{2}
\end{equation*}
defines an element in $\Hom_{\E_\A^*}(\pi'\otimes \Omega^{-1},\C)\otimes\Hom_{\E_\A^*}(\hat{\pi'}\otimes \Omega,\C)$, which is also one dimensional. Then it should be a multiple of the product of local integrals given by (\ref{Waldsmatrixcoef}), that is,
\begin{align}\label{WaldsbyMatrixcoef}
&\text{\ \ }\int\limits_{Z_{\A}\E^{*}\backslash \E^{*}_{\A}} F_{1}(t_{1}) \Omega^{-1}(t_{1})dt_{1}\int\limits_{Z_{\A}\E^{*}\backslash \E^{*}_{\A}} F_{2}(t_{2})\Omega(t_{2}) dt_{2}\\
&=C_0\prod\limits_{v}\int\limits_{\F_v^*\backslash \E_v^*}\Omega_v(e)< F_{1,v},\hat{\pi_v'}(e) F_{2,v}>de.\notag
\end{align}

Waldspurger's theorem essentially gives a way to compute the constant $C_0$ explicitly. 
Let 
\begin{equation}
\alpha_v(F_{1,v},F_{2,v})=\frac{L_v(\eta_v,1)L_v(\pi_v,ad,1)}{\zeta_{\F,v}(2)L_v(\Pi_{\sigma,v}\otimes\Omega_v^{-1},1/2)}\int\limits_{\F_v^*\backslash \E_v^*}\Omega_v(e)< F_{1,v},\hat{\pi_v'}(e) F_{2,v}>de.
\end{equation}
Then Waldspurger's formula can be written as 
\begin{align}
&\text{\ \ }\int\limits_{Z_{\A}\E^{*}\backslash \E^{*}_{\A}} F_{1}(t_{1}) \Omega^{-1}(t_{1})dt_{1}\int\limits_{Z_{\A}\E^{*}\backslash \E^{*}_{\A}} F_{2}(t_{2})\Omega(t_{2}) dt_{2}\\
&=\frac{\zeta_{\F}(2)L(\Pi_{\sigma}\otimes\Omega^{-1},1/2)}{8L(\eta,1)^2L(\pi,ad,1)}\prod\limits_{v}\alpha_v(F_{1,v},F_{2,v}).\notag
\end{align}

For more details of this formulation and also the local calculations at some ramified places in the set $S$, see \cite{YZZ}\cite{SWZhang}\cite{Xue}\cite{Popa}.

\section{Global analysis}
Recall we want to work out the following integral:
$$\I(E,F,s)=\int\limits_{Z_{\A}\GL_{2}(\F)\backslash \GL_{2}(\A_{\F})} F(g)E(g,s) dg,$$
where $E(g,s)$ is an Eisenstein series defined over a quadratic algebra $\E$. Write $\E$ as $\F(\sqrt{D})$, where $D\in \F$ is an integer. 
If $\E\simeq \F\oplus\F$, take $D=1$. Then $\E^*$ can be embedded into $\GL_{2}(\F)$ via $1\mapsto I$ and $\sqrt{D} \mapsto\zxz{0}{1}{D}{0}$. Note that the quadratic norm is consistent with the determinant of matrices for this embedding.
\begin{lem}
\begin{equation}\label{eq2}
\I(E,F,s)=\int\limits_{Z_{\A}\E^{*}\backslash \GL_{2}(\A)} \Phi_s(\zxz{1}{0}{\sqrt{D}}{1} g)F(g)dg.
\end{equation}
\end{lem}
\begin{proof}
We  start with analyzing the orbits of $B(\E)\backslash \GL_{2}(\E)$ under right multiplication of $\GL_{2}(\F)$. In particular, by Bruhat decomposition, $$\GL_{2}(\E)=B\bigcup (\bigcup\limits_{n\in \E}B\omega\zxz{1}{n}{0}{1})=B\bigcup (\bigcup\limits_{m\in \E}B\zxz{1}{0}{m}{1}\omega)$$ 
for $\omega =\zxz{0}{1}{-1}{0}$.
Note $\omega\in \GL_{2}(\F)$. We want to know when $B(\E)\zxz{1}{0}{m_1}{1}$
is equivalent to 
$B(\E)\zxz{1}{0}{m_2}{1}$
for the right $\GL_{2}(\F)$ action. Equivalently, this is to ask when 
$$\zxz{1}{0}{m_1}{1}\zxz{a}{b}{c}{d}\zxz{1}{0}{-m_2}{1}\in B(\E) $$ 
for some $\zxz{a}{b}{c}{d}\in \GL_2(\F)$. 
By equating the lower left element of the product to 0, we get the following condition: $$m_{2}=\frac{am_{1}+c}{bm_{1}+d}.$$

 There are only two equivalent classes under the above condition, that is, $m=0$ and $m=\sqrt{D}$ as representatives.
We can also use the same condition to decide stabilizers of these two orbits under $\GL_{2}(\F)$ action. Let $m_1=m_2=m$, so we get 
$$m=\frac{am+c}{bm+d}.$$ 
This is equivalent to $bm^{2}+(d-a)m-c = 0$.
\begin{enumerate}
\item[(i)]Case $m=0$, the stabilizer is $\{c=0\}=N(\F)$, the unipotent subgroup. The corresponding orbit is called negligible.
\item[(ii)]Case $m=\sqrt{D}$, the stabilizer is $\{d=a,c=bD\}=\{aI+b\zxz{0}{1}{D}{0}\}$, which can be further identified with $\E^*$.
\end{enumerate}
So $B(\E)\backslash \GL_2(\E)=\zxz{1}{0}{0}{1}N(\F)\backslash \GL_2(\F)\cup \zxz{1}{0}{\sqrt{D}}{1}\E^*\backslash \GL_2(\F) $. Then (\ref{eq1}) can be rewritten as 
\begin{align}
&\text{\ \ \ \ \ }\I(E,F,s)\\
&=\int\limits_{Z_{\A}\GL_2(\F)\backslash \GL_{2}(\A)}\sum\limits_{\gamma\in B(\E)\backslash \GL_2(\E)} \Phi_s(\gamma g)F(g)dg\notag\\
&=\int\limits_{Z_{\A}\GL_2(\F)\backslash \GL_{2}(\A)} (\sum\limits_{\alpha\in N(\F)\backslash \GL_2(\F)}\Phi_s(\zxz{1}{0}{0}{1}\alpha g)+\sum\limits_{\alpha\in \E^*\backslash \GL_2(\F)}\Phi_s(\zxz{1}{0}{\sqrt{D}}{1}\alpha g))F(g)dg\notag\\
&=\int\limits_{Z_{\A}N(\F)\backslash \GL_{2}(\A)} \Phi_s(g)F(g)dg+\int\limits_{Z_{\A}\E^*\backslash \GL_{2}(\A)} \Phi_s(\zxz{1}{0}{\sqrt{D}}{1}g)F(g)dg.\notag
\end{align}
One just has to see that the first term 
$$\int\limits_{Z_{\A}N(\F)\backslash \GL_{2}(\A)} \Phi_s(g)F(g)dg=\int\limits_{Z_{\A}N(\A)\backslash \GL_{2}(\A)}\int\limits_{N(\F)\backslash N(\A)} \Phi_s(g)F(ng)dndg=0,$$
since $F$ is a cusp form. This is why the corresponding orbit is called negligible.
\end{proof}
We denote $\gamma_0 =\zxz{1}{0}{\sqrt{D}}{1}$ from now on. 
For $t=\zxz{a}{b}{bD}{a}\in \E^{*}_{\A}$, one can check that $$\gamma_0  t\gamma_0^{-1}=\zxz{a-b\sqrt{D}}{b}{0}{a+b\sqrt{D}}$$
is actually upper triangular. Recall $\Phi_s$ satisfies  $$\Phi_{s}(\zxz{a}{b}{0}{d}x)=\chi_{1}(a)\chi_{2}(d)|\frac{a}{d}|_{\A}^{s+1/2}\Phi_{s}(x).$$ From now on we fix our notation for $\Omega$ as follows:
\begin{defn}\label{DefnofOmega}

Define for $t\in \E_\A^*$
\begin{equation}
\Omega(t)=\chi_1(\overline{t})\chi_2(t)=\chi_{1}(a-b\sqrt{D})\chi_{2}(a+b\sqrt{D}).
\end{equation}
\end{defn}
\begin{lem}\label{lemofOmega}
With notations as above, we have $\Phi_s (\gamma_0 tg)=\Phi_s (\gamma_0 g) \Omega(t)$ for any $g\in \GL_{2}(\A)$ and $t\in \E_\A^*$.
\end{lem}
\begin{proof}From definition and that $\gamma_0 t\gamma_0^{-1}=\zxz{a-b\sqrt{D}}{b}{0}{a+b\sqrt{D}}$, We have
\begin{align*}
\Phi_s (\gamma_0 tg) & =\Phi_s (\zxz{a-b\sqrt{D}}{b}{0}{a+b\sqrt{D}}\gamma_0 g)\\
&=\chi_{1}(a-b\sqrt{D})\chi_{2}(a+b\sqrt{D})|\frac{a-b\sqrt{D}}{a+b\sqrt{D}}|_{\E_\A}^{s+1/2}\Phi_s (\gamma_0 g)\\
&=\chi_{1}(a-b\sqrt{D})\chi_{2}(a+b\sqrt{D})\Phi_s (\gamma_0 g)=\Phi_s (\gamma_0 g)\Omega(t).
\end{align*}
Here $|\frac{a-b\sqrt{D}}{a+b\sqrt{D}}|_{\E_\A}=1$ because $a,b\in\A_{\F}$.

\end{proof}
Now we can further write (\ref{eq2}) as 
\newline
\begin{equation}\label{eq3}
\I(E,F,s)=\int\limits_{\E^{*}_{\A}\backslash \GL_{2}(\A)} \Phi_s(\gamma_0 g)\int\limits_{Z_{\A}\E^{*}\backslash \E^{*}_{\A}}F(tg)\Omega(t)dtdg.
\end{equation}

Note that the interior part of the integral is a period integral for which one can apply Theorem \ref{Walds2}. The whole integral can be thought of as a weighted integral of period integrals. 


To fit into Theorem \ref{Walds2}, take the quaternion algebra $\B$ there to be $M_2(\F)$. Pick $F_{2}=F\in \pi=\hat{\pi'}$. Then $F_1\in \hat{\pi}\simeq\pi'$ and $\varphi \in \hat{\pi}\simeq \sigma$. (The way we use Waldspurger's formula may look a little strange. We will see the reason later.) Pick $\Omega$ as the one we defined above, and pick $h_1\equiv 1$, $h_2=g$. Then there are two possible situations: 

First, if 
$\int\limits_{Z_{\A}\E^{*}\backslash \E^{*}_{\A}}F(tg)\Omega(t)dt=0$ for any $g$, then $\I(E,F,s)=0$. In particular we have the following corollary:
\begin{cor}\label{Corofautomatic0}
Let $\Pi$ be the base change of $\pi$ to $\E$. If $\Hom_{\E_\A^*}(\pi\otimes\Omega,\C)=0$ or $L(\Pi\otimes\Omega,1/2)=0$, then $\I(E,F,s)=0$.
\end{cor}

Second, if that's not the case, we have the following lemma:
\begin{lem}\label{lemcontraperiod}
 If $\int\limits_{Z_{\A}\E^{*}\backslash \E^{*}_{\A}}F(tg)\Omega(t)dt\neq 0$ for some $g\in \GL_2(\A)$, then there exists some $F_1\in\hat{\pi}$ such that $\int\limits_{Z_{\A}\E^{*}\backslash \E^{*}_{\A}}F_1(t)\Omega^{-1}(t)dt\neq 0$
\end{lem}

\begin{proof}
It's known (see \cite{Bump}) that if $F\in \pi$ is an automorphic form, then $F_1(x)=F(^T x^{-1} g')$ for any $g'\in \GL_2(\A)$ is an automorphic form in $\hat{\pi}$, where $^T x$ is the transpose of $x$. So 
\begin{align*}
 \int\limits_{Z_{\A}\E^{*}\backslash \E^{*}_{\A}}F_1(t)\Omega^{-1}(t)dt &= \int\limits_{Z_{\A}\E^{*}\backslash \E^{*}_{\A}}F(^T t^{-1}g')\Omega^{-1}(t)dt\\
&=\int\limits_{Z_{\A}\E^{*}\backslash \E^{*}_{\A}}F(^T tg')\Omega(t)dt.
\end{align*}
The second equality follows from the substitution $t\rightarrow t^{-1}$, and the multiplicative Haar measure is invariant under this substitution. If $D\neq 1$, the matrix $g_0=\zxz{D}{-1}{-1}{1}\in \GL_2(\F)$ satisfies $g_0 tg_0^{-1}=^Tt$ for all $t\in \E_\A^*$. If $D=1$, then $t=^Tt$ and we can pick $g_0=1$. Now pick $g'=g_0g$ and then
\begin{align*}
 \int\limits_{Z_{\A}\E^{*}\backslash \E^{*}_{\A}}F_1(t)\Omega^{-1}(t)dt &=\int\limits_{Z_{\A}\E^{*}\backslash \E^{*}_{\A}}F(^T tg_0g)\Omega(t)dt\\
&=\int\limits_{Z_{\A}\E^{*}\backslash \E^{*}_{\A}}F(g_0tg)\Omega(t)dt\\
&=\int\limits_{Z_{\A}\E^{*}\backslash \E^{*}_{\A}}F(tg)\Omega(t)dt\neq 0.
\end{align*}
The last equality follows from the fact that $F$ is an automorphic form and thus is left invariant by $g_0\in \GL_2(\F)$.
\end{proof}

\begin{rem}
 It's actually possible to show that $\int\limits_{Z_{\A}\E^{*}\backslash \E^{*}_{\A}}F(tg)\Omega(t)dt$ is not identically zero if and only if $\Hom_{\E_\A^*}(\pi\otimes\Omega,\C)\neq 0$ and $L(\Pi\otimes\Omega,1/2)\neq 0$.
\end{rem}

So if $\int\limits_{Z_{\A}\E^{*}\backslash \E^{*}_{\A}}F(tg)\Omega(t)dt$ is not identically zero, we fix a $F_1\in\hat{\pi}$ such that the period integral $C=\int\limits_{Z_{\A}\E^{*}\backslash \E^{*}_{\A}}F_{1}(t_{1})\Omega^{-1}(t_{1})dt_{1}$ is not zero. 

Note that since $\Omega$ does not depend on $s$, 
this fixed period integral $C$ is also independent of $s$.
Then by Theorem \ref{Walds2} we have the following relation:
\begin{align}\label{Globalpaired}
 &\text{\ \ \ }C\cdot\I(E,F,s) 
 \\
&=\int\limits_{\E^{*}_{\A}\backslash \GL_{2}(\A)} \Phi_s(\gamma_0 g)\int\limits_{Z_{\A}\E^{*}\backslash \E^{*}_{\A}}F_{1}(t_{1})\Omega^{-1}(t_{1})dt_{1}\int\limits_{Z_{\A}\E^{*}\backslash \E^{*}_{\A}}F(t_2g)\Omega(t_2)dt_2dg\notag\\
&=\int\limits_{\E^{*}_{\A}\backslash \GL_{2}(\A)} \Phi_s(\gamma_0 g)\int\limits_{Z_{\A}N_{\A}\backslash \GL_{2}(\A)}\int\limits_{\E^{*}_{\A}} W_{\varphi}^-(\sigma)\Delta(\sigma)^{w-1/2}r'(\sigma)r''(1,g)f(t,Q(t)^{-1})\Omega(t)dtd\sigma dg|_{w=1/2}.\notag
\end{align}

The theta lifting for $\varphi\in\hat{\pi}$ is $\theta(f,\varphi,h_1,h_2)=F_1(h_1)F_2(h_2)$, just as in Theorem \ref{Walds2}.
\newline
Recall $\Phi_s (\gamma_0 tg)=\Phi_s (\gamma_0 g) \Omega(t)$. By the definition of the Weil representation, in particular by formula (\ref{quatcal}), we have $r''(1,g)f(t,Q(t)^{-1})=f(tg,\det (tg)^{-1})$. Then we can actually combine the integrals in $t$ and $g$. This is the reason we applied Waldspurger's formula in a seemingly strange way.
Now we have the following main proposition for the global situation:
\begin{prop}\label{mainprop}
\begin{enumerate}
\item[(1)] If $\int\limits_{Z_{\A}\E^{*}\backslash \E^{*}_{\A}}F(tg)\Omega(t)dt$ is always zero, then $\I(E,F,s)=0$.\\

\item[(2)] Otherwise, we can fix a nonzero period integral $$C=\int\limits_{Z_{\A}\E^{*}\backslash \E^{*}_{\A}}F_{1}(t_{1})\Omega^{-1}(t_{1})dt_{1}$$ which is independent of $s$. Then
\begin{equation}
C\cdot\I(E,F,s)=\int\limits_{Z_{\A}N_{\A}\backslash \GL_{2}(\A)}\int\limits_{\GL_{2}(\A)} W(\sigma)\Delta(\sigma)^{w-1/2}r'(\sigma)f(g,\det (g)^{-1})\Phi_s(\gamma_0 g)dgd\sigma|_{w=1/2},
\end{equation} 
where we write $W(\sigma)=W_{\varphi}^-(\sigma)$ for short.
\end{enumerate}
\end{prop}

\begin{rem}
It seems that we are adding some equally complicated constant $C$ to solve our problem. But $C$ is already studied in detail in several papers after Waldspurger. (See \cite{Popa}\cite{Xue}\cite{SWZhang}.) More importantly, this constant $C$ does not depend on $s$. So it is easy to deal with if we are interested in, for example, the zeros of $\I(E,F,s)$ as a function of $s$, or zeros of its derivative, etc,.
\end{rem}

When we introduced Waldspurger's formula, we mentioned the way to formulate the local integrals by matrix coefficients. We can do similar things here. Start with (\ref{Globalpaired}) and then use (\ref{WaldsbyMatrixcoef}), we will get
\begin{align}\label{MybyMatrixcoef}
 &C\cdot\int\limits_{\E^{*}_{\A}\backslash \GL_{2}(\A)} \Phi_s(\gamma_0 g)\int\limits_{Z_{\A}\E^{*}\backslash \E^{*}_{\A}}F(tg)\Omega(t)dtdg\\
&=\int\limits_{\E^{*}_{\A}\backslash \GL_{2}(\A)} \Phi_s(\gamma_0 g)\int\limits_{Z_{\A}\E^{*}\backslash \E^{*}_{\A}}F_{1}(t_{1})\Omega^{-1}(t_{1})dt_{1}\int\limits_{Z_{\A}\E^{*}\backslash \E^{*}_{\A}}F(t_2g)\Omega(t_2)dt_2dg\notag\\
&=C_0\prod\limits_{v}\int\limits_{\E_v^*\backslash \GL_2(\F_v)}\Phi_{s,v}(\gamma_0 g)\int\limits_{\F_v^*\backslash \E_v^*}\Omega_v(e)< F_{1,v},\pi_v(eg) F_{v}>dedg\notag\\
&=C_0\prod\limits_{v}\int\limits_{\F_v^*\backslash \GL_2(\F_v)}\Phi_{s,v}(\gamma_0 g)< F_{1,v},\pi_v(g) F_{v}>dg.\notag
\end{align}


\section{Local calculation at unramified places}
In this section and the next two sections, we will do calculations for the local integrals. For the simplicity of notations, we will omit subscript $v$ and everything should be understood locally. 
First we recall some notations. $\F$ is the local field at a non-archimedean place $v$ and $\E$ is a quadratic extension of it. $O_F$ is the ring of integer for $\F$ with a fixed uniformizer $\varpi$. Denote by $v(x)$ the valuation of $x\in \F^*$.  $q=|\varpi|^{-1}$. $K=\GL_{2}(O_{F})$.

For a multiplicative character $\chi$, we write $\chi$ for $\chi(\varpi)$ if we don't specify which element it's taking. 
For simplicity of the formulae we will write $\chi_{1,s}=\chi_1|\cdotp|_{\E }^{s+1/2}$ and $\chi_{2,s}=\chi_2|\cdot|_{\E}^{-s-1/2}$. Then by the definition of $\Phi_s$, 
$$\Phi_s(\zxz{a_1}{m}{0}{a_2}g)=\chi_{1,s}(a_1)\chi_{2,s}(a_2)\Phi_s(g).$$

When $\pi$ is unramified at $v$, we write $\pi=\pi(\mu_1,\mu_2)$. Then $\varphi\in\hat{\pi}=\pi(\mu_1^{-1},\mu_2^{-1})$ and the central character of $\hat{\pi}$ satisfies $w_{\hat{\pi}}=\chi_{1,s}\chi_{2,s}|_\F$.

We will calculate in this section the following local integral that appears in Proposition \ref{mainprop} for unramified places: 
\begin{equation}\label{localtarget}
\P (s,w,f,\Phi_s)=\int\limits_{ZN\backslash \GL_{2}(\F )}W(\sigma)\Delta(\sigma)^{w-1/2}\int\limits_{\GL_{2}(\F )}r'(\sigma)f(g,\det (g)^{-1})\Phi_{s}(\gamma_0 g) dgd\sigma.
\end{equation}

We specify here what we mean by an unramified place: the quadratic extension $\E$ over $\F$ is either inert or split at this place; $\pi $ is unramified and the corresponding Whittaker function $W(\sigma)$ is right $K$-invariant and normalized; $\chi_{i,s}$ is unramified for $i=1,2$; $\Phi_s$ is right $K-$ invariant and normalized; $f$ is the  Schwartz function $char(\zxz{O_F}{O_F}{O_F}{O_F})\times char(O_{F}^{*})$. 

For all the local calculations, we will always pick $W$ to be the normalized Whittaker function for the newform of the representation, so we won't take it as a variable for $\P$. 
We say a $K-$invariant function $W$ (or $\Phi_s$) is normalized if $W(1)=1$.
If it is only right $K_1(\varpi^c)-$invariant and supported on $B\zxz{1}{0}{\varpi^j}{1}K_1(\varpi^c)$ for some $0\leq j\leq c$, we mean $W(\zxz{1}{0}{\varpi^j}{1})=1$ by saying it's normalized.
As $f$ and $\Phi_s$ will be fixed in this section, we will write $\P (s,w)$ in short for $\P (s,w,f,\Phi_s)$.

Assume that the Haar measure on $\GL_2$ is so normalized that the volume of $K$ is $1$. Since $W$ is $K$-invariant and $f$ is also $K-$invariant under the Weil representation $r'$,
\begin{equation}
\P (s,w)=\int\limits_{\F^{*}}W(\zxz{\alpha}{0}{0}{1}) |\alpha |^{w/2-1/4}
\int\limits_{\GL_{2}(\F)}r'(\zxz{\alpha}{0}{0}{1})f(g,\det (g)^{-1})\Phi_{s}(\gamma_0 g)dg|\alpha |^{-1}d^*\alpha.
\end{equation}
Here we have used that the Haar measure for the Iwasawa decomposition $ZN\{\zxz{\alpha}{0}{0}{1}\}K$ is 
$$d^*zdn|\alpha|^{-1}d^*\alpha dk.$$
By the definition of the Weil representation, in particular by equation (\ref{Weilsilim}),

$\int\limits_{\GL_{2}(\F )}r'(\zxz{\alpha}{0}{0}{1})f(g,\det (g)^{-1})\Phi_{s}(\gamma_0 g)dg=|\alpha |\int\limits_{\GL_{2}(\F )}f(\alpha g,\alpha^{-1}\det (g)^{-1})\Phi_{s}(\gamma_0 g)dg$

By substituting $\alpha g\rightarrow g$, we get

$$|\alpha |\int\limits_{\GL_{2}(\F )}f(\alpha g,\alpha^{-1}\det (g)^{-1})\Phi_{s}(\gamma_0 g)dg=|\alpha |\Phi_{s}(\alpha)^{-1}\int\limits_{\GL_{2}(\F )}f(g,\alpha \det (g)^{-1})\Phi_{s}(\gamma_0 g)dg.$$
To be precise, $\Phi_s(\alpha)$ here should be understood as $\chi_1\chi_2(\alpha)$ which is actually independent of $s$. 
Then the local integral becomes 
\begin{equation}
\label{local1}
\int\limits_{\F ^{*}}W(\zxz{\alpha}{0}{0}{1}) |\alpha |^{w/2-1/4} \Phi_{s}(\alpha)^{-1}
\int\limits_{\GL_{2}(\F )}f(g,\alpha \det (g)^{-1})\Phi_{s}(\gamma_0 g)dgd^*\alpha.
\end{equation}
Denote 
\begin{equation}\label{DefofI}
I(\alpha,f,\Phi_s)=\int\limits_{\GL_{2}(\F )}f(g,\alpha \det (g)^{-1})\Phi_{s}(\gamma_0 g)dg.
\end{equation}
Again we write $I(\alpha)$ in short for $I(\alpha,f,\Phi_s)$ in this section. 
For unramified places $f$ and $\Phi_{s}$ are right $K-$ invariant, so we just have to do the integral over $B(\F )$ for $I(\alpha)$. Note that the right $K-$invariance of $f(g,\alpha \det (g)^{-1})$ as a function of $g$ is the same as the right invariance of $f(x,u)$ as clarified in subsection (2.2). Let $n=v(a_1)$, $k=v(m)$, $l=v(a_2)$. By definition, $\zxz{a_1}{m}{0}{a_2}$ is in the support of $f$ if and only if $n,l,k\geq 0$ and $\frac{\alpha}{a_1a_2}\in O_F^*$. The latter implies $l+n=v(\alpha)$. So
\begin{equation}\label{localI}
I(\alpha)=\int\limits_{0\leq n\leq v(\alpha)}\int\limits_{k\geq 0}\int\limits_{l=v(\alpha)-n}\Phi_s(\zxz{1}{0}{\sqrt{D}}{1}\zxz{a_1}{m}{0}{a_2})d^*a_2|a_1|^{-1}dmd^*a_1.
\end{equation}
Here we have used that the left Haar measure for the Borel subgroup is $$d^*a_2|a_1|^{-1}dmd^*a_1.$$
We need to work out the value of $\Phi_s$ more explicitly. 
\begin{lem}
\label{localcases}
Note $\Phi_s(\zxz{1}{0}{\sqrt{D}}{1}\zxz{a_1}{m}{0}{a_2})=\Phi_s(\zxz{a_1}{m}{a_1\sqrt{D}}{a_2+m\sqrt{D}})$.
\begin{enumerate}
\item[(1)]If $v(a_2+m\sqrt{D})\geq v(a_1\sqrt{D})$, then $\Phi_s(\zxz{a_1}{m}{a_1\sqrt{D}}{a_2+m\sqrt{D}})=\chi_{1,s}(\frac{a_2}{\sqrt{D}})\chi_{2,s}(a_1\sqrt{D})$.
\item[(2)]If $v(a_2+m\sqrt{D})\leq v(a_1\sqrt{D})$, then $\Phi_s(\zxz{a_1}{m}{a_1\sqrt{D}}{a_2+m\sqrt{D}})=\chi_{1,s}(\frac{a_1a_2}{a_2+m\sqrt{D}})\chi_{2,s}(a_2+m\sqrt{D})$.
\end{enumerate}
\end{lem}
\begin{proof} (1)When $v(a_2+m\sqrt{D})\geq v(a_1\sqrt{D})$, 
\begin{align*}\zxz{a_1}{m}{a_1\sqrt{D}}{a_2+m\sqrt{D}}&=\zxz{1}{*}{0}{1}\zxz{0}{-\frac{a_2}{\sqrt{D}}}{a_1\sqrt{D}}{a_2+m\sqrt{D}}\\
 &=\zxz{1}{*}{0}{1}\zxz{\frac{a_2}{\sqrt{D}}}{0}{0}{a_1\sqrt{D}}\zxz{0}{-1}{1}{\frac{a_2+m\sqrt{D}}{a_1\sqrt{D}}}.
\end{align*}
(2)When $v(a_2+m\sqrt{D})\leq v(a_1\sqrt{D})$, 
\begin{align*}
\zxz{a_1}{m}{a_1\sqrt{D}}{a_2+m\sqrt{D}}&=\zxz{1}{*}{0}{1}\zxz{\frac{a_1a_2}{a_2+m\sqrt{D}}}{0}{a_1\sqrt{D}}{a_2+m\sqrt{D}}\\
&=\zxz{1}{*}{0}{1}\zxz{\frac{a_1a_2}{a_2+m\sqrt{D}}}{0}{0}{a_2+m\sqrt{D}}\zxz{1}{0}{\frac{a_1\sqrt{D}}{a_2+m\sqrt{D}}}{1}.                              
\end{align*}
Then the statements follow from the definition of $\Phi_s$ and its right $K-$ invariance.
\end{proof}
Now we have to consider the inert places separately from the split places. 

\subsection{Inert places}
In this subsection we assume $v$ is an inert place. As a result, $v(a_2+m\sqrt{D})=\min\{l,k\}$. Note that for this place $\sqrt{D}$ will be a unit in the local field. Then by the lemma above, we get
\begin{lem}\label{localPhi}
\begin{enumerate}
\item[1.]If $0\leq n \leq \frac{v(\alpha)}{2}$, then $l=v(\alpha)-n\geq n$.
\begin{enumerate}
\item[(1i)] If $k\geq n$, $\Phi_s(\zxz{a_1}{m}{a_1\sqrt{D}}{a_2+m\sqrt{D}})=\chi_{1,s}(\frac{a_2}{\sqrt{D}})\chi_{2,s}(a_1\sqrt{D})=\chi_{1,s}^{v(\alpha)-n}\chi_{2,s}^{n}$.
\item[(1ii)] If $0\leq k< n$, $\Phi_s(\zxz{a_1}{m}{a_1\sqrt{D}}{a_2+m\sqrt{D}})=\chi_{1,s}(\frac{a_1a_2}{a_2+m\sqrt{D}})\chi_{2,s}(a_2+m\sqrt{D})=\chi_{1,s}^{v(\alpha)-k}\chi_{2,s}^{k}$.
\end{enumerate}
\item[2.]If $\frac{v(\alpha)}{2}\leq n\leq v(\alpha)$, then $l\leq n$.
\begin{enumerate}
\item[(2i)] If $k\geq l$, $\Phi_s(\zxz{a_1}{m}{a_1\sqrt{D}}{a_2+m\sqrt{D}})=\chi_{1,s}^n\chi_{2,s}^{v(\alpha)-n}$.
\item[(2ii)] If $k<l$, $\Phi_s(\zxz{a_1}{m}{a_1\sqrt{D}}{a_2+m\sqrt{D}})=\chi_{1,s}^{v(\alpha)-k}\chi_{2,s}^{k}$.
\end{enumerate}
\end{enumerate}
\end{lem}
It's clear that $\int\limits_{O^*_F}d^*a=1$ , $\int\limits_{O^*_F}dm=1-q^{-1}$ and $|a_1|^{-1}=q^n$. 
Then the integral (\ref{localI}) becomes
\begin{equation}\label{inertsum}
I=\sum\limits_{0\leq n\leq\frac{v(\alpha)}{2}}[\sum\limits_{0\leq k< n}\chi_{1,s}^{v(\alpha)-k}\chi_{2,s}^{k}q^{n-k}(1-q^{-1})
+\sum\limits_{n\leq k<\infty}\chi_{1,s}^{v(\alpha)-n}\chi_{2,s}^{n}q^{n-k}(1-q^{-1})]
\end{equation}
\begin{equation*}
+\sum\limits_{\frac{v(\alpha)}{2}< n\leq v(\alpha)}[\sum\limits_{0\leq k< v(\alpha)-n}\chi_{1,s}^{v(\alpha)-k}\chi_{2,s}^{k}q^{n-k}(1-q^{-1})+
\sum\limits_{v(\alpha)-n\leq k<\infty}\chi_{1,s}^{n}\chi_{2,s}^{v(\alpha)-n}q^{n-k}(1-q^{-1})].
\end{equation*}

Then it's a tedious process of summation and combining terms. We will skip the process and give the conclusion directly:
\begin{lem}
Let $v$ be an inert place. When $v(\alpha)<0$, $I=0$. When $v(\alpha)\geq 0$
\begin{align}\label{inertI}I(\alpha) &=\chi_{1,s}^{v(\alpha)}\frac{1-q^{v(\alpha)+1}}{1-q}\frac{1-q^{-1}}{1-\frac{\chi_{2,s}}{q\chi_{1,s}}}
+\chi_{1,s}^{v(\alpha)}\frac{1-(\frac{\chi_{2,s}}{\chi_{1,s}})^{b+1}}{1-\frac{\chi_{2,s}}{\chi_{1,s}}}\frac{q^{-1}(1-\frac{\chi_{2,s}}{\chi_{1,s}})}{1-\frac{\chi_{2,s}}{q\chi_{1,s}}}\\
&\text{\ \ }+(q^{-1}\chi_{2,s})^{v(\alpha)}\frac{(\frac{q^2\chi_{1,s}}{\chi_{2,s}})^{b+1}-(\frac{q^2\chi_{1,s}}{\chi_{2,s}})^{v(\alpha)+1}}{1-\frac{q^2\chi_{1,s}}{\chi_{2,s}}}\frac{q^{-1}(1-\frac{\chi_{2,s}}{\chi_{1,s}})}{1-\frac{\chi_{2,s}}{q\chi_{1,s}}}\text{\ \ \ \ \ \ \ if\ } v(\alpha)=2b,2b+1 \notag\\
&
=-\frac{q(1+q)\frac{\chi_{1,s}}{\chi_{2,s}}}{1-q^2\frac{\chi_{1,s}}{\chi_{2,s}}}\chi_{1,s}^{v(\alpha)}q^{v(\alpha)}+\frac{1}{1-q^2\frac{\chi_{1,s}}{\chi_{2,s}}}\begin{cases}(1+q\frac{\chi_{1,s}}{\chi_{2,s}})\chi_{1,s}^b\chi_{2,s}^b,&\text{if } v(\alpha)=2b;\\(1+q)\chi_{1,s}^{b+1}\chi_{2,s}^b,&\text{if } v(\alpha)=2b+1.\end{cases} \notag
\end{align}
\end{lem}
Now we return to integral (\ref{local1}). For an unramified place $v$ and $\varphi\in\pi(\mu_1^{-1},\mu_2^{-1})$, we have
\begin{equation}
W(\zxz{\alpha}{0}{0}{1})=\begin{cases}
|\alpha|^{1/2}\frac{\mu_1^{-1}(\varpi\alpha)-\mu_2^{-1}(\varpi\alpha)}{\mu_1^{-1}(\varpi)-\mu_2^{-1}(\varpi)},&\text{if }v(\alpha)\geq 0;\\0,&\text{otherwise}.
\end{cases}
\end{equation} 
By the condition (\ref{centercondition}), we have $\mu_1\mu_2\chi_1\chi_2=1$. Then the integral (\ref{local1}) becomes
\begin{align}
\int\limits_{\F ^{*}}W(\zxz{\alpha}{0}{0}{1}) |\alpha |^{w/2-1/4} \Phi_{s}(\alpha)^{-1}I d^*\alpha
&=\sum\limits_{v(\alpha)=0}^{\infty}|\alpha|^{\frac{w}{2}+\frac{1}{4}}\frac{\mu_1^{-1}(\varpi\alpha)-\mu_2^{-1}(\varpi\alpha)}{\mu_1^{-1}(\varpi)-\mu_2^{-1}(\varpi)}(\chi_1\chi_2)^{-1}(\alpha)I  \notag\\
&=\sum\limits_{v(\alpha)=0}^{\infty}q^{-(\frac{w}{2}+\frac{1}{4})v(\alpha)}\frac{\mu_1^{-(v(\alpha)+1)}-\mu_2^{-(v(\alpha)+1)}}{\mu_1^{-1}-\mu_2^{-1}}(\mu_1\mu_2)^{v(\alpha)}I  \notag\\
&=\sum\limits_{v(\alpha)=0}^{\infty}q^{-(\frac{w}{2}+\frac{1}{4})v(\alpha)}\frac{\mu_2^{v(\alpha)}\mu_1^{-1}-\mu_1^{v(\alpha)}\mu_2^{-1}}{\mu_1^{-1}-\mu_2^{-1}}I.  \label{inertapproach}
\end{align}
Let $\delta_i=q^{-(\frac{w}{2}+\frac{1}{4})}\mu_i$. We first calculate:

\begin{align}\label{deltaI}
&\text{\ \ }\sum\limits_{v(\alpha)=0}^{\infty}\delta_i^{v(\alpha)}I(\alpha)\\
 &=\sum\limits_{v(\alpha)=0}^{\infty}-\frac{q(1+q)\frac{\chi_{1,s}}{\chi_{2,s}}}{1-q^2\frac{\chi_{1,s}}{\chi_{2,s}}}\delta_i^{v(\alpha)}\chi_{1,s}^{v(\alpha)}q^{v(\alpha)}
+\sum\limits_{b=0}^{+\infty}\frac{1+q\frac{\chi_{1,s}}{\chi_{2,s}}}{1-q^2\frac{\chi_{1,s}}{\chi_{2,s}}}\delta_i^{2b}\chi_{1,s}^b\chi_{2,s}^b
+\sum\limits_{b=0}^{+\infty}\frac{1+q}{1-q^2\frac{\chi_{1,s}}{\chi_{2,s}}}\delta_i^{2b+1}\chi_{1,s}^{b+1}\chi_{2,s}^b\notag\\
&=-\frac{q(1+q)\frac{\chi_{1,s}}{\chi_{2,s}}}{1-q^2\frac{\chi_{1,s}}{\chi_{2,s}}}\frac{1}{1-\delta_i\chi_{1,s}q}
+\frac{1+q\frac{\chi_{1,s}}{\chi_{2,s}}}{1-q^2\frac{\chi_{1,s}}{\chi_{2,s}}}\frac{1}{1-\delta_i^2\chi_{1,s}\chi_{2,s}}
+\frac{1+q}{1-q^2\frac{\chi_{1,s}}{\chi_{2,s}}}\frac{\delta_i\chi_{1,s}}{1-\delta_i^2\chi_{1,s}\chi_{2,s}}\notag\\
&=\frac{1+\delta_i\chi_{1,s}}{(1-\delta_i\chi_{1,s}q)(1-\delta_i^2\chi_{1,s}\chi_{2,s})}.\notag
\end{align}
\begin{rem}
 It seems that the denominator $1-q^2\frac{\chi_{1,s}}{\chi_{2,s}}$ of formula (\ref{inertI}) somehow disappeared when we reach formula (\ref{deltaI}). But this denominator essentially comes from the summation of a finite geometric series in (\ref{inertsum}). So its cancellation is really not surprising.
\end{rem}

Put formula (\ref{deltaI}) back into integral (\ref{inertapproach}), we get
\begin{equation}\label{localmu's}
\P (s,w)=\frac{\mu_1^{-1}}{\mu_1^{-1}-\mu_2^{-1}}\sum\limits_{v(\alpha)=0}^{\infty}\delta_2^{v(\alpha)}I(\alpha)-\frac{\mu_2^{-1}}{\mu_1^{-1}-\mu_2^{-1}}\sum\limits_{v(\alpha)=0}^{\infty}\delta_1^{v(\alpha)}I(\alpha).
\end{equation}
Let $\delta=q^{-(\frac{w}{2}+\frac{1}{4})}$. Recall $\chi_1\chi_2\mu_1\mu_2=1$. We will skip tedious calculations and show results directly:
\begin{equation*}
\P (s,w)=\frac{(1+\delta^2)(1-q\frac{\chi_{1,s}}{\chi_{2,s}}\delta^2)+(\mu_1+\mu_2)\chi_{1,s}\delta(1-q\delta^2)}
{(1-q\mu_1\chi_{1,s}\delta)(1-q\mu_2\chi_{1,s}\delta)(1-\mu_1^2\chi_{1,s}\chi_{2,s}\delta^2)(1-\mu_2^2\chi_{1,s}\chi_{2,s}\delta^2)}.
\end{equation*}
Now if we evaluate $\P (s,w)$ at $w=1/2$, then $\delta=q^{-(\frac{w}{2}+\frac{1}{4})}=q^{-1/2}$, and the numerator of the above integral can be simplified greatly:
\begin{equation*}
\P (s,1/2)=\frac{1+q^{-1}}
{(1-\mu_1^2\chi_{1,s}\chi_{2,s}q^{-1})(1-\mu_2^2\chi_{1,s}\chi_{2,s}q^{-1})}
\times\frac{1-\frac{\chi_{1,s}}{\chi_{2,s}}}{(1-\mu_1\chi_{1,s}q^{1/2})(1-\mu_2\chi_{1,s}q^{1/2})}.
\end{equation*}
 For a character $\chi$ of $\F ^*$, define $s(\chi)$ to be the real number such that $|\chi(x)| =|x| ^{s(\chi)}$. We have following proposition for the inert case

\begin{prop}\label{propinert}
Let $v$ be a non-archimedean inert place for $\E/\F$. Suppose $\Re(s)\geq(s(\chi_2)-s(\chi_1))/4$. 
\begin{enumerate}
\item[(i)]There exists $\epsilon >0$ such that, for $D=\{w\in\C;\Re (w)>1/2-\epsilon \}$, the integral $\P (s,w)$ converges uniformly in any compact subset of $D$. It's holomorphic in $D$. 
\item[(ii)] For a general inert place we have:
\begin{equation*}
\P (s,w)=\frac{(1+\delta^2)(1-q\frac{\chi_{1,s}}{\chi_{2,s}}\delta^2)+(\mu_1+\mu_2)\chi_{1,s}\delta(1-q\delta^2)}
{(1-q\mu_1\chi_{1,s}\delta)(1-q\mu_2\chi_{1,s}\delta)(1-\mu_1^2\chi_{1,s}\chi_{2,s}\delta^2)(1-\mu_2^2\chi_{1,s}\chi_{2,s}\delta^2)}.
\end{equation*}
where $\delta_i=q^{-(\frac{w}{2}+\frac{1}{4})}\mu_i$. If we evaluate at $w=1/2$, or equivalently $\delta=q^{-1/2}$, then we get
\begin{equation*}
\P (s,1/2)=\frac{1+q^{-1}}
{(1-\mu_1^2\chi_{1,s}\chi_{2,s}q^{-1})(1-\mu_2^2\chi_{1,s}\chi_{2,s}q^{-1})}
\times\frac{1-\frac{\chi_{1,s}}{\chi_{2,s}}}{(1-\mu_1\chi_{1,s}q^{1/2})(1-\mu_2\chi_{1,s}q^{1/2})}.
\end{equation*}
Writing out variable $s$ explicitly, we get
\begin{align*}\label{formulainert}
\P (s,1/2)&=\frac{1+q^{-1}}
{(1-\mu_1^2\chi_{1}\chi_{2}q^{-1})(1-\mu_2^2\chi_{1}\chi_{2}q^{-1})}
\times\frac{1-\frac{\chi_{1}}{\chi_{2}}q^{-(4s+2)}}{(1-\mu_1\chi_{1}q^{-(2s+1/2)})(1-\mu_2\chi_{1}q^{-(2s+1/2)})}\\
&=\frac{L (\Pi\otimes\Omega,1/2)L (\pi\otimes\chi_{1}|_{\F ^*},2s+1/2)}{L (\eta,1)L^{\E} (\chi,2s+1)}.
\end{align*}
Recall $\chi=\frac{\chi_1}{\chi_2}$. $L^{\E} $ means it's a product of L factors over all places of $\E$ above $v$. In this case there is only one place with the order of the residue field being $q^2$.
\end{enumerate}
\end{prop}

\begin{proof} To finish the proof of the formulae in (ii), one just has to write out variable $s$ explicitly. Recall that  $\chi_{1,s}=\chi_1|\cdotp|_{\E }^{s+1/2}$. The absolute value is the extension of $v$ to $\E $, thus $\chi_{1,s}(\varpi)=\chi_1q^{-(2s+1)}$. So is $\chi_{2,s}$. Then the last formula is clear.

For part (i), we just imitate Waldspurger's original proof for his local integrals. Denote by $w_{\hat{\pi}}=\mu_1^{-1}\mu_2^{-1}$ the cental character for $\hat{\pi}$. $s(w_{\hat{\pi}})=-s(\mu_1)-s(\mu_2)=s(\chi_1)+s(\chi_2)$. Note that $W,f$ are both $K-$finite, so the absolute convergence can be essentailly reduced to a bound for the following integral
\begin{equation}
J(w,s)=\int\limits_{\F ^{*}}|W(\zxz{\alpha}{0}{0}{1})| |\alpha |^{w/2-1/4} |\Phi_{s}(\alpha)^{-1}|
\int\limits_{\GL_{2}(\F )}|f(g,\alpha \det (g)^{-1})||\Phi_{s}(\gamma_0 g)|dgd^*\alpha,
\end{equation}
which follows from integral (\ref{local1}). It's known that $|W(\zxz{\alpha}{0}{0}{1})|=O(|\alpha|^{\nu}|w_{\hat{\pi}}(\alpha)|^{1/2})$ for some $\nu>0$ when $|\alpha|\rightarrow 0$. When $|\alpha|\rightarrow\infty$, it's rapidly decreasing for archimedean places and compactly supported for non-archimedean places. So the main trouble is when $|\alpha|\rightarrow 0$, or $v(\alpha)\rightarrow\infty$. The idea to deal with this is to make use of the calculations we already did for the general places.

For a non-archimedean place, one can find fixed integers $c_1$, $c_2$, such that the Schwartz function $f(x,u)$ is supported on $c_1\leq v(u)\leq c_2$. One can further find integers $N$, $a$ and a Schwartz function $f_0=char(\varpi^{-a}M_2(O_F))$ such that $|f(x,u)|\leq Nf_0(x)$ when $c_1\leq v(u)\leq c_2$. 
On the other hand, $|\Phi_{s}(\zxz{a}{b}{0}{d}g)|=|\chi_{1,s}(a)||\chi_{2,s}(d)||\Phi_{s}(g)|$. It can be considered as induced from two unramified characters $|\chi_{i,s}|$ for $i=1,2$. 

So the integral $\int\limits_{\GL_{2}(\F )}|f(g,\alpha \det (g)^{-1})||\Phi_{s}(\gamma_0 g)|dgd^*\alpha$ can be bounded by a finite sum of integrals, each of them having the form 
\begin{equation*}
N\int\limits_{g\in \GL_2(\F ), v(\alpha)-v(\det g)=c}f_0(g)|\Phi_{s}(\gamma_0 g)|dg,
\end{equation*}
where $c_1\leq c\leq c_2$. Then by an easy change of variable, the integral above will be just (\ref{inertI}) up to a constant multiple and a shift in $v(\alpha)$. So their asymptotic behaviors are the same.

From (\ref{inertI}) and the discussion above, we know now $\int\limits_{\GL_{2}(\F )}|f(g,\alpha \det (g)^{-1})||\Phi_{s}(\gamma_0 g)|dgd^*\alpha=O(|\chi_{1,s}(\alpha)||\alpha|^{-1}+|\chi_1\chi_2(\alpha)|^{1/2})$. The condition $\Re(s)\geq(s(\chi_2)-s(\chi_1))/4$ implies  $|\chi_{1,s}(\alpha)||\alpha|^{-1}=O(|\chi_1\chi_2(\alpha)|^{1/2})$ when $|\alpha|\rightarrow 0$. Note we have used $|\cdot|_\E=|\cdot|^2$ for an inert place here.

Now $J(s,w)$ is bounded up to a constant by
\begin{align*}&\text{\ \ \ }\int\limits_{|\alpha|<1}|\alpha|^{\nu}|\mu_1^{-1}\mu_2^{-1}(\alpha)|^{1/2}|\alpha|^{\Re(w)/2-1/4}|\chi_1\chi_2(\alpha)|^{-1}|\chi_1\chi_2(\alpha)|^{1/2}d^*\alpha\\
&=\int\limits_{|\alpha|<1}|\alpha|^{\nu+\Re(w)/2-1/4}d^*\alpha
\end{align*}
near $|\alpha|= 0$. Thus we get our definition for $D$ in the proposition and actually we just pick $\epsilon=2\nu$. The uniform convergence and holomorphicity are then clear.
\end{proof}

\subsection{Split places}
Now we consider the case when $v$ splits into two places $v_1$ and $v_2$ of $\E$.  For simplicity we assume that 2 is a unit, or equivalently $2\nmid v$. $D$ is now a square in $\F $. Fix one of its square roots and denote it by $\sqrt{D}$ and call the other one $-\sqrt{D}$. We write $\Phi_{s}=\Phi_{s}^{(1)}\cdotp\Phi_{s}^{(2)}$, where  $\Phi_{s}^{(i)}(\zxz{a_1}{m}{0}{a_2}g)=\chi^{(i)}_{1,s}(a_1)\chi^{(i)}_{2,s}(a_2)\Phi_{s}^{(i)}(g)$ and $\Phi_{s}^{(i)}(1)=1$. Here $\chi^{(i)}_{1,s}=\chi^{(i)}_1|\cdotp|_{\E_{v_i}}^{s+1/2}$, $\chi^{(i)}_{2,s}=\chi^{(i)}_2|\cdotp|_{\E_{v_i}}^{-(s+1/2)}$ like before. 
$\gamma_0$ should be understood as $\zxz{1}{0}{(\sqrt{D},-\sqrt{D})}{1}$. In the same language, $\Omega(t)=\chi_1(a-b\sqrt{D})\chi_2(a+b\sqrt{D})$ should be understood as $\chi^{(1)}_1(a-b\sqrt{D})\chi^{(2)}_1(a+b\sqrt{D})\chi^{(1)}_2(a+b\sqrt{D})\chi^{(2)}_2(a-b\sqrt{D})=\chi^{(1)}_1\chi^{(2)}_2(a-b\sqrt{D})\chi^{(1)}_2\chi^{(2)}_1(a+b\sqrt{D})$.

We start with $I(\alpha)=\int\limits_{\GL_{2}(\F )}f(g,\alpha \det (g)^{-1})\Phi_{s}(\gamma_0 g)dg$ for $f=char(M_{2}(O_F))\times char(O_F^{*})$. Recall $n=v(a_1)$, $k=v(m)$, $l=v(a_2)$.
Note $f$ is left invariant by $\zxz{1}{0}{\sqrt{D}}{1}$. By substituting $\zxz{1}{0}{-\sqrt{D}}{1} g\mapsto g$ we get
\begin{align}\label{splitpreI}
I(\alpha)&=\int\limits_{\GL_{2}(\F )}f(g,\alpha \det (g)^{-1})\Phi_{s}^{(1)}(\zxz{1}{0}{2\sqrt{D}}{1} g)\Phi_{s}^{(2)}(g)dg\\
&=\int\limits_{0\leq n\leq v(\alpha)}\int\limits_{k\geq 0}\int\limits_{l=v(\alpha)-n}\Phi_{s}^{(1)}(\zxz{1}{0}{2\sqrt{D}}{1} \zxz{a_1}{m}{0}{a_2})\Phi_{s}^{(2)}(\zxz{a_1}{m}{0}{a_2})d^*a_2|a_1|^{-1}dmd^*a_1 \notag\\
&=\int\limits_{0\leq n\leq v(\alpha)}\int\limits_{k\geq 0}\int\limits_{l=v(\alpha)-n}\Phi_{s}^{(1)}(\zxz{a_1}{m}{2a_1\sqrt{D}}{a_2+2m\sqrt{D}})\chi^{(2)}_{1,s}(a_1)\chi^{(2)}_{2,s}(a_2)d^*a_2|a_1|^{-1}dmd^*a_1.\notag
\end{align}
Then we can apply Lemma \ref{localcases} for $\Phi_{s}^{(1)}(\zxz{a_1}{m}{2a_1\sqrt{D}}{a_2+2m\sqrt{D}})$, as $2\sqrt{D}$ is still a unit. One can further expect Lemma \ref{localPhi} to hold mostly, with one important difference: when $\frac{v(\alpha)}{2}< n\leq v(\alpha)$ and $k=l<n$, we would expect $v(a_2+2m\sqrt{D})=\min\{v(a_2),v(m)\}$ for the inert case. So Lemma \ref{localPhi} suggests that $\Phi_{s}^{(1)}(\zxz{a_1}{m}{2a_1\sqrt{D}}{a_2+2m\sqrt{D}})=(\chi^{(1)}_{1,s})^n(\chi^{(1)}_{2,s})^{v(\alpha)-n}$. But $v(a_2+2m\sqrt{D})=\min\{v(a_2),v(m)\}$ is actually not true any more for the split case. When $v(a_2)=v(m)$ and $a_2\equiv -2m\sqrt{D}$, $v(a_2+2m\sqrt{D})$ will be larger than $v(a_2)$ or $v(m)$.

We introduce here a correction term $\Delta I$ for $I(\alpha)$: $I(\alpha)=I'+\Delta I$. Here $I'$ is the result we get if we follow Lemma \ref{localPhi} completely. As an analogue of the first expression of  (\ref{inertI}), we have (skipping some tedious steps):
\begin{align}
I'=&(\chi^{(1)}_{1,s}\chi^{(2)}_{2,s})^{v(\alpha)}
\frac{1-(\frac{q\chi^{(2)}_{1,s}}{\chi^{(2)}_{2,s}})^{v(\alpha)+1}}{1-\frac{q\chi^{(2)}_{1,s}}{\chi^{(2)}_{2,s}}}
\frac{1-q^{-1}}{1-\frac{\chi^{(1)}_{2,s}}{q\chi^{(1)}_{1,s}}}
+(\chi^{(1)}_{1,s}\chi^{(2)}_{2,s})^{v(\alpha)}
\frac{1-(\frac{\chi^{(2)}_{1,s}\chi^{(1)}_{2,s}}{\chi^{(1)}_{1,s}\chi^{(2)}_{2,s}})^{b+1}}{1-\frac{\chi^{(2)}_{1,s}\chi^{(1)}_{2,s}}{\chi^{(1)}_{1,s}\chi^{(2)}_{2,s}}}
\frac{q^{-1}(1-\frac{\chi^{(1)}_{2,s}}{\chi^{(1)}_{1,s}})}{1-\frac{\chi^{(1)}_{2,s}}{q\chi^{(1)}_{1,s}}}\\
&+(q^{-1}\chi^{(1)}_{2,s}\chi^{(2)}_{2,s})^{v(\alpha)}
\frac{(\frac{q^2\chi^{(1)}_{1,s}\chi^{(2)}_{1,s}}{\chi^{(1)}_{2,s}\chi^{(2)}_{2,s}})^{b+1}-(\frac{q^2\chi^{(1)}_{1,s}\chi^{(2)}_{1,s}}{\chi^{(1)}_{2,s}\chi^{(2)}_{2,s}})^{v(\alpha)+1}}{1-\frac{q^2\chi^{(1)}_{1,s}\chi^{(2)}_{1,s}}{\chi^{(1)}_{2,s}\chi^{(2)}_{2,s}}}
\frac{q^{-1}(1-\frac{\chi^{(1)}_{2,s}}{\chi^{(1)}_{1,s}})}{1-\frac{\chi^{(1)}_{2,s}}{q\chi^{(1)}_{1,s}}}\notag
\end{align}
for $v(\alpha)=2b,2b+1$. 
\newline
We give here a more detailed description of the correction term. Fix $n$ such that  $\frac{v(\alpha)}{2}< n\leq v(\alpha)$ and fix $m$ such that $k=v(m)=l$. Consider the integration in $d^*a_2$ for (\ref{splitpreI}), that is,
\begin{equation}\label{eq4}
\int\limits_{v(a_2)=v(\alpha)-n}\Phi_{s}^{(1)}(\zxz{a_1}{m}{2a_1\sqrt{D}}{a_2+2m\sqrt{D}})\chi^{(2)}_{1,s}(a_1)\chi^{(2)}_{2,s}(a_2)d^*a_2.
\end{equation}
For the value of $\Phi_{s}^{(1)}(\zxz{a_1}{m}{2a_1\sqrt{D}}{a_2+2m\sqrt{D}})$, there is a $\frac{q-2}{q-1}$ chance that $v(a_2+2m\sqrt{D})=k$; $\frac{1}{q-1}\frac{q-1}{q}=\frac{1}{q}$ chance that $v(a_2+2m\sqrt{D})=k+1$; $\frac{1}{q-1}\frac{1}{q}\frac{q-1}{q}=\frac{1}{q^2}$ chance that $v(a_2+2m\sqrt{D})=k+2$, etc,. But once $v(a_2+2m\sqrt{D})\geq n$, the value of $\Phi_{s}^{(1)}$ will just remain to be $(\chi^{(1)}_{1,s})^{v(\alpha)-n}(\chi^{(1)}_{2,s})^{n}$. The chance of this case is 
$$[\frac{1}{q^{2n-v(\alpha)}}+\frac{1}{q^{2n-v(\alpha)+1}}+\cdots]=\frac{q}{q^{2n-v(\alpha)}(q-1)}.$$
Then the integral (\ref{eq4}) becomes
\begin{equation*}(\chi^{(1)}_{1,s}\chi^{(2)}_{1,s})^n(\chi^{(1)}_{2,s}\chi^{(2)}_{2,s})^{v(\alpha)-n}\{\frac{q-2}{q-1}+\frac{1}{q}\frac{\chi^{(1)}_{2,s}}{\chi^{(1)}_{1,s}}+\frac{1}{q^2}(\frac{\chi^{(1)}_{2,s}}{\chi^{(1)}_{1,s}})^2+\cdots+\frac{q}{q^{2n-v(\alpha)}(q-1)}(\frac{\chi^{(1)}_{2,s}}{\chi^{(1)}_{1,s}})^{2n-v(\alpha)}  \}. \end{equation*}
Comparing with the supposed value $(\chi^{(1)}_{1,s}\chi^{(2)}_{1,s})^n(\chi^{(1)}_{2,s}\chi^{(2)}_{2,s})^{v(\alpha)-n}$, we get the correction
\begin{align*}
& (\chi^{(1)}_{1,s}\chi^{(2)}_{1,s})^n(\chi^{(1)}_{2,s}\chi^{(2)}_{2,s})^{v(\alpha)-n}\{
-\frac{1}{q-1}+\frac{1}{q}\frac{\chi^{(1)}_{2,s}}{\chi^{(1)}_{1,s}}+\frac{1}{q^2}(\frac{\chi^{(1)}_{2,s}}{\chi^{(1)}_{1,s}})^2+\cdots+\frac{q}{q-1}(\frac{\chi^{(1)}_{2,s}}{q\chi^{(1)}_{1,s}})^{2n-v(\alpha)}\}\\
&=-(\chi^{(1)}_{1,s}\chi^{(2)}_{1,s})^n(\chi^{(1)}_{2,s}\chi^{(2)}_{2,s})^{v(\alpha)-n}\frac{1-\frac{\chi^{(1)}_{2,s}}{\chi^{(1)}_{1,s}}}{1-\frac{\chi^{(1)}_{2,s}}{q\chi^{(1)}_{1,s}}}\frac{1}{q-1}(1-(\frac{\chi^{(1)}_{2,s}}{q\chi^{(1)}_{1,s}})^{2n-v(\alpha)}).
\end{align*}
Integrating this in $dm$ gives a factor $q^{n-v(\alpha)}(1-q^{-1})$ (recall $k=l=v(\alpha)-n$). Integrating further against $|a_1|^{-1}d^*a_1$ for all possible $v(a_1)$ gives a sum
\begin{align}\Delta I=-\sum\limits_{\frac{v(\alpha)}{2}<n\leq v(\alpha)}(q^{-1}\chi^{(1)}_{2,s}\chi^{(2)}_{2,s})^{v(\alpha)}(\frac{q^2\chi^{(1)}_{1,s}\chi^{(2)}_{1,s}}{\chi^{(1)}_{2,s}\chi^{(2)}_{2,s}})^n\frac{q^{-1}(1-\frac{\chi^{(1)}_{2,s}}{\chi^{(1)}_{1,s}})}{1-\frac{\chi^{(1)}_{2,s}}{q\chi^{(1)}_{1,s}}}(1-(\frac{\chi^{(1)}_{2,s}}{q\chi^{(1)}_{1,s}})^{2n-v(\alpha)}).\end{align}
After simplification we will get
\begin{align}
 \Delta I=&-(q^{-1}\chi^{(1)}_{2,s}\chi^{(2)}_{2,s})^{v(\alpha)}
\frac{(\frac{q^2\chi^{(1)}_{1,s}\chi^{(2)}_{1,s}}{\chi^{(1)}_{2,s}\chi^{(2)}_{2,s}})^{b+1}-(\frac{q^2\chi^{(1)}_{1,s}\chi^{(2)}_{1,s}}{\chi^{(1)}_{2,s}\chi^{(2)}_{2,s}})^{v(\alpha)+1}}{1-\frac{q^2\chi^{(1)}_{1,s}\chi^{(2)}_{1,s}}{\chi^{(1)}_{2,s}\chi^{(2)}_{2,s}}}
\frac{q^{-1}(1-\frac{\chi^{(1)}_{2,s}}{\chi^{(1)}_{1,s}})}{1-\frac{\chi^{(1)}_{2,s}}{q\chi^{(1)}_{1,s}}}
\\
&+(\chi^{(1)}_{1,s}\chi^{(2)}_{2,s})^{v(\alpha)}
\frac{(\frac{\chi^{(2)}_{1,s}\chi^{(1)}_{2,s}}{\chi^{(1)}_{1,s}\chi^{(2)}_{2,s}})^{b+1}-(\frac{\chi^{(2)}_{1,s}\chi^{(1)}_{2,s}}{\chi^{(1)}_{1,s}\chi^{(2)}_{2,s}})^{v(\alpha)+1}}{1-\frac{\chi^{(2)}_{1,s}\chi^{(1)}_{2,s}}{\chi^{(1)}_{1,s}\chi^{(2)}_{2,s}}}
\frac{q^{-1}(1-\frac{\chi^{(1)}_{2,s}}{\chi^{(1)}_{1,s}})}{1-\frac{\chi^{(1)}_{2,s}}{q\chi^{(1)}_{1,s}}}
\notag
\end{align}
for $v(\alpha)=2b,2b+1$. Note that most terms cancel with those from $I'$, especially all the terms with $b$. So
\begin{equation}\label{splitI}
I=(\chi^{(1)}_{1,s}\chi^{(2)}_{2,s})^{v(\alpha)}
\frac{1-(\frac{q\chi^{(2)}_{1,s}}{\chi^{(2)}_{2,s}})^{v(\alpha)+1}}{1-\frac{q\chi^{(2)}_{1,s}}{\chi^{(2)}_{2,s}}}
\frac{\frac{\chi^{(1)}_{1,s}}{\chi^{(1)}_{2,s}}(1-q)}{1-q\frac{\chi^{(1)}_{1,s}}{\chi^{(1)}_{2,s}}}
+(\chi^{(1)}_{1,s}\chi^{(2)}_{2,s})^{v(\alpha)}
\frac{1-(\frac{\chi^{(2)}_{1,s}\chi^{(1)}_{2,s}}{\chi^{(1)}_{1,s}\chi^{(2)}_{2,s}})^{v(\alpha)+1}}{1-\frac{\chi^{(2)}_{1,s}\chi^{(1)}_{2,s}}{\chi^{(1)}_{1,s}\chi^{(2)}_{2,s}}}
\frac{q^{-1}(1-\frac{\chi^{(1)}_{2,s}}{\chi^{(1)}_{1,s}})}{1-\frac{\chi^{(1)}_{2,s}}{q\chi^{(1)}_{1,s}}}.
\end{equation}

The rest story will be the same as the inert case, so we will skip some steps and show the results directly. By the assumption on the central character, we have $\mu_1\mu_2\chi^{(1)}_{1,s}\chi^{(1)}_{2,s}\chi^{(2)}_{1,s}\chi^{(2)}_{2,s}=1$. Recall $\delta=q^{-(\frac{w}{2}+\frac{1}{4})}$. Then
\begin{equation}
\sum\limits_{v(\alpha)=0}^{\infty}(\delta\mu_i)^{v(\alpha)}I(\alpha)=\frac{1-\delta\chi^{(1)}_{1,s}\chi^{(2)}_{1,s}\mu_i}{(1-\delta\chi^{(1)}_{2,s}\chi^{(2)}_{1,s}\mu_i)(1-\delta\chi^{(1)}_{1,s}\chi^{(2)}_{2,s}\mu_i)(1-q\delta\chi^{(1)}_{1,s}\chi^{(2)}_{1,s}\mu_i)}.
\end{equation}
So by (\ref{localmu's}) we have
\begin{align*}
\P (s,w)=&\frac{1-\delta^2-(\mu_1+\mu_2)\chi^{(1)}_{1,s}\chi^{(2)}_{1,s}\delta(1-q\delta^2)+(1-q)(\frac{\chi^{(1)}_{1,s}}{\chi^{(1)}_{2,s}}+\frac{\chi^{(2)}_{1,s}}{\chi^{(2)}_{2,s}})\delta^2+\frac{q\chi^{(1)}_{1,s}\chi^{(2)}_{1,s}}{\chi^{(1)}_{2,s}\chi^{(2)}_{2,s}}\delta^2(1-\delta^2)}
{(1-\mu_1\chi^{(1)}_{2,s}\chi^{(2)}_{1,s}\delta)(1-\mu_1\chi^{(1)}_{1,s}\chi^{(2)}_{2,s}\delta)(1-\mu_2\chi^{(1)}_{2,s}\chi^{(2)}_{1,s}\delta)(1-\mu_2\chi^{(1)}_{1,s}\chi^{(2)}_{2,s}\delta)}\\
&\times \frac{1}{(1-q\mu_1\chi^{(1)}_{1,s}\chi^{(2)}_{1,s}\delta)(1-q\mu_2\chi^{(1)}_{1,s}\chi^{(2)}_{1,s}\delta)}.
\end{align*}
If we evaluate at $w=1/2$, i.e. $\delta=q^{-1/2}$, then we can simplify the numerator a lot:
\begin{align*}
\P (s,1/2)=&\frac{(1-q^{-1})}
{(1-\mu_1\chi^{(1)}_{2,s}\chi^{(2)}_{1,s}q^{-1/2})(1-\mu_1\chi^{(1)}_{1,s}\chi^{(2)}_{2,s}q^{-1/2})(1-\mu_2\chi^{(1)}_{2,s}\chi^{(2)}_{1,s}q^{-1/2})(1-\mu_2\chi^{(1)}_{1,s}\chi^{(2)}_{2,s}q^{-1/2})}\\
&\times \frac{(1-\frac{\chi^{(1)}_{1,s}}{\chi^{(1)}_{2,s}})(1-\frac{\chi^{(2)}_{1,s}}{\chi^{(2)}_{2,s}})}{(1-\mu_1\chi^{(1)}_{1,s}\chi^{(2)}_{1,s}q^{1/2})(1-\mu_2\chi^{(1)}_{1,s}\chi^{(2)}_{1,s}q^{1/2})}.
\end{align*}
Recall in the beginning of this subsection we have rewritten $\Omega=\chi^{(1)}_1\chi^{(2)}_2(a-b\sqrt{D})\chi^{(1)}_2\chi^{(2)}_1(a+b\sqrt{D})$. Define 
$s(\Omega)=s(\chi^{(1)}_{1,s}\chi^{(2)}_{2,s})-s(\chi^{(1)}_{2,s}\chi^{(2)}_{1,s})=s(\chi^{(1)}_{1}\chi^{(2)}_{2})-s(\chi^{(1)}_{2}\chi^{(2)}_{1})$. This is independent of $s$.

\begin{prop}\label{propsplit}
Let $v$ be a non-archimedean split place for $\E/\F$ and $\delta=q^{-(\frac{w}{2}+\frac{1}{4})}$. Suppose $\Re(s)>(s(\chi^{(1)}_{2}\chi^{(2)}_{2})-s(\chi^{(1)}_{1}\chi^{(2)}_{1}))/4$, 
\begin{enumerate}
\item[(i)]There exists an $\epsilon' >0$ such that, in $D'=\{w\in\C;\Re (w)>1/2+|s(\Omega)|-\epsilon' \}$ the integral $\P (s,w)$ converges and uniformly in any compact subset of $D'$. It's holomorphic in $D'$. 

\item[(ii)]For an unramified place we have:
\begin{align*}
\P (s,w)=&\frac{1-\delta^2-(\mu_1+\mu_2)\chi^{(1)}_{1,s}\chi^{(2)}_{1,s}\delta(1-q\delta^2)+(1-q)(\frac{\chi^{(1)}_{1,s}}{\chi^{(1)}_{2,s}}+\frac{\chi^{(2)}_{1,s}}{\chi^{(2)}_{2,s}})\delta^2+\frac{q\chi^{(1)}_{1,s}\chi^{(2)}_{1,s}}{\chi^{(1)}_{2,s}\chi^{(2)}_{2,s}}\delta^2(1-\delta^2)}
{(1-\mu_1\chi^{(1)}_{2,s}\chi^{(2)}_{1,s}\delta)(1-\mu_1\chi^{(1)}_{1,s}\chi^{(2)}_{2,s}\delta)(1-\mu_2\chi^{(1)}_{2,s}\chi^{(2)}_{1,s}\delta)(1-\mu_2\chi^{(1)}_{1,s}\chi^{(2)}_{2,s}\delta)}\\
&\times \frac{1}{(1-q\mu_1\chi^{(1)}_{1,s}\chi^{(2)}_{1,s}\delta)(1-q\mu_2\chi^{(1)}_{1,s}\chi^{(2)}_{1,s}\delta)}.
\end{align*}
When $|s(\Omega)|$ is small enough, we can evaluate at $w=1/2$ and get much simpler result
\begin{align*}
\P (s,1/2)=&\frac{(1-q^{-1})}
{(1-\mu_1\chi^{(1)}_{2,s}\chi^{(2)}_{1,s}q^{-1/2})(1-\mu_1\chi^{(1)}_{1,s}\chi^{(2)}_{2,s}q^{-1/2})(1-\mu_2\chi^{(1)}_{2,s}\chi^{(2)}_{1,s}q^{-1/2})(1-\mu_2\chi^{(1)}_{1,s}\chi^{(2)}_{2,s}q^{-1/2})}\\
&\times \frac{(1-\frac{\chi^{(1)}_{1,s}}{\chi^{(1)}_{2,s}})(1-\frac{\chi^{(2)}_{1,s}}{\chi^{(2)}_{2,s}})}{(1-\mu_1\chi^{(1)}_{1,s}\chi^{(2)}_{1,s}q^{1/2})(1-\mu_2\chi^{(1)}_{1,s}\chi^{(2)}_{1,s}q^{1/2})}.
\end{align*}
Writing out variable $s$ explicitly, we get
\begin{align}
\label{formulasplit}
\P (s,1/2)&=\frac{(1-q^{-1})}
{(1-\mu_1\chi^{(1)}_{2}\chi^{(2)}_{1}q^{-1/2})(1-\mu_1\chi^{(1)}_{1}\chi^{(2)}_{2}q^{-1/2})(1-\mu_2\chi^{(1)}_{2}\chi^{(2)}_{1}q^{-1/2})(1-\mu_2\chi^{(1)}_{1}\chi^{(2)}_{2}q^{-1/2})}\\
&\text{\ \ }\times \frac{(1-\frac{\chi^{(1)}_{1}}{\chi^{(1)}_{2}}q^{-(2s+1)})(1-\frac{\chi^{(2)}_{1}}{\chi^{(2)}_{2}}q^{-(2s+1)})}{(1-\mu_1\chi^{(1)}_{1}\chi^{(2)}_{1}q^{-(2s+1/2)})(1-\mu_2\chi^{(1)}_{1}\chi^{(2)}_{1}q^{-(2s+1/2)})}\notag\\
&=\frac{L (\Pi\otimes\Omega,1/2)L (\pi\otimes\chi_{1}|_{\F ^*},2s+1/2)}{L (\eta,1)L^{\E} (\chi,2s+1)}.\notag
\end{align}
Recall $\chi=\frac{\chi_1}{\chi_2}$. $L^{\E} $ means it's a product of L factors over all places of $\E$ above $v$. In the split case there are two places over $v$, thus two factors.
\end{enumerate}
\end{prop}
\begin{proof} We can get the formulae directly from previous calculations. For the convergence, one can imitate the proof of Proposition \ref{propinert}. In particular the expression for $I$ in (\ref{splitI}) would explain why we need $\Re(s)>(s(\chi^{(1)}_{2}\chi^{(2)}_{2})-s(\chi^{(1)}_{1}\chi^{(2)}_{1}))/4$ and $\Re (w)>1/2+|s(\Omega)|-\epsilon'$.
\end{proof}

\section{Local calculation for other non-archimedean places}
In this section, we will compute the local integral $\P(s,w,f,\Phi_s)$ for ramified non-archimedean places. We will still omit subscript $v$ and everything should be understood locally in this section. Recall
\begin{equation}
\P (s,w,f,\Phi_s)=\int\limits_{ZN\backslash \GL_{2}(\F )}W(\sigma)\Delta(\sigma)^{w-1/2}\int\limits_{\GL_{2}(\F )}r'(\sigma)f(g,\det (g)^{-1})\Phi_{s}(\gamma_0 g) dgd\sigma.
\end{equation}
The calculations in the last section showed that 
$$\P(s,1/2,f,\Phi_s)=\frac{L (\Pi\otimes\Omega,1/2)L (\pi\otimes\chi_{1}|_{\F ^*},2s+1/2)}{L (\eta,1)L^{\E} (\chi,2s+1)}$$
for unramified places. We also expect similar $L-$factors for ramified places. So we normalize the local integral and denote the following by $\P^0(s,1/2,f,\Phi_s)$
$$\frac{L (\eta,1)L^{\E} (\chi,2s+1)}{L (\Pi\otimes\Omega,1/2)L (\pi\otimes\chi_{1}|_{\F ^*},2s+1/2)}\int\limits_{ZN\backslash \GL_{2}(\F )}W(\sigma)\Delta(\sigma)^{w-1/2}\int\limits_{\GL_{2}(\F )}r'(\sigma)f(g,\det (g)^{-1})\Phi_{s}(\gamma_0 g) dgd\sigma|_{w=1/2}.$$
In this section we will compute the above expression for some ramified placed with specific choices of $f$ and $\Phi_s$. The choice for $\Phi_s$ will be normalized in the sense that if $\Phi_s$ is $K_1(\varpi^c)-$invariant and supported on $B\zxz{1}{0}{\varpi^i}{1}K_1(\varpi^c)$, then $\Phi_s(\zxz{1}{0}{\varpi^i}{1})=1$.

Before we start, we first list the expected $L-$factors $$\frac{L (\Pi\otimes\Omega,1/2)L (\pi\otimes\chi_{1}|_{\F ^*},2s+1/2)}{L (\eta,1)L^{\E} (\chi,2s+1)}$$ for the following ramified cases we are going to consider in this section:
\newline

\begin{tabular}{|c|p{2.1cm}|c|c|c|c|}
\hline
Case	&$\pi_v$	&$\chi_1$ and $\chi_2$	&$\E_v/\F_v$	
&Expected L-factors\\ \hline
1	&unramified	&unramified	&ramified	&$\frac{1-\chi(\varpi_\E)q^{-(2s+1)}}
{(1-(\mu^2_{1}\chi_1\chi_2)(\varpi_E)q^{-\frac{1}{2}})(1-(\mu^2_{2}\chi_1\chi_2)(\varpi_E)q^{-\frac{1}{2}})(1-\mu_1\chi_{1}q^{-(2s+1/2)})(1-\mu_2\chi_{1}q^{-(2s+1/2)})}$	\\ \hline
2	&unramified special	&unramified	&split	&$\frac{(1-q^{-1})(1-\chi^{(1)}q^{-(2s+1)})(1-\chi^{(2)}q^{-(2s+1)})}
{(1-\mu_2\chi^{(1)}_{2}\chi^{(2)}_{1}q^{-1/2})(1-\mu_2\chi^{(1)}_{1}\chi^{(2)}_{2}q^{-1/2})(1-\mu_2\chi^{(1)}_{1}\chi^{(2)}_{1}q^{-(2s+1/2)})}$\\ \hline
3	&supercupidal or ramified principal	&unramified	&split	&$(1-q^{-1})(1-\chi^{(1)}q^{-(2s+1)})(1-\chi^{(2)}q^{-(2s+1)})$\\ \hline
4	&unramified	&$\chi_{1}$ level c	&inert	&$\frac{1+q^{-1}}
{(1-\mu_1\chi_{1}(\varpi)q^{-(2s+1/2)})(1-\mu_2\chi_{1}(\varpi)q^{-(2s+1/2)})}$\\ \hline
5	&$\mu_2$ level c	&$\chi_1$ level c	&inert	& $\frac{1+q^{-1}}{1-\mu_2\chi_{1}q^{-(2s+1/2)}}$\\ \hline
\end{tabular}
\newline
\newline

Recall $\chi=\frac{\chi_1}{\chi_2}$. For case 1, $\varpi_E$ is a fixed uniformizer for $\E $ such that $\varpi_E^2=\varpi$. We also assume $v(\sqrt{D})=1/2$.

The characters not mentioned (that is, $\mu_1$ and $\chi_2$) in cases 4 and 5 are all unramified. This implies that $\chi_{1}|_{\F^*}$ is unramified in case 4 and is of level $c$ in case 5. 

\subsection{$\E /\F $ ramified}
Here we consider the case when $\pi$ and $\Phi_{s}$ are both unramified, but $\E /\F $ is a ramified extension. We still let $v(\varpi)=1$ and write $\mu$ or $\chi$ in short for $\mu(\varpi)$ or $\chi(\varpi)$. For simplicity we suppose $v(\sqrt{D})=\frac{1}{2}$. 
 We will prove in this subsection the following result

\begin{prop}\label{prop6-1}
Suppose that  $\pi$ and $\Phi_{s}$ are both unramified at $v$, $\E /\F $ is ramified with $v(\sqrt{D})=1/2$. We pick $f$ and $\Phi_s$ as in the last section. Then 
\begin{equation}
\P (s,w,f,\Phi_s)=\frac{(1+(\mu_1^2\chi_{1,s}\chi_{2,s})(\varpi_E)\delta+(\mu_2^2\chi_{1,s}\chi_{2,s})(\varpi_E)\delta+\delta^2)(1-q\delta^2\frac{\chi_{1,s}}{\chi_{2,s}}(\varpi_E))}
{(1-q\mu_1\chi_{1,s}\delta)(1-q\mu_2\chi_{1,s}\delta)(1-\mu_1^2\chi_{1,s}\chi_{2,s}\delta^2)(1-\mu_2^2\chi_{1,s}\chi_{2,s}\delta^2)}
\end{equation}
for $\delta=q^{-(\frac{w}{2}+\frac{1}{4})}$. When $w=\frac{1}{2}$, 
\begin{align}\label{formulaEram}
\P (s,\frac{1}{2},f,\Phi_s)&=\frac{(1+(\mu^2_{1}\chi_1\chi_2)(\varpi_E)q^{-\frac{1}{2}})(1+(\mu^2_{2}\chi_1\chi_2)(\varpi_E)q^{-\frac{1}{2}})}
{(1-\mu_1^2\chi_{1}\chi_{2}q^{-1})(1-\mu_2^2\chi_{1}\chi_{2}q^{-1})}
\times\frac{1-\frac{\chi_{1}}{\chi_{2}}(\varpi_E)q^{-(2s+1)}}
{(1-\mu_1\chi_{1}q^{-(2s+1/2)})(1-\mu_2\chi_{1}q^{-(2s+1/2)})}
\end{align}
is just as expected. Thus $\P^0(s,1/2,f,\Phi_s)=1$.
\end{prop}

For the sake of comparison with the unramified case, we write $\sqrt{\chi_{i,s}}$ or $\chi_{i,s}^{\frac{1}{2}}$ to mean $\chi_{i,s}(\varpi_E)$.
As in the inert case, we can start with equation (\ref{local1}) and (\ref{localI}). Lemma \ref{localcases} still holds. Then we have the following lemma as an analogue of Lemma \ref{localPhi}:
\begin{lem}\label{EramPhi}
\begin{enumerate}
\item[(1)]If $0\leq n < \frac{v(\alpha)}{2}$, then $l> n$.
\begin{enumerate}
\item[(1i)] If $k\geq n$, $\Phi_s(\zxz{a_1}{m}{a_1\sqrt{D}}{a_2+m\sqrt{D}})=\chi_{1,s}(\frac{a_2}{\sqrt{D}})\chi_{2,s}(a_1\sqrt{D})=\chi_{1,s}^{v(\alpha)-n-\frac{1}{2}}\chi_{2,s}^{n+\frac{1}{2}}$.
\item[(1ii)] If $0\leq k< n$, $\Phi_s(\zxz{a_1}{m}{a_1\sqrt{D}}{a_2+m\sqrt{D}})=\chi_{1,s}(\frac{a_1a_2}{a_2+m\sqrt{D}})\chi_{2,s}(a_2+m\sqrt{D})=\chi_{1,s}^{v(\alpha)-k-\frac{1}{2}}\chi_{2,s}^{k+\frac{1}{2}}$.
\end{enumerate}
\item[(2)]If $\frac{v(\alpha)}{2}\leq n\leq v(\alpha)$, then $l\leq n$.
\begin{enumerate}
\item[(2i)] If $k\geq l$, $\Phi_s(\zxz{a_1}{m}{a_1\sqrt{D}}{a_2+m\sqrt{D}})=\chi_{1,s}^n\chi_{2,s}^{v(\alpha)-n}$.
\item[(2ii)] If $k<l$, $\Phi_s(\zxz{a_1}{m}{a_1\sqrt{D}}{a_2+m\sqrt{D}})=\chi_{1,s}^{v(\alpha)-k-\frac{1}{2}}\chi_{2,s}^{k+\frac{1}{2}}$.
\end{enumerate}
\end{enumerate}

\end{lem}
\begin{proof}
One just need to use Lemma \ref{localcases}, and note that $v(a_2+m\sqrt{D})=Min\{l,k+\frac{1}{2}\}$, $v(a_1\sqrt{D})=n+\frac{1}{2}$. Then the result is clear.
\end{proof}
Now to compute $I(\alpha,f,\Phi_s)$ for this case, we compare Lemma \ref{EramPhi} to Lemma \ref{localPhi}. We see in all cases the values of $\Phi_{s}$ differ by $\sqrt{\frac{\chi_{2,s}}{\chi_{1,s}}}$ except the case $\frac{v(\alpha)}{2}\leq n\leq v(\alpha), k\geq l$. Denote by $I$ the formula (\ref{inertI}) for the inert case, and we get the relation
\begin{equation}
I(\alpha,f,\Phi_s)=\sqrt{\frac{\chi_{2,s}}{\chi_{1,s}}}[I+(\sqrt{\frac{\chi_{1,s}}{\chi_{2,s}}}-1)\sum\limits_{\frac{v(\alpha)}{2}\leq n\leq v(\alpha)}
\sum\limits_{v(\alpha)-n\leq k<\infty}\chi_{1,s}^{n}\chi_{2,s}^{v(\alpha)-n}q^{n-k}(1-q^{-1})].
\end{equation}
As a result,
\begin{equation}
I(\alpha,f,\Phi_s)=\sqrt{\frac{\chi_{2,s}}{\chi_{1,s}}}I+(1-\sqrt{\frac{\chi_{2,s}}{\chi_{1,s}}})\chi_{2,s}^{v(\alpha)}q^{-v(\alpha)}\frac{(q^2\frac{\chi_{1,s}}{\chi_{2,s}})^{v(\alpha)-m}-(q^2\frac{\chi_{1,s}}{\chi_{2,s}})^{v(\alpha)+1}}{1-q^2\frac{\chi_{1,s}}{\chi_{2,s}}}
\end{equation}
for $v(\alpha)=2b,2b+1$.
\newline
After skipping some tedious and unimportant steps, we get first
\begin{equation}
\sum\limits_{v(\alpha)=0}^{\infty}\delta_i^{v(\alpha)}I(\alpha,f,\Phi_s)(\alpha)=\frac{1+\delta_i\sqrt{\chi_{1,s}\chi_{2,s}}}{(1-\delta_i\chi_{1,s}q)(1-\delta_i^2\chi_{1,s}\chi_{2,s})}.
\end{equation}
And then
\begin{equation}
\P (s,w,f,\Phi_s)=\frac{(1+\mu_1\sqrt{\chi_{1,s}\chi_{2,s}}\delta+\mu_2\sqrt{\chi_{1,s}\chi_{2,s}}\delta+\delta^2)(1-q\delta^2\sqrt{\frac{\chi_{1,s}}{\chi_{2,s}}})}
{(1-\mu_1^2\chi_{1,s}\chi_{2,s}\delta^2)(1-\mu_2^2\chi_{1,s}\chi_{2,s}\delta^2)(1-q\mu_1\chi_{1,s}\delta)(1-q\mu_2\chi_{1,s}\delta)}
\end{equation}
for $\delta=q^{-(\frac{w}{2}+\frac{1}{4})}$. When $w=\frac{1}{2}$, $\delta=q^{-\frac{1}{2}}$. Recall that $\mu_1\mu_2\chi_1\chi_2=1$. Then we will get
\begin{align}\label{formulaEram'}
\P (s,\frac{1}{2},f,\Phi_s)&=\frac{(1+(\mu^2_{1}\chi_1\chi_2)(\varpi_E)q^{-\frac{1}{2}})(1+(\mu^2_{2}\chi_1\chi_2)(\varpi_E)q^{-\frac{1}{2}})}
{(1-\mu_1^2\chi_{1}\chi_{2}q^{-1})(1-\mu_2^2\chi_{1}\chi_{2}q^{-1})}
\times\frac{1-\frac{\chi_{1}}{\chi_{2}}(\varpi_E)q^{-(2s+1)}}
{(1-\mu_1\chi_{1}q^{-(2s+1/2)})(1-\mu_2\chi_{1}q^{-(2s+1/2)})}.
\end{align}

\subsection{unramified special representation}
In this subsection we consider the case when $\pi=\sigma(\mu_1,\mu_2)$ is an unramified special representation. Then $\varphi$ and its corresponding Whittaker function $W=W_\varphi^-$  belong to $\hat{\pi}=\sigma(\mu_1^{-1},\mu_2^{-1})$. 
Since $\sigma(\mu_1^{-1},\mu_2^{-1})$ is equivalent to $\sigma(\mu_2^{-1},\mu_1^{-1})$, we can assume without loss of generality that $\sigma(\mu_1^{-1},\mu_2^{-1})$ is an irreducible subrepresentation of $\pi(\mu_1^{-1},\mu_2^{-1})$. This means $\mu_1^{-1}\mu_2=|\cdotp|_{\F }$.
We shall do the computation in the split case, since the inert case would fail the Tunnell-Saito criterion as seen in Example \ref{ApplyTunnell}. Write $\Phi_s=\Phi_s^{(1)}\Phi_s^{(2)}$ as before.

\begin{prop}\label{prop6-2}
Suppose that  $\chi_1$ and $\chi_2$ are unramified, and $\E /\F $ is split. Suppose that  $\pi=\sigma(\mu_1,\mu_2)$ is an unramified special representation such that $\mu_1^{-1}\mu_2=|\cdotp|_{\F }$. Further assume 2 is a unit here. Pick
$$f=char(\zxz{1}{0}{\sqrt{D}}{1}\zxz{O_F}{O_F}{\varpi O_F}{O_F})\times char(O_F^*).$$
Pick $\Phi_s^{(2)}$ to be the unique right $K_1(\varpi)-$invariant function supported on $BK_1(\varpi)$, and $\Phi_s^{(1)}$ just to be the standard right $K-$invariant function. Then
\begin{align}
\P (s,w,f,\Phi_s)=\frac{1-q^{-1}}{(q+1)^2}\frac{1-\frac{\chi_{1,s}^{(1)}}{\chi_{2,s}^{(1)}}}
{(1-\delta\chi^{(1)}_{2,s}\chi^{(2)}_{1,s}\mu_2)(1-\delta\chi^{(1)}_{1,s}\chi^{(2)}_{2,s}\mu_2)(1-q\delta\chi^{(1)}_{1,s}\chi^{(2)}_{1,s}\mu_2)}.
\end{align}
At $w=1/2$, we have 
\begin{equation}\label{specialunramL}
\P (s,1/2,f,\Phi_s)=\frac{1}{(q+1)^2}\frac{(1-q^{-1})(1-\frac{\chi^{(1)}_{1}}{\chi^{(1)}_{2}}q^{-(2s+1)})}
{(1-\mu_2\chi^{(1)}_{2}\chi^{(2)}_{1}q^{-1/2})(1-\mu_2\chi^{(1)}_{1}\chi^{(2)}_{2}q^{-1/2})(1-\mu_2\chi^{(1)}_{1}\chi^{(2)}_{1}q^{-(2s+1/2)})}.
\end{equation}
The denominator of the expression is as expected, and
\begin{equation}
\P^0(s,1/2,f,\Phi_s)=\frac{1}{(q+1)^2(1-\chi^{(2)}q^{-(2s+1)})}.
\end{equation}
\end{prop}

We first work out the Whittaker function. It's a classical result that there is no $K-$invariant element in an unramified special representation: the $K-$invariant function $\varphi_0$ which is supported on $BK$ only lives in $\pi(\mu_1^{-1},\mu_2^{-1})$. But the $K_1(\varpi)-$invariant subspace of the unramified special representation is one dimensional. We can pick 
\begin{equation}\label{specialphi}
\varphi|_K=char(B(O_F)K_1(\varpi))-q^{-1}char(B(O_F)\zxz{1}{0}{1}{1}K_1(\varpi)).
\end{equation}
One can easily check that 
\begin{equation*} 
\int\limits_{K}\varphi_0\varphi dk=\int\limits_{B(O_F)K_1(\varpi)}dk+\int\limits_{B(O_F)\zxz{1}{0}{1}{1}K_1(\varpi)}-q^{-1}dk=\frac{1}{q+1}+\frac{q}{q+1}(-q^{-1})=0.
\end{equation*}
The corresponding $W $ is also $K_1(\varpi)-$invariant. Suppose that the chosen Schwartz function $f$ is also $K_1(\varpi)-$invariant under the Weil representation $r'$. Let $T_1$ denote the subgroup $\{\zxz{\alpha}{0}{0}{1}|\alpha\in \F ^*\}$. The integral over $ZN\backslash \GL_2$ can be decomposed into integrals over $T_1K_1(\varpi)$  and $ T_1\zxz{1}{0}{1}{1}K_1(\varpi)$. According to the Appendix A, we can rewrite formula (\ref{localtarget}) as 
\begin{align}\label{specialtarget}
&\P (s,w,f,\Phi_s)=
\frac{1}{q+1}\int\limits_{T_1}W (\zxz{\alpha}{0}{0}{1})|\alpha|^{\frac{w}{2}-\frac{1}{4}}\Phi_s^{-1}(\alpha)\int\limits_{\GL_{2}(\F )}f(g,\frac{\alpha}{\det (g)})\Phi_{s}(\gamma_0 g)dgd^*\alpha\\
&+\frac{q}{q+1}\int\limits_{T_1}W (\zxz{\alpha}{0}{0}{1}\zxz{1}{0}{1}{1})|\alpha|^{\frac{w}{2}-\frac{1}{4}}\Phi_s^{-1}(\alpha)\int\limits_{\GL_{2}(\F )}r'(\zxz{1}{0}{1}{1})f(g,\frac{\alpha}{\det (g)})\Phi_{s}(\gamma_0 g)dgd^*\alpha .\notag
\end{align}
We need to figure out the values of $W (\zxz{\alpha}{0}{0}{1})$ and $W (\zxz{\alpha}{0}{0}{1}\zxz{1}{0}{1}{1})$. It is possible to get these values directly from classical theories, but we will start with a more general setting, as this will be helpful for later cases. 
Recall $W $ is the Whittaker funciton associated to $\psi^-$. So 
$$W (g)=\int\limits_{m\in \F }\varphi(\omega\zxz{1}{m}{0}{1}g)\psi(m)dm.$$ 
The first step is to write $$\omega\zxz{1}{m}{0}{1}\zxz{\alpha}{0}{0}{1}\zxz{1}{0}{\varpi^j}{1}=\zxz{\varpi^j}{1}{-\alpha-m\varpi^j}{-m}$$ 
in form of $B(\F )\zxz{1}{0}{\varpi^i}{1}K_1(\varpi^c)$ for $0\leq i,j \leq c$. Note that if $i=c$, then $\zxz{1}{0}{\varpi^i}{1}$ is absorbed into $K_1(\varpi^c)$. Same for j.

\begin{lem}\label{LevelIwaDec}
\begin{enumerate}
\item[(1)]Suppose $i=0$.
\begin{enumerate}
\item[(1i)]If $j=0$, we need $m\notin \alpha(-1+\varpi O_F)$ for  $\zxz{\varpi^j}{1}{-\alpha-m\varpi^j}{-m}\in B\zxz{1}{0}{\varpi^i}{1}K_1(\varpi^c)$;
\item[(1ii)]If $j>0$, we need $v(m)\geq v(\alpha)$. 

\end{enumerate}
Under above conditions we can write $\zxz{\varpi^j}{1}{-\alpha-m\varpi^j}{-m}$ as
\begin{equation*}
\zxz{-\frac{\alpha}{\alpha+m\varpi^j}}{\varpi^j+\frac{\alpha}{\alpha+m\varpi^j}}{0}{-\alpha-m\varpi^j}\zxz{1}{0}{1}{1}\zxz{1}{-1+\frac{m}{\alpha+m\varpi^j}}{0}{1}.
\end{equation*}
\item[(2)]Suppose $i=c$.
\begin{enumerate}
\item[(2i)]If $j<c$, we need $m\in \alpha\varpi^{-j}(-1+\varpi^{c-j}O_F)$;
\item[(2ii)]If $j=c$, we need $v(m)\leq v(\alpha)-c$.
\end{enumerate}
Under above conditions, we can write $\zxz{\varpi^j}{1}{-\alpha-m\varpi^j}{-m}$ as
\begin{equation*}
\zxz{-\frac{\alpha}{m}}{1}{0}{-m}\zxz{1}{0}{\frac{\alpha}{m}+\varpi^j}{1}.
\end{equation*}

\item[(3)]Suppose $0<i<c$.
\begin{enumerate}
\item[(3i)]If $j<i$, we need $m\in \alpha\varpi^{-j}(-1+\varpi^{i-j}O_F^*)$;
\item[(3ii)]If $j>i$, we need $v(m)=v(\alpha)-i$;
\item[(3iii)]If $j=i$, we need $v(m)\leq v(\alpha)-i$ but $m\notin \alpha\varpi^{-i}(-1+\varpi O_F)$. 
\end{enumerate}
Under above conditions we can write $\zxz{\varpi^j}{1}{-\alpha-m\varpi^j}{-m}$ as
\begin{equation*}
\zxz{-\frac{\alpha\varpi^i}{\alpha+m\varpi^j}}{1}{0}{-m}\zxz{1}{0}{\varpi^i}{1}\zxz{\frac{\alpha+m\varpi^j}{m\varpi^i}}{0}{0}{1}.
\end{equation*}
\end{enumerate}
\end{lem}
\begin{proof}
To write $\zxz{\varpi^j}{1}{-\alpha-m\varpi^j}{-m}$ in form of $B(\F )\zxz{1}{0}{\varpi^i}{1}K_1(\varpi^c)$ is equivalent to find $K_1=\zxz{k_1}{k_2}{k_3}{k_4}=K_1'\zxz{1}{0}{-\varpi^i}{1}$ for $K_1'\in K_1(\varpi^c)$, such that $\zxz{\varpi^j}{1}{-\alpha-m\varpi^j}{-m}K_1$ is upper triangular. So we get $(-\alpha-m\varpi^j)k_1-mk_3=0$, or equivalently,
\begin{equation*} 
\frac{\alpha}{m}+\varpi^j+\frac{k_3}{k_1}=0.
\end{equation*}
The effect of $\zxz{1}{0}{-\varpi^i}{1}$ on $K_1$ is such that $\begin{cases}
k_1\in O_F, k_3\in O_F^*, &\text{\ if\ }i=0;\\k_1\in O_F^*,v(k_3)=i, &\text{\ if\ }0<i<c;\\k_1\in O_F^*, v(k_3)\geq c, &\text{\ if\ }i=c.
\end{cases}$ 
\newline
Then it's easy to check all the cases listed in the lemma. For example, if $i=0,j=0$, we can pick $$K_1=\zxz{\frac{m}{\alpha+m}}{\frac{\alpha}{\alpha+m}}{-1}{1}.$$ 
The requirement $m\notin \alpha(-1+\varpi O_F)$ will gurantee that $k_1=\frac{m}{\alpha+m}$ is an integer. We leave other cases to readers to check.
\end{proof}
\begin{cor}
Assume $\mu_1$ and $\mu_2$ are unramified and $\varphi\in \sigma(\mu_1^{-1},\mu_2^{-1})$ is given by (\ref{specialphi}).
Let $W$ be the normalized Whittaker function associated to $\varphi$. Then 
\begin{equation}\label{specialW1}
W (\zxz{\alpha}{0}{0}{1})=
\begin{cases}
\mu_1^{-v(\alpha)}q^{-v(\alpha)/2}, &\text{\ if\ }v(\alpha)\geq 0;\\
0,&\text{\ if\ }v(\alpha)<0,
\end{cases}
\end{equation}
and
\begin{equation}\label{specialW2}
W (\zxz{\alpha}{0}{0}{1}\zxz{1}{0}{1}{1})=\begin{cases}
-q^{-1}\mu_1^{-v(\alpha)}q^{-v(\alpha)/2}\psi(-\alpha), &\text{\ if\ }v(\alpha)\geq -1;\\
0,&\text{\ if\ }v(\alpha)<-1.
\end{cases}
\end{equation}
\end{cor}

\begin{proof}
 Put $c=1$, and consider $j=1$. By formula (\ref{specialphi}) and (1ii) (2ii) of the lemma:
\begin{align}\label{specialpreW1}
&\text{\ \ \ }W (\zxz{\alpha}{0}{0}{1})=\int\limits_{m\in \F }\varphi(\omega\zxz{1}{m}{0}{1}\zxz{\alpha}{0}{0}{1})\psi(m) dm\\
&=\int\limits_{v(m)\leq v(\alpha)-1}\mu_1^{-1}(-\frac{\alpha}{m})\mu_2^{-1}(-m)|\frac{\alpha}{m^2}|^{1/2}\psi(m)dm-q^{-1}\int\limits_{v(m)\geq v(\alpha)}\mu_1^{-1}(-1)\mu_2^{-1}(-\alpha)|\frac{1}{\alpha}|^{1/2}\psi(m)dm\notag\\
&=\begin{cases}
(-q^{-1}-q^{-2})\mu_1^{-v(\alpha)}q^{-v(\alpha)/2}, \text{\ if\ }v(\alpha)\geq 0;\\
0,\text{\ if\ }v(\alpha)<0.
\end{cases}\notag
\end{align}
In the last equation, we have used $\mu_1^{-1}\mu_2=|\cdotp|$, and $$\int\limits_{v(m)=j}\psi(m)dm=\begin{cases}
0,&\text{\ if\ }j<-1;\\-1,&\text{\ if\ }j=-1;\\q^{-j}(1-q^{-1}),&\text{\ if\ }j\geq 0.
\end{cases}$$ 
If we normalize $W (\zxz{1}{0}{0}{1})$ to be 1, then 
\begin{equation}\label{specialW1'}
W (\zxz{\alpha}{0}{0}{1})=
\begin{cases}
\mu_1^{-v(\alpha)}q^{-v(\alpha)/2}, &\text{\ if\ }v(\alpha)\geq 0;\\
0,&\text{\ if\ }v(\alpha)<0.
\end{cases}
\end{equation}
Similarly, we consider the case when $j=0$. From (1i) and (2i):
\begin{align}\label{specialpreW2}
&\text{\ \ \ }W (\zxz{\alpha}{0}{0}{1}\zxz{1}{0}{1}{1})=\int\limits_{m\in \F }\varphi(\omega\zxz{1}{m}{0}{1}\zxz{\alpha}{0}{0}{1}\zxz{1}{0}{1}{1})\psi(m) dm\\
&=\int\limits_{m\in \alpha(-1+\varpi O_F)}\mu_1^{-1}(-\frac{\alpha}{m})\mu_2^{-1}(-m)|\frac{\alpha}{m^2}|^{1/2}\psi(m)dm\notag\\
&\text{\ \ \ }-q^{-1}\int\limits_{m\notin \alpha(-1+\varpi O_F)}\mu_1^{-1}(-\frac{\alpha}{\alpha+m})\mu_2^{-1}(-\alpha-m)|\frac{\alpha}{(\alpha+m)^2}|^{1/2}\psi(m)dm\notag\\
&=\begin{cases}
q^{-1}(q^{-1}+q^{-2})\mu_1^{-v(\alpha)}q^{-v(\alpha)/2}, &\text{\ if\ }v(\alpha)\geq 0;\\
q^{-1}(1+q^{-1})\mu_1q^{-1/2}\psi(-\alpha),&\text{\ if\ }v(\alpha)=-1;\\
0,&\text{\ if\ }v(\alpha)<-1.
\end{cases}\notag
\end{align}
After normalization, we get 
\begin{equation}\label{specialW2'}
W (\zxz{\alpha}{0}{0}{1}\zxz{1}{0}{1}{1})=\begin{cases}
-q^{-1}\mu_1^{-v(\alpha)}q^{-v(\alpha)/2}\psi(-\alpha), &\text{\ if\ }v(\alpha)\geq -1;\\
0,&\text{\ if\ }v(\alpha)<-1.
\end{cases}
\end{equation}
\end{proof}

Putting these Whittaker functions back to (\ref{specialtarget}), we have:
\begin{align}\label{specialcombine}
\P (s,w,f,\Phi_s)&=\frac{1}{q+1}\int\limits_{v(\alpha\geq 0}\mu_1^{-v(\alpha)}q^{-(w/2+1/4)v(\alpha)}\Phi_s(\alpha)^{-1}I(\alpha,f,\Phi_s)d^*\alpha \\
&+\frac{q}{q+1}\int\limits_{v(\alpha)\geq -1}-q^{-1}\mu_1^{-v(\alpha)}\psi(-\alpha) q^{-(w/2+1/4)v(\alpha)}\Phi_s(\alpha)^{-1}I(\alpha, r'(\zxz{1}{0}{1}{1})f,\Phi_s)d^*\alpha. \notag
\end{align}
Now we calculate $I(\alpha,f,\Phi_s)$ and $I(\alpha, r'(\zxz{1}{0}{1}{1})f,\Phi_s)$. Recall we pick $$f=char(\zxz{1}{0}{\sqrt{D}}{1}\zxz{O_F}{O_F}{\varpi O_F}{O_F})\times char(O_F^*).$$ 
It's $K_1(\varpi)-$invariant under both the right action and the Weil representation. 
This choice of Schwartz function is motivated by Example \ref{Grtest2}. Note $$\zxz{1}{-\frac{1}{\sqrt{D}}}{\sqrt{D}}{1}\zxz{O_F}{O_F}{\varpi O_F}{O_F}=\zxz{1}{0}{\sqrt{D}}{1}\zxz{O_F}{O_F}{\varpi O_F}{O_F}.$$ 
The invariance on the other side of the Shimizu lifting can be translated into left and right invariance of the Schwartz function.

Alternatively $f$ can be written as 
\begin{equation}
f=\sum\limits_{a_0\in O_F/\varpi O_F}char(\zxz{a_0+\varpi O_F}{O_F}{a_0\sqrt{D}+\varpi O_F}{O_F})\times char(O_F^*).
\end{equation}
It is easier to see from this expression that $f$ is $K_1(\varpi)-$invariant under the Weil representation. One can further calculate according to Lemma \ref{lemofexoticSchwartz} that
\begin{align}
&\text{\ \ } r'(\zxz{1}{0}{1}{1})f\\
&=q^{-2}\sum\limits_{a_0\in O_F/\varpi O_F}\psi(u[(x_1-a_0)x_4-x_2(x_3-a_0\sqrt{D})]) char(\zxz{O_F}{\varpi^{-1}O_F}{O_F}{\varpi^{-1}O_F})\times char(O_F^*).\notag
\end{align}
This sum is also right $K_1(\varpi)-$invariant.

Recall $\Phi_s^{(2)}$ is the unique right $K_1(\varpi)-$invariant function supported on $BK_1(\varpi)$, and $\Phi_s^{(1)}$ is the standard right $K-$invariant function.

We will write $\GL_2(\F )=\zxz{1}{0}{\sqrt{D}}{1}BK_1(\varpi)\cup \zxz{1}{0}{\sqrt{D}}{1}B\zxz{1}{0}{1}{1}K_1(\varpi)$. The matrix $\zxz{1}{0}{\sqrt{D}}{1}$ on the left is just a change of variable. Then by our choice of $\Phi_s^{(2)}$, in particular its support, we only need to integrate over $\zxz{1}{0}{\sqrt{D}}{1}BK_1(\varpi)$ for $I(\alpha,f,\Phi_s)$ and $I(\alpha,r'(\zxz{1}{0}{1}{1})f,\Phi_s)$. By the right $K_1(\varpi)-$invariance of $\Phi_s$, we can write
\begin{align}
&\text{\ \ \ } I(\alpha,f,\Phi_s)\\
&=\frac{1}{q+1}\int\limits_{
\begin{array}{c}
a_1,a_2,m\in O_F\\
v(a_1)+v(a_2)=v(\alpha)
\end{array}
}\Phi_s^{(1)}(\zxz{1}{0}{2\sqrt{D}}{1}\zxz{a_1}{m}{0}{a_2})\Phi_s^{(2)}(\zxz{a_1}{m}{0}{a_2})dmd^*a_2|a_1|^{-1}d^*a_1.\notag
\end{align}
Note that the domain and the integrand of this integral is exactly the same as (\ref{splitpreI}). Let's still denote by $I$ the result we got in (\ref{splitI}). Then $I(\alpha,f,\Phi_s)=\frac{1}{q+1}I$.

Now let's consider $I(\alpha,r'(\zxz{1}{0}{1}{1})f,\Phi_s)$. Note for $\zxz{1}{0}{\sqrt{D}}{1}\zxz{a_1}{m}{0}{a_2}=\zxz{a_1}{m}{a_1\sqrt{D}}{a_2+m\sqrt{D}}$,
\begin{align}
\psi(u[(x_1-a_0)x_4-x_2(x_3-a_0\sqrt{D})])&=\psi(\frac{\alpha}{a_1a_2}[(a_1-a_0)(a_2+m\sqrt{D})-m(a_1\sqrt{D}-a_0\sqrt{D})])\\
&=\psi(\alpha(1-\frac{a_0}{a_1})).\notag
\end{align}
Similarly we can write
\begin{align}
I(\alpha,r'(\zxz{1}{0}{1}{1})f,\Phi_s)
=\frac{q^{-2}}{q+1}\sum\limits_{a_0\in O_F/\varpi O_F}\int\limits_{
\begin{array}{c}
a_1\in O_F\\
a_2,m\in\varpi^{-1} O_F\\
v(a_1a_2)=v(\alpha)
\end{array}
}&\psi(\alpha(1-\frac{a_0}{a_1}))\Phi_s^{(1)}(\zxz{a_1}{m}{2a_1\sqrt{D}}{a_2+2m\sqrt{D}})\\
&\Phi_s^{(2)}(\zxz{a_1}{m}{0}{a_2})dmd^*a_2|a_1|^{-1}d^*a_1.\notag
\end{align}
Compare the domain of each integral in this expression with the domain of (\ref{splitpreI}), we note that we have two additional parts: $$\{v(a_1),v(a_2)\geq 0, v(m)=-1 \}\text{\ \ and\ \ } \{v(a_1)=v(\alpha)+1,v(a_2)=-1, v(m)\geq -1 \}.$$ 
Also note $\psi(\alpha(1-\frac{a_0}{a_1}))=1$ if $v(\alpha)\geq 0$ and $v(a_1)\leq v(\alpha)$. So over the common domain $\{a_1,m,a_2\in O_F \}$, the integral gives $I$ as in (\ref{splitI}). Over the part $\{v(a_1),v(a_2)\geq 0, v(m)=-1 \}$, the integral gives
\begin{align}
&\text{\ \ \ }\sum\limits_{v(a_1)=0}^{v(\alpha)}\chi_{1,s}^{(1)\text{\ }v(\alpha)+1}\chi_{2,s}^{(1)\text{\ }-1}\chi_{1,s}^{(2)\text{\ }v(a_1)}\chi_{2,s}^{(2)\text{\ }v(\alpha)-v(a_1)}q^{v(a_1)}(q-1)\\
&=(q-1)\frac{\chi_{1,s}^{(1)}}{\chi_{2,s}^{(1)}}(\chi_{1,s}^{(1)}\chi_{2,s}^{(2)})^{v(\alpha)}
\frac{1-(\frac{q\chi_{1,s}^{(2)}}{\chi_{2,s}^{(2)}})^{v(\alpha)+1}}
{1-\frac{q\chi_{1,s}^{(2)}}{\chi_{2,s}^{(2)}}}.\notag
\end{align}
The integral over the part $\{v(a_1)=v(\alpha)+1,v(a_2)=-1, v(m)\geq -1 \}$ is more complicated, as $v(a_2+2m\sqrt{D})$ can be larger than $-1$. But the domain and $\Phi_s$ values are independent of $a_0$. We can change the order of the integral and the summation in $a_0$. Then it's easy to see that
\begin{equation}
\sum\limits_{a_0\in O_F/\varpi O_F}\psi(\alpha(1-\frac{a_0}{a_1}))=0
\end{equation}
for $v(a_1)>v(\alpha)$. So this part has no contribution to $I(\alpha,r'(\zxz{1}{0}{1}{1})f,\Phi_s)$. In particular when $v(\alpha)<0$, $I(\alpha,r'(\zxz{1}{0}{1}{1})f,\Phi_s)=0$. When $v(\alpha)\geq 0$,
\begin{equation}
I(\alpha,r'(\zxz{1}{0}{1}{1})f,\Phi_s)=\frac{q^{-2}}{q+1}q(I+(q-1)\frac{\chi_{1,s}^{(1)}}{\chi_{2,s}^{(1)}}(\chi_{1,s}^{(1)}\chi_{2,s}^{(2)})^{v(\alpha)}
\frac{1-(\frac{q\chi_{1,s}^{(2)}}{\chi_{2,s}^{(2)}})^{v(\alpha)+1}}
{1-\frac{q\chi_{1,s}^{(2)}}{\chi_{2,s}^{(2)}}}).
\end{equation}
Putting our results back into (\ref{specialcombine}), we get 
\begin{align}
\P (s,w,f,\Phi_s)&=\frac{1}{q+1}\sum\limits_{v(\alpha)\geq 0}\delta_2^{v(\alpha)}I(\alpha,f,\Phi_s)
+\frac{q}{q+1}\sum\limits_{v(\alpha)\geq -1}-q^{-1}\delta_2^{v(\alpha)}\psi(\alpha) I(\alpha, r'(\zxz{1}{0}{1}{1})f,\Phi_s) \notag\\
&=\frac{1}{(q+1)^2}\sum\limits_{v(\alpha)\geq 0}\delta_2^{v(\alpha)}I-\frac{q^{-1}}{(q+1)^2}\sum\limits_{v(\alpha)\geq 0}\delta_2^{v(\alpha)}(I+(q-1)\frac{\chi_{1,s}^{(1)}}{\chi_{2,s}^{(1)}}(\chi_{1,s}^{(1)}\chi_{2,s}^{(2)})^{v(\alpha)}
\frac{1-(\frac{q\chi_{1,s}^{(2)}}{\chi_{2,s}^{(2)}})^{v(\alpha)+1}}
{1-\frac{q\chi_{1,s}^{(2)}}{\chi_{2,s}^{(2)}}})\notag\\
&=\frac{1-q^{-1}}{(q+1)^2}\frac{1-\frac{\chi_{1,s}^{(1)}}{\chi_{2,s}^{(1)}}}
{(1-\delta\chi^{(1)}_{2,s}\chi^{(2)}_{1,s}\mu_2)(1-\delta\chi^{(1)}_{1,s}\chi^{(2)}_{2,s}\mu_2)(1-q\delta\chi^{(1)}_{1,s}\chi^{(2)}_{1,s}\mu_2)}.\notag
\end{align}
At $w=1/2$, we have 
\begin{equation}\label{specialunramL'}
\P (s,1/2,f,\Phi_s)=\frac{1}{(q+1)^2}\frac{(1-q^{-1})(1-\frac{\chi^{(1)}_{1}}{\chi^{(1)}_{2}}q^{-(2s+1)})}
{(1-\mu_2\chi^{(1)}_{2}\chi^{(2)}_{1}q^{-1/2})(1-\mu_2\chi^{(1)}_{1}\chi^{(2)}_{2}q^{-1/2})(1-\mu_2\chi^{(1)}_{1}\chi^{(2)}_{1}q^{-(2s+1/2)})}.
\end{equation}
\begin{rem}
The numerator we get in $\P (s,w,f,\Phi_s)$ is not symmetric in the two places over $v$. This is because of our choice of $\Phi_s^{(i)}$ for $i=1,2$. We made this choice so we can make use of the results from Section 5.2 directly. One can try to pick $\Phi_s^{(1)}$ also to be $K_1(\varpi)-$invariant and supported  on $BK_1(\varpi)$. But one need more complicated calculations for that choice.
\end{rem}

\subsection{$\pi$ supercupidal and $\E/\F$ split} In this subsection we will prove:
\begin{prop}\label{propsc}
Suppose that  $\pi$ is a supercuspidal representation of level $c$ with unramified central character. Suppose that  $\chi_1$ and $\chi_2$ are unramified and $\E /\F $ is split. Further suppose 2 is a unit. Pick $f$ to be
$$char(\zxz{1}{0}{\sqrt{D}}{1}\zxz{O_F}{O_F}{\varpi^cO_F}{O_F})\times char(O_F^*).$$
Pick $\Phi_s^{(2)}$ to be the unique right $K_1(\varpi^c)-$invariant function supported on $BK_1(\varpi^c)$, and $\Phi_s^{(1)}$ to be the standard right $K-$invariant function.
Then 
\begin{equation}\label{scfinal}
\P (s,w,f,\Phi_s)=\P (s,1/2,f,\Phi_s)=\frac{1-q^{-1}}{(q+1)^2q^{2c-2}}(1-\frac{\chi_{1}^{(1)}}{\chi_{2}^{(1)}}q^{-(2s+1)}).
\end{equation}
The denominator is as expected (actually no denominator), and
\begin{equation}
 \P^0(s,1/2,f,\Phi_s)=\frac{1}{(q+1)^2q^{2c-2}(1-\chi^{(2)}q^{-(2s+1)})}.
\end{equation}

\end{prop}

For a supercuspidal representation, it is easier to use the Kirillov model to describe the elements in the representation and the group actions. But we will have to come back to Whittaker model for specific calculations.
For basic properties of the Kirillov model, one can read \cite{JL70}. For the level and the new form of the Kirillov model of a supercuspidal representation, we mainly follow \cite{Ca73}. Here we just recount part of the facts necessary for our computations.

For the fixed additive character $\psi^-$ there is a unique realization of the supercupidal representation $\hat{\pi}$ on $S(\F ^*)$ such that
\begin{equation}\label{Kirilmodel}
 \hat{\pi}(\zxz{a_1}{m}{0}{a_2})\varphi(x)=w_{\hat{\pi}}(a_2)\psi(-ma_2^{-1}x)\varphi(a_1a_2^{-1}x),
\end{equation}
where $w_{\hat{\pi}}$ is the central character for $\hat{\pi}$. Then by Bruhat decompostion, one just has to know the action of $\omega=\zxz{0}{1}{-1}{0}$ to understand the whole group action. This is however described in a more complicated way.

Fix a uniformizer $\varpi$ for the local field $\F $. For any $\varphi \in S(\F ^*)$ and character $\nu$ of $O_F^*$, consider the Mellin transform
\begin{equation}
\tilde{\varphi_n}(\nu)=\int\limits_{O_F^*}\varphi(\varpi^n x)\nu(x)d^*x.
\end{equation}
Define $$\charf_{\nu,n}(x)=\begin{cases}
                        \nu(u), &\text{if\ } x=u\varpi^n\text{\  for\ } u\in O_F^*;\\
						0,&\text{otherwise}.
                       \end{cases}
 $$ Roughly speaking, it's the character $\nu$ supported at $v(x)=n$. 
 Then we can recover $\varphi$ from $\tilde{\varphi_n}(\nu)$ as
\begin{equation}
 \varphi(x)=\sum\limits_{\nu}\sum\limits_{n\in \Z} \tilde{\varphi_n}(\nu^{-1}) \charf_{\nu,n}(x).
\end{equation}
For fixed $\nu$, we can define a formal power series $\tilde{\varphi}(\nu,t)=\sum t^n\tilde{\varphi_n}(\nu)$. Denote $\varphi'=\hat{\pi}(\omega)\varphi$. Then we can define the action of $\omega$ in this language:
\begin{equation}\label{omegaaction}
 \tilde{\varphi'} (\nu,t)=C(\nu, t)\tilde{\varphi}(\nu^{-1}w_0^{-1},t^{-1}z_0^{-1}).
\end{equation}
Here $w_0$ denotes $w_{\hat{\pi}}|_{O_F^*}$, and $z_0=w_{\hat{\pi}}(\varpi)$. $C(\nu,t )$ is a formal power series. It's related to L-function and $\epsilon$ factors by:
\begin{equation}\label{connectiontoepsilon}
 C(w_0^{-1},z_0^{-1}q^{s-1/2})=\frac{L(\hat{\pi},1-s)\epsilon(\pi,\psi,s)}{L(\pi,s)}.
\end{equation}

It's shown in \cite{JL70} that for a supercuspidal representation, $C(\nu,t)$ is actually a monomial $C_{\nu}t^{n_{\nu}}$ such that $n_{\nu}\leq -2$. 
The relation $\omega^2=-\zxz{1}{0}{0}{1}$ implies $C(\nu, t)C(\nu^{-1}w_0^{-1},t^{-1}z_0^{-1})=w_0(-1)$, which in turn implies
\begin{equation}
n_{\nu}=n_{\nu^{-1}w_0^{-1}},\text{\ \ } C_\nu C_{\nu^{-1}w_0^{-1}}=w_0(-1)z_0^{n_\nu}.
\end{equation}
We can also compute the action of $\omega$ on $\varphi=\charf_{\nu,n}$ explicitly by (\ref{omegaaction}) using $n_\nu$ and $C_\nu$. 
\begin{lem}
\begin{equation}\label{singleaction}
\hat{\pi}(\omega)\charf_{\nu,n}=C_{\nu w_0^{-1}}z_0^{-n}\charf_{\nu^{-1}w_0,-n+n_{\nu w_0^{-1}}}=C_{\nu w_0^{-1}}z_0^{-n}\charf_{\nu^{-1}w_0,-n+n_{\nu^{-1}}}.
\end{equation}
\end{lem}

In particular, if $w_0=1$, the action of $\omega$ will not change the level of the characters.

According to the Appendix, the argument in \cite{Ca73} with slight modification can show that there is a unique up to constant element $\varphi$ in the supercupidal representation which is invariant under $K_1(\varpi^{-n_1})$. One can actually pick $\varphi=\charf_{1,0}$. This is the analogue of the newform for a supercuspidal representation. Let $c=-n_1\geq 2$. It's the level of the supercuspidal representation.

From now on, we assume that 
the central character $w_{\hat{\pi}}$ is unramified, so $w_0=w_{\hat{\pi}}|_{O_F^*}=1$.
For the newform $\varphi=\charf_{1,0}$, its associated Whittaker function $W $ is also right $K_1(\varpi^c)-$invariant. We can calculate $W=W_\varphi^- $ according to the relation between the Kirillov model and the Whittaker model:
\begin{equation}
W (\zxz{\alpha}{0}{0}{1}\zxz{1}{0}{\varpi^i}{1})=\hat{\pi}(\zxz{1}{0}{\varpi^i}{1})\varphi(\alpha).
\end{equation}
The idea here is that it may be difficult/not very enlightening to write out the value of $W (\zxz{\alpha}{0}{0}{1}\zxz{1}{0}{\varpi^i}{1})$ explicitly. But it will suffice to know only some specific integrals for $W (\zxz{\alpha}{0}{0}{1}\zxz{1}{0}{\varpi^i}{1})$.
\begin{lem}\label{scWhittaker}
Suppose $W$ is the Whittaker function associated to $\varphi=\charf_{1,0}\in S(\F^*)$.
\begin{enumerate}
\item[(1)]
$W (\zxz{\alpha}{0}{0}{1})=\charf_{1,0}$. For $0\leq i<c$, $W (\zxz{\alpha}{0}{0}{1}\zxz{1}{0}{\varpi^i}{1})$ is supported only at $v(\alpha)=c+\min\{-c, -2(c-i)\}$.
\item[(2)]
$\int\limits_{v(\alpha)=c+\min\{-c, -2(c-i)\}}W (\zxz{\alpha}{0}{0}{1}\zxz{1}{0}{\varpi^i}{1})d^*\alpha=\begin{cases}
1, &\text{\ if\ }i\geq c;\\-\frac{1}{q-1}, &\text{\ if\ }i=c-1;\\0, &\text{\ otherwise}.
\end{cases}$
\item[(3)]
$\int\limits_{v(\alpha)=c+\min\{-c, -2(c-i)\}}W (\zxz{\alpha}{0}{0}{1}\zxz{1}{0}{\varpi^i}{1})\psi(\varpi^{-i}\alpha) d^*\alpha=\begin{cases}
C_1, &\text{\ if\ }i=0;\\-\frac{1}{q-1}C_1w_{\hat{\pi}}, &\text{\ if\ }i=1;\\0, &\text{\ otherwise}.
\end{cases}$
\end{enumerate}
\end{lem}
\begin{proof}
The first statement of (1) is clear. Now let $0\leq i<c$. According to Proposition \ref{powerofmonomials}, if $\nu$ is a character of level $i$, then $n_\nu=\min\{n_1,-2i\}$.
Note $$\zxz{1}{0}{\varpi^i}{1}=-\omega\zxz{1}{-\varpi^i}{0}{1}\omega, \text{\ \ } \hat{\pi}(\omega)\charf_{1,0}=C_{w_0^{-1}}\charf_{w_0,n_1}=C_1\charf_{1,n_1}.$$ The action of $\zxz{1}{-\varpi^i}{0}{1}$ for $i< -n_1$ will give a non-trivial factor $\psi(\varpi^i x)$ at $v(x)=n_1$. By the classical result of Gauss sum, $\hat{\pi}(\zxz{1}{-\varpi^i}{0}{1}\omega)\charf_{1,0}$ is a linear combination of all characters of level $-n_1-i=c-i$ (It should be understood that if $c-i=1$, then this is a linear combination of all characters of level 1 and 0), supported at $v(x)=n_1$. Another action of $\omega$ will keep their levels. Their support will become $-n_1+\min\{n_1, -2(c-i)\}$, according to formula (\ref{singleaction}).

When we integrate $W (\zxz{\alpha}{0}{0}{1}\zxz{1}{0}{\varpi^i}{1})$, we are just finding the level 0 component of it. By the discussion above, $\hat{\pi}(\zxz{1}{-\varpi^i}{0}{1}\omega)\charf_{1,0}$ consists of all characters of level $c-i$, and the action of $\omega$ keeps their levels. In particular, when $i<c-1$, $W (\zxz{\alpha}{0}{0}{1}\zxz{1}{0}{\varpi^i}{1})$ contains no level 0 component. 
When $i=c-1$, one can compute that the level 0 component of $\hat{\pi}(\zxz{1}{-\varpi^i}{0}{1}\omega)\charf_{1,0}$ is $-\frac{1}{q-1}C_1\charf_{1,n_1}$, as $\int\limits_{x\in \varpi^{-1}O_F^*}\psi(x)d^*x=-\frac{1}{q-1}$. Then the action of $\omega$ will map it to $-\frac{1}{q-1}\charf_{1,0}$. 
Its support agrees with the support of the integral which is $v(\alpha)=c+\min\{-c, -2(c-i)\}=0$ as $c\geq 2$. So
\begin{equation*}
\int\limits_{v(\alpha)=0}W (\zxz{\alpha}{0}{0}{1}\zxz{1}{0}{\varpi^{c-1}}{1})d^*\alpha=-\frac{1}{q-1}.
\end{equation*}
When $i=c$ it is obvious that
\begin{equation*}
\int\limits_{v(\alpha)=0}W (\zxz{\alpha}{0}{0}{1})d^*\alpha=\int\limits_{v(\alpha)=0}\charf_{1,0}d^*\alpha=1.
\end{equation*}
Now to integrate $W (\zxz{\alpha}{0}{0}{1}\zxz{1}{0}{\varpi^i}{1})\psi(\varpi^{-i}\alpha)$, the idea is to interpret $\psi(\varpi^{-i}\alpha)$ as a factor one can get by the group action in the Kirillov model. More specifically,
\begin{align}
W (\zxz{\alpha}{0}{0}{1}\zxz{1}{0}{\varpi^i}{1})\psi(\varpi^{-i}\alpha)&=\hat{\pi}(\zxz{1}{0}{\varpi^i}{1})\varphi(\alpha)\psi(\varpi^{-i}\alpha)\\
&=\hat{\pi}(\zxz{1}{-\varpi^{-i}}{0}{1}\zxz{1}{0}{\varpi^i}{1})\varphi(\alpha)\notag\\
&=\hat{\pi}(\zxz{0}{-\varpi^{-i}}{\varpi^i}{1})\varphi(\alpha)\notag\\
&=\hat{\pi}(-\omega\zxz{\varpi^i}{1}{0}{\varpi^{-i}})\varphi(\alpha).\notag
\end{align}
By definition of the Kirillov model, $$\hat{\pi}(\zxz{\varpi^i}{1}{0}{\varpi^{-i}})\varphi(x)=w_{\hat{\pi}}^{-i}(\varpi)\psi(-\varpi^i x)\charf_{1,0}(\varpi^{2i}x).$$
It is supported at $v(x)=-2i$, and is a linear combination of all characters of level $i$ (again if $i=1$, this should be understood as level 1 and 0). The action of $\omega$ will keep the levels. Note that the integration in $\alpha$ finally is again to find the level 0 component of $W (\zxz{\alpha}{0}{0}{1}\zxz{1}{0}{\varpi^i}{1})\psi(\varpi^{-i}\alpha)$. In particular, the integral will be zero when $i>1$, as there is no level 0 component.
When $i=0$, the level 0 component of $\hat{\pi}(\zxz{\varpi^i}{1}{0}{\varpi^{-i}})\varphi$ is just $\charf_{1,0}$. The action of $\omega$ will change it into $C_1\charf_{1,n_1}$. The action of $-1$ just gives a factor $w_{\hat{\pi}}(-1)=1$. Note that the support of the level 0 component of $W (\zxz{\alpha}{0}{0}{1}\zxz{1}{0}{1}{1})\psi(\alpha)$ agrees with the support of the integral which is $v(\alpha)=c+\min\{-c,-2c\}=-c$. Then
\begin{equation*}
\int\limits_{v(\alpha)=-c}W (\zxz{\alpha}{0}{0}{1}\zxz{1}{0}{1}{1})\psi(\alpha)d^*\alpha=C_1.
\end{equation*}
Similarly when $i=1$, the level 0 component of $\hat{\pi}(\zxz{\varpi^1}{1}{0}{\varpi^{-1}})\varphi$ is $-\frac{1}{q-1}w_{\hat{\pi}}^{-1}\charf_{1,-2}$. Here by $w_{\hat{\pi}}$ we mean $w_{\hat{\pi}}(\varpi)$. Then according to (\ref{singleaction}), the action of $\omega$ will map it to $$-\frac{1}{q-1}C_1w_{\hat{\pi}}^{-1}z_0^2\charf_{1,2-c}=-\frac{1}{q-1}C_1w_{\hat{\pi}}\charf_{1,2-c}.$$
Its support agrees with $v(\alpha)=c+\min\{-c,-2(c-1)\}=2-c$, as $c\geq 2$ implies $-c\geq -2(c-1)$. Then
\begin{equation*}
\int\limits_{v(\alpha)=2-c}W (\zxz{\alpha}{0}{0}{1}\zxz{1}{0}{\varpi}{1})\psi(\alpha)d^*\alpha=-\frac{1}{q-1}C_1w_{\hat{\pi}}.
\end{equation*}
\end{proof}
Now we work on $I(\alpha,f,\Phi_s)$. We will basically follow the technique used for the unramified special case in the last subsection. Recall we pick $f$ to be $$char(\zxz{1}{0}{\sqrt{D}}{1}\zxz{O_F}{O_F}{\varpi^cO_F}{O_F})\times char(O_F^*).$$ It is $K_1(\varpi^c)-$invariant under the right action and the Weil representation. Alternatively it can be written as 
\begin{equation}
f=\sum\limits_{a_0\in O_F/\varpi^cO_F}char(\zxz{a_0+\varpi^cO_F}{O_F}{a_0\sqrt{D}+\varpi^cO_F}{O_F})\times char(O_F^*).
\end{equation}
One can calculate for $0\leq i\leq c$ that
\begin{align}
&\text{\ \ } r'(\zxz{1}{0}{\varpi^i}{1})f\\
&=q^{2(i-c)}\sum\limits_{a_0\in O_F/\varpi^{c}O_F}\psi(u\varpi^{-i}[(x_1-a_0)x_4-x_2(x_3-a_0\sqrt{D})]) char(\zxz{a_0+\varpi^{i}O_F}{\varpi^{i-c}O_F}{a_0\sqrt{D}+\varpi^iO_F}{\varpi^{i-c}O_F})\times char(O_F^*).\notag
\end{align}
The sum is right $K_1(\varpi^c)-$invariant for any $i$.

Recall $\Phi_s=\Phi_s^{(1)}\Phi_s^{(2)}$, and $\gamma_0$ should be understood as $\zxz{1}{0}{(\sqrt{D},-\sqrt{D})}{1}$. Recall $\Phi_s^{(2)}$ is the unique right $K_1(\varpi^c)-$invariant function supported on $BK_1(\varpi^c)$, and $\Phi_s^{(1)}$ is the standard right $K-$invariant function.

Now the local integral can be written as 
\begin{equation}\label{scgoal}
\P (s,w,f,\Phi_s)=\sum\limits_{0\leq i\leq c}A_i\int\limits_{v(\alpha)=c+\min\{-c, -2(c-i)\}}W (\zxz{\alpha}{0}{0}{1}\zxz{1}{0}{\varpi^i}{1})|\alpha|^{\frac{w}{2}-\frac{1}{4}}\Phi_s(\alpha)^{-1} I(\alpha, r'(\zxz{1}{0}{\varpi^i}{1})f,\Phi_s)d^*\alpha,
\end{equation}
where $A_i$'s are decided in Lemma \ref{integraldecompcoeff}. Recall 
$$I(\alpha,r'(\zxz{1}{0}{\varpi^i}{1})f,\Phi_s)=\int\limits_{\GL_{2}(\F)}r'(\zxz{1}{0}{\varpi^i}{1})f(g,\alpha \det (g)^{-1})\Phi_{s}(\gamma_0 g)dg.$$
We will write 
$$\GL_2(\F )=\coprod\limits_{0\leq j\leq c}\zxz{1}{0}{\sqrt{D}}{1}B\zxz{1}{0}{\varpi^j}{1}K_1(\varpi^c).$$
Then by our choice of $\Phi_s^{(2)}$, in particular its support, we only need to integrate over $\zxz{1}{0}{\sqrt{D}}{1}BK_1(\varpi^c)$ for $I(\alpha,r'(\zxz{1}{0}{\varpi^i}{1})f,\Phi_s)$. By the right $K_1(\varpi^c)-$invariance of $\Phi_s$, we can write
\begin{align}
&\text{\ \ \ } I(\alpha,r'(\zxz{1}{0}{\varpi^i}{1})f,\Phi_s)\\
&=A_c\int\Phi_s^{(1)}(\zxz{1}{0}{2\sqrt{D}}{1}\zxz{a_1}{m}{0}{a_2})\Phi_s^{(2)}(\zxz{a_1}{m}{0}{a_2})r'(\zxz{1}{0}{\varpi^i}{1})f(\zxz{a_1}{m}{a_1\sqrt{D}}{a_2+m\sqrt{D}},\frac{\alpha}{a_1a_2})dmd^*a_2|a_1|^{-1}d^*a_1.\notag
\end{align}
For $\zxz{a_1}{m}{a_1\sqrt{D}}{a_2+m\sqrt{D}}$,
\begin{align}
\psi(u\varpi^{-i}[(x_1-a_0)x_4-x_2(x_3-a_0\sqrt{D})])=\psi(\varpi^{-i}\alpha(1-\frac{a_0}{a_1})).
\end{align}
So
\begin{align}
&\text{\ \ \ } r'(\zxz{1}{0}{\varpi^i}{1})f(\zxz{1}{0}{\sqrt{D}}{1}\zxz{a_1}{m}{0}{a_2})\\
&=q^{2(i-c)}\sum\limits_{a_0\in O_F/\varpi^{c}O_F}\psi(\varpi^{-i}\alpha(1-\frac{a_0}{a_1}))char(\zxz{a_0+\varpi^{i}O_F}{\varpi^{i-c}O_F}{a_0\sqrt{D}+\varpi^iO_F}{\varpi^{i-c}O_F})(\zxz{a_1}{m}{a_1\sqrt{D}}{a_2+m\sqrt{D}}).\notag
\end{align}
For each $a_0$, the corresponding term in the above expression is not zero if and only if 
\begin{equation}
a_1\equiv a_0 \mod{(\varpi^iO_F)}, \text{\ \ \ \ }m,a_2\in \varpi^{i-c}O_F.
\end{equation}
\begin{lem}\label{scpossibleI}
For any $0\leq i\leq c$ and fixed $v(\alpha)$, $I(\alpha,r'(\zxz{1}{0}{\varpi^i}{1})f,\Phi_s)$ as a function of $\alpha$ is a linear combination of constant independent of $\alpha$ and $\psi(\varpi^{-i}\alpha)$.
\end{lem}
\begin{proof}
For fixed $a_0$, the corresponding term in $I(\alpha,r'(\zxz{1}{0}{\varpi^i}{1})f,\Phi_s)$ is
\begin{align}\label{Intforeacha_0}
A_cq^{2(i-c)}\int\limits_{
\begin{array}{c}
a_1\equiv a_0 \mod{(\varpi^i)}\\
v(a_1a_2)=v(\alpha)\\
v(a_2),v(m)\geq i-c
\end{array}
}&\Phi_s^{(1)}(\zxz{a_1}{m}{2a_1\sqrt{D}}{a_2+2m\sqrt{D}})\chi_{1,s}^{(2)}(a_1)\chi_{2,s}^{(2)}(a_2)\\
&\psi(\varpi^{-i}\alpha(1-\frac{a_0}{a_1}))dmd^*a_2|a_1|^{-1}d^*a_1.\notag
\end{align}
There are two cases. If $a_0\in \varpi^iO_F$, the domain for $a_1$ is $\varpi^iO_F$. We first integrate in $a_1$ for fixed $v(a_1)$. Note that $\Phi_s^{(1)}(\zxz{a_1}{m}{2a_1\sqrt{D}}{a_2+2m\sqrt{D}})$ only depends on $v(a_1)$ instead of the specific value of $a_1$, as one can see from Lemma \ref{localcases}. So we are essentially integrating $$\psi(\varpi^{-i}\alpha(1-\frac{a_0}{a_1}))=\psi(\varpi^{-i}\alpha)\psi(-\varpi^{-i}\alpha\frac{a_0}{a_1}).$$ 
Then we get either 0 or a multiple of $\psi(\varpi^{-i}\alpha)$. 

If $a_0\notin \varpi^i O_F$, we consider the sum in $a_0$ for fixed $v(a_0)<i$. Note $v(a_1)=v(a_0)$ would also be fixed. As the value of $\Phi_s$ and the domains for the integrals in $m$ and $a_2$ are actually independent of $a_0$, we can change the order of the integral in $a_2$, $m$ and the summation in $a_0$. Then the sum in $a_0$ is essentially
\begin{align}
&\text{\ \ }\sum\limits_{a_0}\int\limits_{a_1\equiv a_0\mod{(\varpi^i)}}\psi(\varpi^{-i}\alpha\frac{a_1-a_0}{a_1})d^*a_1\\
&=\int\limits_{a_1}\sum\limits_{a_0\equiv a_1\mod{(\varpi^i)}}\psi(\varpi^{-i}\alpha\frac{a_1-a_0}{a_1})d^*a_1.\notag
\end{align}
One can now easily see that the inner sum is either 0 or a constant independent of $\alpha$.
\end{proof}
As a result of this Lemma and Lemma \ref{scWhittaker}, we only have to care about the constant part when $i=c,c-1$ and the $\psi(\varpi^{-i}\alpha)$ part when $i=0,1$.
\begin{lem}\label{scIs}
Suppose $v(\alpha)=c+\min\{-c,2i-2c\}$.
\begin{enumerate}
\item[(1)]$I(\alpha,f,\Phi_s)=A_c$.

The constant part of $I(\alpha,r'(\zxz{1}{0}{\varpi^{c-1}}{1})f,\Phi_s)$ is $A_c[q^{-1}+(1-q^{-1})\frac{\chi_{1,s}^{(1)}}{\chi_{2,s}^{(1)}}]$.

\item[(2)]$I(\alpha,r'(\zxz{1}{0}{1}{1})f,\Phi_s)=0 $.

The $\psi(\varpi^{-1}\alpha)$ part of $I(\alpha,r'(\zxz{1}{0}{\varpi}{1})f,\Phi_s)$ is 0.
\end{enumerate}
\end{lem}
\begin{proof}
(1)When $i\geq c-1$, $v(\alpha)=0$ as $c\geq 2$ for supercuspidal representations. We shall follow formula (\ref{Intforeacha_0}) for each $a_0$. When $i=c$, the domain of the integral is non-empty when $v(a_0)=0$. In that case, $a_1\equiv a_0\mod{(\varpi^c)}, v(a_2)=0, v(m)\geq 0$. By Lemma \ref{localcases}, $$\Phi_{s}^{(1)}(\zxz{a_1}{m}{2a_1\sqrt{D}}{a_2+2m\sqrt{D}})=\chi_{1,s}^{(1)}(\frac{a_2}{\sqrt{D}})\chi_{2,s}^{(1)}(a_1\sqrt{D})=1.$$ 
Also $\psi(\varpi^{-c}\alpha\frac{a_1-a_0}{a_1})=1$.
Then
\begin{align}
I(\alpha,f,\Phi_s)=A_c\sum\limits_{a_0\in O_F^*/\varpi^cO_F}\frac{1}{(q-1)q^{c-1}}=A_c.
\end{align}

When $i=c-1$, we need $v(a_1)\leq 1$ as $v(\alpha)=0$ and $v(a_2)\geq -1$. This means if $c>2$, we only need to consider $v(a_0)\leq 1$; if $c=2$, then we need to consider all possible $v(a_0)$.

If $v(a_0)=0$, then $v(a_1)=0$, and the domain of the integral is $a_1\equiv a_0\mod{(\varpi^{c-1})}, v(a_2)=0, v(m)\geq -1$. Then $\psi(\varpi^{-c+1}\alpha(1-\frac{a_0}{a_1}))=1$. When $v(m)\geq 0$, $\Phi_{s}^{(1)}(\zxz{a_1}{m}{2a_1\sqrt{D}}{a_2+2m\sqrt{D}})=1$. When $v(m)=-1$, it is $\frac{\chi_{1,s}^{(1)}}{\chi_{2,s}^{(1)}}$.

Then the contribution of these parts to $I(\alpha,r'(\zxz{1}{0}{\varpi^{c-1}}{1})f,\Phi_s)$ is 
\begin{align}
&\text{\ \ }A_cq^{-2}\sum\limits_{a_0\in O_F^*/\varpi^cO_F}[\frac{1}{(q-1)q^{c-2}}+\frac{1}{(q-1)q^{c-2}}(q-1)\frac{\chi_{1,s}^{(1)}}{\chi_{2,s}^{(1)}}]\\
&=A_c[q^{-1}+(1-q^{-1})\frac{\chi_{1,s}^{(1)}}{\chi_{2,s}^{(1)}}] .  \notag
\end{align}

If $c=2$ and $v(a_0)\geq 1$, the domain for $a_1$ is just $\varpi O_F^*$. Then by the proof of the previous lemma, we will get $\psi(\varpi^{-c+1}\alpha)$ part. So we don't have to discuss it here. 

If $c> 2$ and $v(a_0)=1$, the domain for $a_1$ is $a_1\equiv a_0\mod{(\varpi^{c-1})}$. We also follow the proof in the previous lemma and change the order of the integral in $a_1$ and the sum over $a_0$. In particular, we will have a sum like
\begin{equation}
\sum\limits_{a_0\equiv a_1\mod{(\varpi^{c-1})}}\psi(\varpi^{-c+1}\alpha\frac{a_1-a_0}{a_1}).
\end{equation}
One can easily check that this sum is zero, as $v(a_1)=1$ and $\int\limits_{x\in O_F}\psi(\varpi^{-1}x)dx=0$.
\newline

(2)When $i\leq 1$, $v(\alpha)=2i-c$ as $c\geq 2$. When $i=0$, the domain of the integral is always $v(a_1)=0$, $v(m)\geq v(a_2)=-c$. We change order of integral in $a_1$ and summation in $a_0$. One can easily see that  
\begin{equation}
\sum\limits_{a_0\in O_F/\varpi^cO_F}\psi(\alpha(1-\frac{a_0}{a_1}))=0,
\end{equation}
as $v(\alpha)=-c<0$ and $v(a_1)=0$.

When $i=1$, $v(\alpha)=2-c$. If $v(a_0)=0$, the domain of the integral is $a_1\equiv a_0\mod{(\varpi^{1})}$, $v(m)\geq 1-c$, $v(a_2)=2-c$. According to the proof of the previous lemma, this part gives constant part, which we don't have to discuss here. When $v(a_0)\geq 1$, the domain of the integral is $v(a_1)=1$, $v(m)\geq v(a_2)=1-c$.
We can change the order as before and then
\begin{equation}
\sum\limits_{a_0\in \varpi O_F/\varpi^cO_F}\psi(\varpi^{-1}\alpha(1-\frac{a_0}{a_1}))=0
\end{equation}
as $c\geq 2$.
\end{proof}

\begin{rem}
It may seem that we are just getting zero for $\psi(\varpi^{-i}\alpha)$ part and we do not need the full power of Lemma \ref{scWhittaker} and Lemma \ref{scpossibleI} for our calculations. But the conclusion in Lemma \ref{scpossibleI} still holds for many other choices of Schwartz functions. Then part (3) of Lemma \ref{scWhittaker} will be important to the fact that $\P (s,w,f,\Phi_s)$ vanishes if we choose $f$ to be, for example, $char(\zxz{O_F}{O_F}{\varpi^cO_F}{\varpi^cO_F})\times char(O_F^*)$ and $\Phi_s$ to be $K-$ invariant.
\end{rem}
Now we combine Lemma \ref{scWhittaker} and Lemma \ref{scIs} to compute (\ref{scgoal}). Note that only $i=0,1$ terms are non-zero, and $v(\alpha)=0$ for these terms. Then 
\begin{align}\label{scfinal'}
&\text{\ \ \ }\P (s,w,f,\Phi_s)\\
&=A_c\int\limits_{v(\alpha)=0}W (\zxz{\alpha}{0}{0}{1})I(\alpha, f,\Phi_s)d^*\alpha+A_{c-1}\int\limits_{v(\alpha)=0}W (\zxz{\alpha}{0}{0}{1}\zxz{1}{0}{\varpi^{c-1}}{1})I(\alpha, r'(\zxz{1}{0}{\varpi^{c-1}}{1})f,\Phi_s)d^*\alpha\notag\\
&=A_c\cdot A_c  +A_{c-1}(-\frac{1}{q-1})A_c[q^{-1}+(1-q^{-1})\frac{\chi_{1,s}^{(1)}}{\chi_{2,s}^{(1)}}]\notag\\
&=\frac{1-q^{-1}}{(q+1)^2q^{2c-2}}(1-\frac{\chi_{1}^{(1)}}{\chi_{2}^{(1)}}q^{-(2s+1)}).\notag
\end{align}
Note this result is independent of $w$.

\subsection{$\pi$ highly ramified principal series}
Now we consider the case when $\pi$ is a ramified principal series. We still assume $\chi_1$ and $\chi_2$ are unramified and $\E /\F $ is split. Let $c_1,c_2$ be levels of $\mu_1,\mu_2$. Then the assumption implies $c_1=c_2$, and $\mu_1\mu_2=(\chi_{1,s}\chi_{2,s})^{-1}$ is unramified. Let $k=c_1=c_2\geq 1$ and $c$ be level of $\pi(\mu_1^{-1},\mu_2^{-1})$, then $c=2k\geq 2$. It's a classical result (refer to [Ca]) that the subspace of $\pi(\mu_1^{-1},\mu_2^{-1})$ which is $K_1(\varpi^c)-$invariant is one dimensional and one can pick a representative $\varphi$ such that $\varphi$ is supported on $B\zxz{1}{0}{\varpi^{c_2}}{1}K_1(\varpi^c)$. Such $\varphi$ is called a newform. The idea is to consider its Whittaker function and try to get an analogue of Lemma \ref{scWhittaker}. Then we can choose the same $f$ and $\Phi_s$ as in the previous subsection and get results directly. So we will prove
\begin{prop}\label{prophighlyrampi}
Suppose that  $\pi$ is a ramified principal series of even level $c$ with $\mu_1$ $\mu_2$ both of level $c/2$. Suppose that  $\chi_1$ and $\chi_2$ are unramified and $\E /\F $ is split. Further suppose 2 is a unit. Pick $f$ and $\Phi_s$ to be the same as in Proposition \ref{propsc}.
Then 
\begin{equation}
\P (s,w,f,\Phi_s)=\P (s,1/2,f,\Phi_s)=\frac{1-q^{-1}}{(q+1)^2q^{2c-2}}(1-\frac{\chi_{1}^{(1)}}{\chi_{2}^{(1)}}q^{-(2s+1)})
\end{equation}
The denominator is as expected (actually no denominator), and
\begin{equation}
 \P^0(s,1/2,f,\Phi_s)=\frac{1}{(q+1)^2q^{2c-2}(1-\chi^{(2)}q^{-(2s+1)})}.
\end{equation}
\end{prop}

Before we start, we record an easy lemma first.
\begin{lem}\label{highlevelcharint}
Suppose $\mu$ is a character of level $k>0$ on $\F ^*$. Then 
\begin{equation}
\int\limits_{x\in O_F^*}\mu(1+\varpi^i x)dx=\begin{cases}
0, &\text{\ \ if\ }i<k-1;\\
-q^{-1},&\text{\ \ if\ }i=k-1;\\
1-q^{-1},&\text{\ \ if\ }i\geq k.
\end{cases}
\end{equation}
\end{lem}
\begin{proof}
The case $i\geq k$ is clear as $\mu(1+\varpi^i x)=\mu(1)=1$ by definition. 
If $k>1$, there is some $n\in O_F^*$ such that $\mu(1+\varpi^{k-1}n)\neq 1$. So we can get the following by changing variable:
\begin{align}\label{eqbasicblock}
\int\limits_{x\in O_F}\mu(1+\varpi^{k-1} x)dx&=\int\limits_{x\in O_F}\mu(1+\varpi^{k-1} (x+n))dx\\
&=\int\limits_{x\in O_F}\mu((1+\varpi^{k-1} x)(1+\varpi^{k-1}n))dx       \notag\\
&=\mu(1+\varpi^{k-1}n)\int\limits_{x\in O_F}\mu(1+\varpi^{k-1} x)dx.\notag
\end{align}
So the integral has to be zero. We have used that $\mu$ is of level $k>1$ for the second equality. From this result, we get
\begin{align}
0=\int\limits_{x\in O_F}\mu(1+\varpi^{k-1} x)dx=\int\limits_{x\in O_F^*}\mu(1+\varpi^{k-1} x)dx+\int\limits_{x\in \varpi O_F}\mu(1+\varpi^{k-1} x)dx.
\end{align}
The second integral is just $q^{-1}$ by $\mu$ being level $k$. Then the statement for $i=k-1$ is true when $k>1$. 

When $i<k-1$, we can break the integral into pieces, each of which has the form (\ref{eqbasicblock}). So we get zero for the integral when $i<k-1$.

Now if $k=1$, we have 
\begin{equation}
\int\limits_{x\in O_F^*}\mu(x)dx=0.
\end{equation}
Then we have
\begin{equation}
\int\limits_{x\in O_F^*}\mu(1+x)dx=\int\limits_{x\in \varpi O_F}\mu(x)dx+\int\limits_{x\in O_F^*}\mu(x)dx-\int\limits_{x\in 1+\varpi O_F}\mu(x)dx.
\end{equation}
The first two integrals on the right-hand side should be zero, while the last integral gives $q^{-1}$. So the statement for the case $k=1$, $i=k-1=0$ is also true.

Finally when $i<0$ and $k=1$, 
\begin{equation}
\int\limits_{x\in O_F^*}\mu(1+\varpi^i x)dx=\int\limits_{x\in O_F^*}\mu(\varpi^i x)dx=0.
\end{equation}
\end{proof}
\begin{lem}\label{Wiofram}
We denote the value $W (\zxz{\alpha}{0}{0}{1}\zxz{1}{0}{\varpi^i}{1})$ by $W_i(\alpha)$ for short.
\begin{enumerate}
\item[(1)]If $i<k$, then $W_i(\alpha)=0$ except when $v(\alpha)=2i-c$. In that case, its integral against 1 is always 0. 
\begin{equation}
\int\limits_{v(\alpha)=2i-c}W_i(\alpha)\psi(\varpi^{-i}\alpha)d^*\alpha=w_{\pi}(\varpi^{k-i})\mu_1^{-1}(-1)\begin{cases}
1, &\text{\ \ if\ }i=0;\\
-\frac{1}{q-1}, &\text{\ \ if\ }i=1<k;\\
0, &\text{\ \ otherwise}.
\end{cases}
\end{equation}
\item[(2)]If $k<i\leq c$, then $W_i(\alpha)$ is zero except when $v(\alpha)=0$. In that case, its integral against $\psi(\varpi^{-i}\alpha)$ is always 0.
\begin{equation}
\text{\ \ \ }\int\limits_{v(\alpha)=0}W_i(\alpha)d^*\alpha=\begin{cases}
1, &\text{\ \ if\ }i=c;\\
-\frac{1}{q-1},&\text{\ \ if\ }i=c-1>k;\\
0,&\text{\ \ otherwise}.
\end{cases}
\end{equation}
\item[(3)]If $i=k$, the integral of $W_k$ against 1 or $\psi(\varpi^{-k}\alpha)$ is always zero if either $k>1$ or $v(\alpha)\neq 0$. When $k=1$ and $v(\alpha)=0$, its integral against 1 is the same as expected from (2) as the limit case, and its integral against $\psi(\varpi^{-k}\alpha)$ is the same as expected from (1).
\end{enumerate}
\end{lem}
\begin{proof}
According to part (3) of Lemma \ref{LevelIwaDec}, we can always write 
\begin{equation}
W_i(\alpha)=\int \mu_1^{-1}(-\frac{\alpha\varpi^k}{\alpha+m\varpi^i})\mu_2^{-1}(-m)|\frac{\alpha\varpi^k}{m(\alpha+m\varpi^i)}|^{1/2}\psi(m)dm.
\end{equation}
Recall it's the Whittaker function associated to $\psi^-$. All terms with $i$ will disappear when $i=c$. The difference is the domain for $m$, which was given in Lemma \ref{LevelIwaDec}. 
\newline
When $i<k$, we have $m\in \alpha\varpi^{-i}(-1+\varpi^{k-i}O_F^*)$. Make a substitutiou $m=\alpha\varpi^{-i}(-1+\varpi^{k-i}u)$ for $u\in O_F^*$. Then the integral can be writen as
\begin{equation*}
W_i(\alpha)=\int\limits_{u\in O_F^*}\mu_1^{-1}(-\frac{\varpi^i}{u})\mu_2^{-1}(\alpha\varpi^{-i}(1-\varpi^{k-i}u))q^{v(\alpha)/2-i}\psi(-\alpha\varpi^{-i}(1-\varpi^{k-i}u))q^{2i-k-v(\alpha)}du.
\end{equation*}
As functions of $u$, $\mu_1^{-1}(-\frac{\varpi^i}{u})$ is of level $k$, $\mu_2^{-1}(\alpha\varpi^{-i}(1-\varpi^{k-i}u))$ is of level $i<k$ (this is neither additive nor multiplicative character of u, 
we understand its level to mean the extent of the function being locally constant), and $\psi(-\alpha\varpi^{-i}(1-\varpi^{k-i}u))$ is of level $2i-k-v(\alpha)$.
If $2i-k-v(\alpha)\neq k$, or equvalently $v(\alpha)\neq 2i-c$, then the integral is zero purely for level reasons. 

When $v(\alpha)=2i-c$, we shall integrate $W_i$ against 1 and $\psi(\varpi^{-i}\alpha)$. First if we integrate against 1, we can switch the order of integral in $\alpha$ and $u$, as the integrals are essentially finite sums.
Then as functions of $\alpha$, $\mu_2^{-1}(\alpha\varpi^{-i}(1-\varpi^{k-i}u))$ is of level $k$ and $\psi(-\alpha\varpi^{-i}(1-\varpi^{k-i}u))$ is of level $-v(\alpha)+i=2k-i> k$. So the integral in $\alpha$ is 0. 

Now we do the integration against $\psi(\varpi^{-i}\alpha)$. Again we can change the order of the integral. Note that
\begin{equation}
\psi(-\alpha\varpi^{-i}(1-\varpi^{k-i}u))\psi(\varpi^{-i}\alpha)=\psi(\varpi^{k-2i}\alpha u).
\end{equation}
This term as a function of $\alpha$ is of level $2i-k-v(\alpha)=k$. Let $C(\psi,\mu_2^{-1},-k)$ denotes the non-zero integral $\int\limits_{x\in O_F^*}\psi(\varpi^{-k}x)\mu_2^{-1}(\varpi^{-k}x)d^*x$. Then
\begin{align}
\int\limits_{v(\alpha)=2i-c}\psi(\varpi^{k-2i}\alpha u)\mu_2^{-1}(\alpha\varpi^{-i}(1-\varpi^{k-i}u))d^*\alpha&=C(\psi,\mu_2^{-1},-k)\mu_2^{-1}(\frac{\varpi^{-i}(1-\varpi^{k-i}u)}{u\varpi^{k-2i}})\\
&=C(\psi,\mu_2^{-1},-k)\mu_2^{-1}(\frac{\varpi^{i-k}(1-\varpi^{k-i}u)}{u}).\notag
\end{align}
Note that $q^{v(\alpha)/2-i}q^{2i-k-v(\alpha)}=1$ when $v(\alpha)=2i-c$. So 
\begin{align}
&\text{\ \ \ }\int\limits_{v(\alpha)=2i-c}W_i(\alpha)\psi(\varpi^{-i}\alpha)d^*\alpha\\
&=\int\limits_{u\in O_F^*}\mu_1^{-1}(-\frac{\varpi^i}{u})C(\psi,\mu_2^{-1},-k)\mu_2^{-1}(\frac{\varpi^{i-k}(1-\varpi^{k-i}u)}{u})du\notag\\
&=C(\psi,\mu_2^{-1},-k)\mu_1^{-1}(-\varpi^i)\mu_2^{-1}(\varpi^{i-k})\int\limits_{u\in O_F^*}\mu_1\mu_2(u)\mu_2^{-1}(1-\varpi^{k-i}u)du.
\end{align}
Recall $\mu_1\mu_2=w_{\pi}$ is unramified, so $\mu_1\mu_2(u)=1$. Then by Lemma \ref{highlevelcharint}, we get
\begin{equation}\label{highrami<k}
\int\limits_{v(\alpha)=2i-c}W_i(\alpha)\psi(\varpi^{-i}\alpha)d^*\alpha=C(\psi,\mu_2^{-1},-k)w_{\pi}(\varpi^{k-i})\mu_1^{-1}(-\varpi^{k})\begin{cases}
1-q^{-1}, &\text{\ \ if\ }i=0;\\
-q^{-1}, &\text{\ \ if\ }i=1<k;\\
0, &\text{\ \ otherwise}.
\end{cases}
\end{equation}
In the case $k<i\leq c$, the domain for $m$ is $v(m)=v(\alpha)-k$. Write $m=\varpi^{-k}\alpha u$ for $u\in O_F^*$. The integral is
\begin{equation*}
W_i(\alpha)=\int\limits_{u\in O_F^*} \mu_1^{-1}(-\frac{\varpi^k}{1+u\varpi^{i-k}})\mu_2^{-1}(-\varpi^{-k}\alpha u)q^{-v(\alpha)/2}\psi(\varpi^{-k}\alpha u)du.
\end{equation*}
As functions of $u$, $\mu_1^{-1}(-\frac{\varpi^k}{1+u\varpi^{i-k}})$ is of level $2k-i<k$, $\mu_2^{-1}(-\varpi^{-k}\alpha u)$ is multiplicative of level $k$ and $\psi(-\varpi^{-k}\alpha u)$ is additive of level $k-v(\alpha)$. So if $v(\alpha)\neq 0$, the integral will be zero for level reason. When $v(\alpha)=0$ and $i=c$, $\mu_1^{-1}(-\frac{\varpi^k}{1+u\varpi^{i-k}})=\mu_1^{-1}(-\varpi^k)$ as $\mu_1$ is level $k$. Then
\begin{equation}\label{highramW1}
W_c(1)=\mu_1^{-1}(-\varpi^k)\int\limits_{u\in O_F^*} \mu_2^{-1}(-\varpi^{-k}\alpha u)\psi(\varpi^{-k}\alpha u)du=C(\psi,\mu_2^{-1},-k)\mu_1^{-1}(\varpi^k)(1-q^{-1}).
\end{equation}

We have used $\mu_1\mu_2(-1)=1$ here. Now for $k<i\leq c$ we integrate $W_i$ against $\psi(\varpi^{-i}\alpha)$ when $v(\alpha)=0$. Note $
\psi(\varpi^{-k}\alpha u)\psi(\varpi^{-i}\alpha)=\psi(\varpi^{-i}\alpha(1+u\varpi^{i-k}))$. As functions in $\alpha$, $\mu_2^{-1}(-\varpi^{-k}\alpha u)$ is of level $k$, and $\psi(\varpi^{-i}\alpha(1+u\varpi^{i-k}))$ is of level $i>k$. Thus the integral in $\alpha$ would be zero.

Now we do the integral $\int\limits_{v(\alpha)=0}W_i(\alpha)d^*\alpha$. Note that
\begin{equation}
\int\limits_{v(\alpha)}\psi(\varpi^{-k}\alpha u))\mu_2^{-1}(-\varpi^{-k}\alpha u)  d^*\alpha=\mu_2^{-1}(-1)C(\psi,\mu_2^{-1},-k)
\end{equation}
is independent of $u$. $q^{-v(\alpha)}=1$. So
\begin{align}\label{highrami>k}
&\text{\ \ \ }\int\limits_{v(\alpha)=0}W_i(\alpha)d^*\alpha\\
&=\int\limits_{u\in O_F^*} \mu_1^{-1}(-\frac{\varpi^k}{1+u\varpi^{i-k}})
\mu_2^{-1}(-1)C(\psi,\mu_2^{-1},-k)
du\notag\\
&=\mu_1^{-1}(\varpi^k)
C(\psi,\mu_2^{-1},-k)\begin{cases}
1-q^{-1}, &\text{\ \ if\ }i=c;\\
-q^{-1},&\text{\ \ if\ }i=c-1>k;\\
0,&\text{\ \ otherwise}.
\end{cases}\notag
\end{align}

In case $i=k$, the domain for $m$ is $v(m)\leq v(\alpha)-k, m\notin \alpha\varpi^{-k}(-1+\varpi O_F)$. By substituting $m=\alpha u$, we get
\begin{align*}
W_k(\alpha)&=\int\limits_{v(m)\leq v(\alpha)-k, m\notin \alpha\varpi^{-k}(-1+\varpi O_F)} \mu_1^{-1}(-\frac{\alpha\varpi^k}{\alpha+m\varpi^k})\mu_2^{-1}(-m)|\frac{\alpha\varpi^k}{m(\alpha+m\varpi^k)}|^{1/2}\psi(m)dm\\
&=\int\limits_{v(u)\leq -k,u\notin \varpi^{-k}(-1+\varpi O_F)}\mu_1^{-1}(-\frac{\varpi^k}{1+u\varpi^k})\mu_2^{-1}(-\alpha u)|\frac{\varpi^k}{\alpha u(1+u\varpi^k)}|^{1/2}\psi(\alpha u)q^{-v(\alpha)}du.
\end{align*}
It may be no longer true that $W_k(\alpha)$ is supported at a single $v(\alpha)$. But by Lemma \ref{scpossibleI} we only care about its integral against $1$ or $\psi(\varpi^{-k}\alpha)$. We will just do these integrals for any fixed $v(\alpha)$.

Let's do the integral against $1$ first. As functions of $\alpha$, $\mu_2^{-1}(\alpha u)$ is of level $k$ and $\psi(\alpha u)$ is of level $-v(u)-v(\alpha)=k-v(\alpha)$. So the integral will be zero unless $v(\alpha)=0$. When $v(\alpha)$ is zero, 
\begin{equation}
\int\limits_{v(\alpha)=0}\mu_2^{-1}(\alpha u)\psi(\alpha u)d^*\alpha=C(\psi,\mu_2^{-1},-k)\neq 0.
\end{equation}
Then
\begin{equation}
\int\limits_{v(\alpha)=0}W_k(\alpha)d^*\alpha=C(\psi,\mu_2^{-1},-k)\int\limits_{v(u)\leq -k,u\notin \varpi^{-k}(-1+\varpi O_F)}\mu_1^{-1}(\frac{\varpi^k}{1+u\varpi^k})|\frac{\varpi^k}{\alpha u(1+u\varpi^k)}|^{1/2}du.
\end{equation}
Note that we have used $\mu_1\mu_2(-1)=1$ here. Make a substitution $n=1+u\varpi^k$. The domain for $n$ is $v(n)\leq 0, n\notin 1+\varpi O_F$, and for fixed $v(n)$ the integrand is essentially $\mu_1(n)$. Then it's clear that the integral will be zero for the $v(n)<0$ part. The integral over $v(n)=0,n\notin 1+\varpi O_F$ will also be zero except when $k=1$ thus $\mu_1$ is identically $1$ on $1+\varpi O_F$. In that case, we will get
\begin{align}
\int\limits_{v(\alpha)=0}W_k(\alpha)d^*\alpha&=C(\psi,\mu_2^{-1},-k)\mu_1^{-1}(\varpi^k)\int\limits_{v(n)=0,n\notin 1+\varpi O_F}\mu_1(n)dn\\
&=-C(\psi,\mu_2^{-1},-k)\mu_1^{-1}(\varpi^k)q^{-1}.\notag
\end{align}
Note that the result here is just as expected in (\ref{highrami>k}), though it's not defined for $i=k=1$ there.

Now we do the integral against $\psi(\varpi^{-k}\alpha)$. $\psi(\alpha u)\psi(\varpi^{-k}\alpha)=\psi(\varpi^{-k}\alpha(1+\varpi^k u))$. As functions of $\alpha$, $\mu_2^{-1}(\alpha u)$ is of level k, and $\psi(\varpi^{-k}\alpha(1+\varpi^k u))$ is of level $k-(v(u)+k)-v(\alpha)$. Only when $v(\alpha)=-v(u)-k$ can the integral in $\alpha$ be non-zero. In that case
\begin{equation}
\int\limits_{v(\alpha)=-v(u)}\mu_2^{-1}(\alpha u)\psi(\varpi^{-k}\alpha(1+\varpi^k u))d^*\alpha=C(\psi,\mu_2^{-1},-k)\mu_2(\varpi^{-k}(1+\varpi^{k}u))\mu_2^{-1}(u).
\end{equation}


Note that $\mu_1^{-1}(-\frac{\varpi^k}{1+u\varpi^k})\mu_2(\varpi^{-k}(1+\varpi^{k}u))=\mu_1^{-1}(-1)\mu_1^{-1}\mu_2^{-1}(\frac{\varpi^k}{1+u\varpi^k})$ only depends on $v(u)$.  When integrating in $u$ for fixed $v(u)$, we are just integrating $\mu_2^{-1}(-u)$. We will get zero for most of time except when $k=1$, $v(u)=-1,u\notin \varpi^{-1}(-1+\varpi O_F)$ and $v(\alpha)=0$. In this case, we have $|\frac{\varpi^k}{\alpha u(1+u\varpi^k)}|^{1/2}=q^{-1}$ and
\begin{align}
&\text{\ \ \ }\int\limits_{v(\alpha)=0}W_k(\alpha)\psi(\varpi^{-k}\alpha)d^*\alpha\\
&=C(\psi,\mu_2^{-1},-k)\mu_1^{-1}(-1)\mu_1^{-1}\mu_2^{-1}(\varpi)\int\limits_{v(u)= -1,u\notin \varpi^{-1}(-1+\varpi O_F)}\mu_2^{-1}(-u)q^{-1}du\notag\\
&=-C(\psi,\mu_2^{-1},-k)\mu_1^{-1}(-1)\mu_1^{-1}\mu_2^{-1}(\varpi)\mu_2^{-1}(\varpi^{-1})q^{-1}\notag\\
&=-C(\psi,\mu_2^{-1},-k)\mu_1^{-1}(-\varpi)q^{-1}.
\end{align}
Note that this result can be expected directly from (\ref{highrami<k}), though it's not defined for $i=k=1$ there.

The last step is just to normalize $W_c(1)=1$ instead of (\ref{highramW1}), then we get the statements in the lemma.
\end{proof}
\begin{rem}
The result here is not completely analoguous to Lemma \ref{scWhittaker}. But as we have seen in the last subsection, we only care about the integral of the Whittaker function against 1 or $\psi(\varpi^{-i}\alpha)$, and the $\psi(\varpi^{-i}\alpha)$ part of $I(\alpha, r'(\zxz{1}{0}{\varpi^i}{1})f,\Phi_s)$ is zero for $i=0,1$. So we can make the same choice for the Schwartz function and $\Phi_s$ as in the last subsection, and we will get exactly same $\P (s,w,f,\Phi_s)$.
\end{rem}

\subsection{Ramification in $\Phi_s$}
In this subsection, we consider the case when $\Phi_s$ has ramification while $\pi$ is unramified. To make things simple, we assume that $\chi_{1}$ is ramified of level $c$ and $\chi_{2}$ is unramified. Note $\mu_1\mu_2(\chi_1\chi_2)|_{\F^*}=1$ implies that $\chi_{1}|_{\F ^*}$ is still unramified.

\begin{prop}\label{prop6-5}
Suppose that  $\pi$ is unramified and $\E /\F $ is inert. Suppose that  $\chi_{1}$ is ramified of level $c$, but $\chi_{1}|_{\F ^*}$ and $\chi_2$ are unramified. Pick 
$$f=char(\zxz{O_F}{\varpi^cO_F}{\varpi^{-c}O_F}{O_F})\times char(O_F^*).$$
Pick $\Phi_s$ to be the new form, that is, the $K_1(\varpi^c)-$invariant function supported on $B\zxz{1}{0}{1}{1}K_1(\varpi^c)$. Then 
\begin{align}
\P (s,w,f,\Phi_s)=\frac{\P_0}
{(1-\delta \mu_1\chi_{1,s}(\varpi)q)(1-\delta \mu_2\chi_{1,s}(\varpi)q)},
\end{align}
where $\P_0$ denotes the expression
\begin{align}\label{ramPhiP0}
&\frac{(\frac{\chi_{2,s}}{q\chi_{1,s}})^{-c}}{\mu_2-\mu_1}[\mu_2\frac{(1-(\delta\mu_2\chi_{2,s})^{c+1})-\frac{\chi_{2,s}}{q^2\chi_{1,s}}(1-(\delta\mu_2\chi_{2,s})^c)}{1-\delta\mu_2\chi_{2,s}}(1-q\delta\mu_1\chi_{1,s})\\
&-\mu_1\frac{(1-(\delta\mu_1\chi_{2,s})^{c+1})-\frac{\chi_{2,s}}{q^2\chi_{1,s}}(1-(\delta\mu_1\chi_{2,s})^c)}{1-\delta\mu_1\chi_{2,s}}(1-q\delta\mu_2\chi_{1,s})
].\notag
\end{align}
The numerators in the expression of $\P_0$ can be cancelled. When $w=1/2$, $\delta=q^{-(w/2+1/4)}=q^{-1/2}$. Then
\begin{equation}\label{ramPhiL}
\P (s,1/2,f,\Phi_s)=\frac{\P_0}
{(1-\mu_1\chi_{1}(\varpi)q^{-(2s+1/2)})(1-\mu_2\chi_{1}(\varpi)q^{-(2s+1/2)})}.
\end{equation}
The denominator is as expected and 
\begin{equation}
\P^0(s,1/2,f,\Phi_s)=\frac{\P_0}
{1+q^{-1}}.
\end{equation}
\end{prop}

Recall the local integral
\begin{equation}
\P (s,w,f,\Phi_s)=\int\limits_{ZN\backslash \GL_{2}(\F )}W(\sigma)\Delta(\sigma)^{w-1/2}\int\limits_{\GL_{2}(\F )}r'(\sigma)f(g,\det (g)^{-1})\Phi_{s}(\gamma_0 g) dgd\sigma.
\end{equation}
We've already known well the $K-$invariant Whittaker function in the formula. The choice
\begin{equation}
f=char(\zxz{O_F}{\varpi^cO_F}{\varpi^{-c}O_F}{O_F})\times char(O_F^*)
\end{equation}
is motivated by Example \ref{Grtest3}. It's clear that this Schwartz function is $K-$invariant under the Weil representation.

The key point is to work out 
\begin{equation}
I(\alpha, f, \Phi_s)=\int\limits_{g\in\GL_2,v(\det g)=v(\alpha)}f(g,\frac{\alpha}{\det g})\Phi_s(\gamma_0g)dg.
\end{equation}

Recall that when the extension $\E/\F$ is inert, $\GL_2=O_E^*\cdot B$.
Also recall $\Phi_s(\gamma_0 t g)=\chi_{1,s}(\bar{t})\chi_{2,s}(t)\Phi_s(\gamma_0g)$ for $t=a+b\sqrt{D}$. By our assumption on the level of $\chi_{1,s}$ and $\chi_{2,s}$, $\Phi_s(\gamma_0g)$ as a function of $g$ is left invariant under $O_F^*+\varpi^cO_E$. $f$ is also left invariant under $\{\zxz{a}{b}{bD}{a}|a\in O_F^*, b\in \varpi^cO_F\}\simeq O_F^*+\varpi^c O_E$. Then we have
\begin{align}\label{ramPhilocalI}
&I(\alpha, f, \Phi_s)=\frac{1}{(q+1)q^{c-1}}\sum\limits_{t\in (O_F^*+\varpi^c O_E)\backslash O_E^*}\text{\ }\int\limits_{B}\chi_{1,s}(\bar{t})\Phi_s(\gamma_0b)f(tb,\frac{\alpha}{\det tb})db\\
&=\frac{1}{(q+1)q^{c-1}}\sum\limits_{t\in (O_F^*+\varpi^c O_E)\backslash O_E^*}\int\limits_{
\begin{array}{c}
v(a_1)+v(a_2)=v(\alpha)\\
t\zxz{a_1}{m}{0}{a_2}\in \zxz{O_F}{\varpi^cO_F}{\varpi^{-c}O_F}{O_F}
\end{array}
} \chi_{1,s}(\bar{t})\Phi_s(\gamma_0\zxz{a_1}{m}{0}{a_2})|a_2|^{-1}dmd^*a_1d^*a_2.\notag
\end{align}
The coset representatives $(O_F^*+\varpi^c O_E)\backslash O_E^*$ can be chosen as 
$$\{1+b_1\sqrt{D}|b_1\in O_F/\varpi^cO_F \}\cup \{b_2+\sqrt{D}|b_2\in \varpi O_F/\varpi^c O_F \}.$$
One can easily see that this set has $(q+1)q^{c-1}$ elements. So $\Vol(O_F^*+\varpi^cO_E)=\frac{1}{(q+1)q^{c-1}}$. That's why we have $\frac{1}{(q+1)q^{c-1}}$ in front of the integral above.

We also need a lemma as an analogue of Lemma \ref{localcases}. We start with a general question, that is, when and how we can write $$\zxz{1}{0}{\sqrt{D}}{1}\zxz{a_1}{m}{0}{a_2}\zxz{1}{0}{\varpi^j}{1}=\zxz{a_1+m\varpi^j}{m}{a_1\sqrt{D}+(a_2+m\sqrt{D})\varpi^j}{a_2+m\sqrt{D}}$$ in form of $B\zxz{1}{0}{\varpi^i}{1}K_1(\varpi^c)$. We state here a result similar to Lemma \ref{LevelIwaDec}. The proof is also very similar, so we will skip it here.

\begin{lem}\label{LevelPhis}
\begin{enumerate}
\item[(1)]Suppose $i=0$.
\begin{enumerate}
\item[(1i)]If $j>0$, we need $v(\frac{a_1\sqrt{D}}{a_2+m\sqrt{D}})\leq 0$ for  $$\zxz{a_1+m\varpi^j}{m}{a_1\sqrt{D}+(a_2+m\sqrt{D})\varpi^j}{a_2+m\sqrt{D}}\in B\zxz{1}{0}{\varpi^i}{1}K_1(\varpi^c).$$
\item[(1ii)]If $j=0$, we need $\frac{a_1\sqrt{D}}{a_2+m\sqrt{D}}\notin -1+\varpi O_E$.
\end{enumerate}
Under above conditions we can write $\zxz{a_1+m\varpi^j}{m}{a_1\sqrt{D}+(a_2+m\sqrt{D})\varpi^j}{a_2+m\sqrt{D}}$ as

$$\zxz{\frac{a_1a_2}{a_1\sqrt{D}+(a_2+m\sqrt{D})\varpi^j}}{a_1+m\varpi^j-\frac{a_1a_2}{a_1\sqrt{D}+(a_2+m\sqrt{D})\varpi^j}}{0}{a_1\sqrt{D}+(a_2+m\sqrt{D})\varpi^j}\zxz{1}{0}{1}{1}\zxz{1}{-1+\frac{a_2+m\sqrt{D}}{a_1\sqrt{D}+(a_2+m\sqrt{D})\varpi^j}}{0}{1}.$$
\item[(2)]Suppose $i=c$.
\begin{enumerate}
\item[(2i)]If $j<c$, we need $\frac{a_1\sqrt{D}}{a_2+m\sqrt{D}}\in \varpi^j(-1+\varpi^{c-j}O_E)$.
\item[(2ii)]If $j=c$, we need $v(\frac{a_1\sqrt{D}}{a_2+m\sqrt{D}})\geq c$.
\end{enumerate}
Under above conditions, we can write $\zxz{a_1+m\varpi^j}{m}{a_1\sqrt{D}+(a_2+m\sqrt{D})\varpi^j}{a_2+m\sqrt{D}}$ as
$$\zxz{\frac{a_1a_2}{a_2+m\sqrt{D}}}{m}{0}{a_2+m\sqrt{D}}\zxz{1}{0}{\varpi^j+\frac{a_1\sqrt{D}}{a_2+m\sqrt{D}}}{1}.$$
\item[(3)]Suppose $0<i<c$.
\begin{enumerate}
\item[(3i)]If $j<i$, we need $\frac{a_1\sqrt{D}}{a_2+m\sqrt{D}}\in \varpi^j(-1+\varpi^{i-j}O_E^*)$.
\item[(3ii)]If $j=i$, we need $v(\frac{a_1\sqrt{D}}{a_2+m\sqrt{D}})\geq i$, but $\frac{a_1\sqrt{D}}{a_2+m\sqrt{D}}\notin \varpi^j(-1+\varpi O_E)$.
\item[(3iii)]If $j>i$, we need $v(\frac{a_1\sqrt{D}}{a_2+m\sqrt{D}})= i$.
\end{enumerate}
Under above conditions, we can write $\zxz{a_1+m\varpi^j}{m}{a_1\sqrt{D}+(a_2+m\sqrt{D})\varpi^j}{a_2+m\sqrt{D}}$ as
$$\zxz{\frac{a_1a_2}{a_2+m\sqrt{D}}}{m\frac{a_1\sqrt{D}+(a_2+m\sqrt{D})\varpi^j}{(a_2+m\sqrt{D})\varpi^i}}{0}{\frac{a_1\sqrt{D}+(a_2+m\sqrt{D})\varpi^j}{\varpi^i}}\zxz{1}{0}{\varpi^i}{1}\zxz{1}{0}{0}{\frac{a_1\sqrt{D}+(a_2+m\sqrt{D})\varpi^j}{\varpi^i(a_2+m\sqrt{D})}}.$$
\end{enumerate}
\end{lem}

\begin{cor}\label{ramPhi}
Assume that  $\chi_{2,s}$ is unramified and $\chi_{1,s}$ is ramified of level $c$, such that $\chi_{1,s}|_{\F ^*}$ is still unramified. Suppose that  $\Phi_s$ is the unique $K_1(\varpi^c)-$invariant function supported on $B\zxz{1}{0}{1}{1}K_1(\varpi^c)$. Then $\Phi_s(\zxz{1}{0}{\sqrt{D}}{1}\zxz{a_1}{m}{0}{a_2})$ is non-zero when $v(a_1)\leq v(a_2+m\sqrt{D})$. In that case, it's equal to $$\chi_{1,s}(\frac{a_2}{\sqrt{D}})\chi_{2,s}(a_1\sqrt{D})=\frac{\chi_{1,s}^{v(a_2)}(\varpi)\chi_{2,s}^{v(a_1)}}{\chi_{1,s}(\sqrt{D})}.$$
\end{cor}

For each representative $t\in(O_F^*+\varpi^c O_E)\backslash O_E^*$, we decide now the domain of the integral which is given by the condition $$t\zxz{a_1}{m}{0}{a_2}\in \zxz{O_F}{\varpi^cO_F}{\varpi^{-c}O_F}{O_F}.$$ If $t=1+b_1\sqrt{D}$, then by $$\zxz{1}{b_1}{b_1D}{1}\zxz{a_1}{m}{0}{a_2} \in \zxz{O_F}{\varpi^cO_F}{\varpi^{-c}O_F}{O_F}$$ we get $v(a_1),v(a_2)\geq 0$, $m\equiv -b_1a_2\mod{(\varpi^c)}$. 
Similarly for $t=\sqrt{D}+b_2$, 
$$\zxz{b_2}{1}{D}{b_2}\zxz{a_1}{m}{0}{a_2} \in \zxz{O_F}{\varpi^cO_F}{\varpi^{-c}O_F}{O_F}.$$ 
The domain will be $v(a_1)\geq -v(b_2),v(m)\geq 0, a_2\equiv -mb_2\mod{(\varpi^c)}$.

The key observation here is that although the domain depends on the specific choice of $b_1$ or $b_2$, the integral of $\Phi_s$ over the domain only depends on $v(b_1)$ and $v(b_2)$. Indeed by Corollary \ref{ramPhi}, the requirement that $v(a_1)\leq v(a_2+m\sqrt{D})$ and the value of $\Phi_s$ both depend only on the valuations of $a_1,a_2$ and $m$. The domains differ slightly but the different parts have the same volume.

More specifically for fixed $v(b_1)$ or $v(b_2)$, let $t_0$ be a fixed representative for $1+b_1\sqrt{D}$ or $b_2+\sqrt{D}$. Then we have
$$\int\limits_{B}\chi_{1,s}(\bar{t})\Phi_s(\gamma_0b)f(tb,\frac{\alpha}{\det tb})db=\chi_{1,s}(\bar{t})\int\limits_{B}\Phi_s(\gamma_0b)f(t_0b,\frac{\alpha}{\det t_0b})db .$$
So when we sum over $t$ for fixed $v(b_1)$ or $v(b_2)$, we are essentially just summing $\chi_{1,s}(\bar{t})$. Then we need a lemma similar to Lemma \ref{highlevelcharint}:
\begin{lem}\label{ramPhicha}
Let $\chi$ be a character of level $c$ on $\E ^*$ which is unramified when restricted to $\F ^*$. 
\begin{enumerate}
\item[(1)]If $c\geq 2$, we have
\begin{equation}
\sum\limits_{b_1 \in O_F/\varpi^cO_F, v(b_1)=i}\chi(1+b_1\sqrt{D})=\begin{cases}
1, \text{\ \ if\ }i=c;\\
-1, \text{\ \ if\ }i=c-1;\\
0, \text{\ \ otherwise}.
\end{cases}
\end{equation}
\begin{equation}
\sum\limits_{b_2 \in O_F/\varpi^cO_F, v(b_2)=i}\chi(\sqrt{D}+b_2)=\chi(\sqrt{D})\begin{cases}
1, \text{\ \ if\ }i=c;\\
-1, \text{\ \ if\ }i=c-1;\\
0, \text{\ \ otherwise}.
\end{cases}
\end{equation}\\
\item[(2)]If $c=1$, we have
\begin{equation}
\sum\limits_{b_1\in O_F/\varpi O_F}\chi(1+b_1\sqrt{D})+\chi(\sqrt{D})=0.
\end{equation}
\end{enumerate}
\end{lem}
\begin{proof}
(1)As $\chi$ is of level $c$ but unramified on $\F ^*$, there exist $x_0\in \varpi^{c-1}O_F/\varpi^c O_F$ such that $\chi(1+x_0\sqrt{D})\neq 1$. Then one can imitate the proof of Lemma \ref{highlevelcharint} and prove the statements for $i>0$. Note $$\chi(\sqrt{D}+b_2)=\chi(\sqrt{D})\chi(1+\frac{b_2}{D}\sqrt{D}).$$ 
For $i=0$, we have
\begin{equation}\label{ramPhicharint1}
0=\sum\limits_{(O_E/\varpi^cO_E)^*}\chi(a+b\sqrt{D})=(q-1)q^{c-1}[\sum\limits_{b_1\in O_F/\varpi^cO_F}\chi(1+b_1\sqrt{D})+\sum\limits_{b_2\in \varpi O_F/\varpi^cO_F}\chi(\sqrt{D}+b_2)].
\end{equation}
Then we use previous results for $i\geq 1$ and we get the case for $i=0$.

(2)When $c=1$, (\ref{ramPhicharint1}) implies 
\begin{equation}
0=(q-1)[\sum\limits_{b_1\in O_F/\varpi O_F}\chi(1+b_1\sqrt{D})+\sum\limits_{b_2\in \varpi O_F/\varpi O_F}\chi(\sqrt{D}+b_2)]=(q-1)[\sum\limits_{b_1\in O_F/\varpi O_F}\chi(1+b_1\sqrt{D})+\chi(\sqrt{D})].
\end{equation}

\end{proof}

By this lemma, we can rewrite (\ref{ramPhilocalI}) as 
\begin{align}\label{ramPhilocalI2}
I(\alpha, f, \Phi_s)=&\frac{1}{(q+1)q^{c-1}}[(\int\limits_{B} \Phi_s(\gamma_0g)f(g,\frac{\alpha}{\det g})dg-\int\limits_{B}\Phi_s(\gamma_0g)f(\zxz{1}{\varpi^{c-1}}{\varpi^{c-1}D}{1}g,\frac{\alpha}{\det g})dg)\\
&+ \chi_{1,s}(\sqrt{D})(\int\limits_{B} \Phi_s(\gamma_0g)f(\zxz{0}{1}{D}{0}g,\frac{\alpha}{\det g})dg-\int\limits_{B} \Phi_s(\gamma_0g)f(\zxz{\varpi^{c-1}}{1}{D}{\varpi^{c-1}}g,\frac{\alpha}{\det g})dg)].\notag
\end{align}
Here we have chosen $\zxz{1}{\varpi^{c-1}}{\varpi^{c-1}D}{1}$ as a representative for $1+b_1\sqrt{D}$ with $v(b_1)=c-1$, and $\zxz{\varpi^{c-1}}{1}{D}{\varpi^{c-1}}$ as a representative for $\sqrt{D}+b_2$ with $v(b_2)=c-1$. Any other choices will give the same result. This formula is true even if $c=1$. Indeed for $c=1$,
$$\zxz{1}{\varpi^{c-1}}{\varpi^{c-1}D}{1}=\zxz{\varpi^{c-1}}{1}{D}{\varpi^{c-1}}=\zxz{1}{1}{D}{1}.$$
According to (2) of Lemma \ref{ramPhicha}, 
$$\sum\limits_{b_1\in (O_F/\varpi O_F)^*}\chi(1+b_1\sqrt{D})=-\chi(1)-\chi(\sqrt{D}).$$ 
So the sum over $b_1\in (O_F/\varpi O_F)^*$ gives the two integrals with the minus signs.

The way we are going to do the computation is to do the substraction by comparing the domains of the above integrals. Recall for $t=1+b_1\sqrt{D}$, the domain is $v(a_1),v(a_2)\geq 0, m\equiv -b_1a_2\mod{(\varpi^c)}$. Consider the change of the domain from $v(b_1)=c-1$ to $v(b_1)=c$. We have an additional part $$\{v(a_1)\geq 0, v(a_2)=0, v(m)\geq c\},$$ 
but we lose the part 
$$\{v(a_1)\geq 0, v(a_2)=0, m\equiv -b_1a_2 \mod{(\varpi^c)}\}.$$ 
Note that we need $v(a_1)\leq v(a_2+m\sqrt{D})$ for $\Phi_s$ to be non-zero, which forces $v(a_1)=0$ for these two parts. When this is satisfied,  $$\Phi_s=\frac{\chi_{1,s}^{v(a_2)}(\varpi)\chi_{2,s}^{v(a_1)}}{\chi_{1,s}(\sqrt{D})}$$ is the same on both parts, and the measure of both parts are the same. In particular, the first pair of integral in (\ref{ramPhilocalI2}) cancels.

Now for $t=\sqrt{D}+b_2$, the domain is $v(a_1)\geq -v(b_2),v(m)\geq 0, a_2\equiv -mb_2\mod{(\varpi^c)}$. When changing from $v(b_2)=c-1$ to $v(b_2)=c$, the additional parts we have are $$\{v(a_1)\geq -c+1, v(m)=0, v(a_2)\geq c\}\text{,\ \ }\{v(a_1)=-c, v(m)\geq 0, v(a_2)\geq c\}.$$ 
The part lost is $$\{v(a_1)\geq -c+1, v(m)=0, a_2\equiv -m\varpi^{c-1}\mod{(\varpi^c)}\}.$$ We still need $v(a_1)\leq v(a_2+m\sqrt{D})$ for $\Phi_s$ to be non-zero. Suppose $v(a_1)+v(a_2)=v(\alpha)$ is fixed. We only need to consider $v(\alpha)\geq 0$ as the Whittaker function will be zero when $v(\alpha)<0$.
The contribution of the first additional part is 
\begin{equation}
\begin{cases}
0,\text{\ \ if\ }v(\alpha)=0;\\
(1-q^{-1})\sum\limits_{v(a_1)=-c+1}^{v(\alpha)-c}\frac{\chi_{1,s}^{v(\alpha)-v(a_1)}(\varpi)\chi_{2,s}^{v(a_1)}}{\chi_{1,s}(\sqrt{D})}q^{(v(\alpha)-v(a_1))},\text{\ \ if\ }1\leq v(\alpha)\leq c-1;\\
(1-q^{-1})\sum\limits_{v(a_1)=-c+1}^{0}\frac{\chi_{1,s}^{v(\alpha)-v(a_1)}(\varpi)\chi_{2,s}^{v(a_1)}}{\chi_{1,s}(\sqrt{D})}q^{(v(\alpha)-v(a_1))},\text{\ \ if\ }v(\alpha)\geq c.
\end{cases}
\end{equation}
The contribution of the second additional part is
\begin{equation}
\frac{\chi_{1,s}^{v(\alpha)+c}(\varpi)\chi_{2,s}^{-c}}{\chi_{1,s}(\sqrt{D})}q^{v(\alpha)+c}.
\end{equation}
When $v(\alpha)\geq 0$, the domain of the lost part is non-empty if and only if $0\leq v(\alpha)\leq c-1$. Then its contribution is
\begin{equation}
(1-q^{-1})\frac{1}{q-1}\frac{\chi_{1,s}^{c-1}(\varpi)\chi_{2,s}^{v(\alpha)-c+1}}{\chi_{1,s}(\sqrt{D})}q^{c-1}.
\end{equation}
From these results, one can then easily get the following lemma:
\begin{lem}
$$I(\alpha,f,\Phi_s)=\frac{1}{(q+1)q^{c-1}}\begin{cases}
(q\chi_{1,s})^{v(\alpha)}((\frac{\chi_{2,s}}{q\chi_{1,s}})^{-c}-(\frac{\chi_{2,s}}{q\chi_{1,s}})^{v(\alpha)-c+1})\frac{1-\frac{\chi_{2,s}}{q^2\chi_{1,s}}}{1-\frac{\chi_{2,s}}{q\chi_{1,s}}},&\text{\ \ if\ }0\leq v(\alpha)\leq c-1;\\
(q\chi_{1,s})^{v(\alpha)}(\frac{\chi_{2,s}}{q\chi_{1,s}})^{-c}\frac{1-\frac{\chi_{2,s}}{q^2\chi_{1,s}}}{1-\frac{\chi_{2,s}}{q\chi_{1,s}}}-(1-q^{-1})(q\chi_{1,s})^{v(\alpha)}\frac{\frac{\chi_{2,s}}{q\chi_{1,s}}}{1-\frac{\chi_{2,s}}{q\chi_{1,s}}},&\text{\ \ if\ }v(\alpha)\geq c.
\end{cases}$$
\end{lem}
Then we follow the steps in section 5 to work out the local integrals (skipping some tedious calculations):
\begin{align}
\sum\limits_{v(\alpha)=0}^{+\infty}\delta_i^{v(\alpha)}I(\alpha,f,\Phi_s)&=\frac{(1-\frac{\chi_{2,s}}{q^2\chi_{1,s}})(\frac{\chi_{2,s}}{q\chi_{1,s}})^{-c}+(\frac{\chi_{2,s}}{q^2\chi_{1,s}}-\delta_i\chi_{2,s})(\delta_iq\chi_{1,s})^c}
{(1-\delta_iq\chi_{1,s})(1-\delta_i\chi_{2,s})}\\
&=\frac{(\frac{\chi_{2,s}}{q\chi_{1,s}})^{-c}[(1-(\delta_i\chi_{2,s})^{c+1})-\frac{\chi_{2,s}}{q^2\chi_{1,s}}(1-(\delta_i\chi_{2,s})^c)]}
{(1-\delta_iq\chi_{1,s})(1-\delta_i\chi_{2,s})}.\notag
\end{align}
Though we have an additional denominator $(1-\delta_i\chi_{2,s})$ here, it can be cancelled from the numerator of the second expression. But we will keep it there as cancelling it will make the expression more complicated. Then
\begin{align}
\P (s,w,f,\Phi_s)&=\frac{\mu_1^{-1}}{\mu_1^{-1}-\mu_2^{-1}}\sum\limits_{v(\alpha)=0}^{\infty}\delta_2^{v(\alpha)}I(\alpha, f, \Phi_s)-\frac{\mu_2^{-1}}{\mu_1^{-1}-\mu_2^{-1}}\sum\limits_{v(\alpha)=0}^{\infty}\delta_1^{v(\alpha)}I(\alpha, f, \Phi_s)\\
&=\frac{\P_0}
{(1-\delta \mu_1\chi_{1,s}(\varpi)q)(1-\delta \mu_2\chi_{1,s}(\varpi)q)},\notag
\end{align}

where $\P_0$ denotes the expression
\begin{align}\label{ramPhiP0'}
&\frac{(\frac{\chi_{2,s}}{q\chi_{1,s}})^{-c}}{\mu_2-\mu_1}[\mu_2\frac{(1-(\delta\mu_2\chi_{2,s})^{c+1})-\frac{\chi_{2,s}}{q^2\chi_{1,s}}(1-(\delta\mu_2\chi_{2,s})^c)}{1-\delta\mu_2\chi_{2,s}}(1-q\delta\mu_1\chi_{1,s})\\
&-\mu_1\frac{(1-(\delta\mu_1\chi_{2,s})^{c+1})-\frac{\chi_{2,s}}{q^2\chi_{1,s}}(1-(\delta\mu_1\chi_{2,s})^c)}{1-\delta\mu_1\chi_{2,s}}(1-q\delta\mu_2\chi_{1,s})
].\notag
\end{align}
Again the numerators in the expression of $\P_0$ can be cancelled. When $w=1/2$, $\delta=q^{-(w/2+1/4)}=q^{-1/2}$, so
\begin{equation}\label{ramPhiL'}
\P (s,1/2,f,\Phi_s)=\frac{\P_0}
{(1-\mu_1\chi_{1}(\varpi)q^{-(2s+1/2)})(1-\mu_2\chi_{1}(\varpi)q^{-(2s+1/2)})}.
\end{equation}

\subsection{Joint ramification}
In this section we give an incomplete study of the local integral when both $\pi(\mu_1,\mu_2)$ and $\Phi_s$ are ramified. Suppose $\E /\F $ is inert. For simplicity, we consider the following special situation: for $\hat{\pi}=\pi(\mu_1^{-1},\mu_2^{-1})$, assume that $\mu_1$ is unramified and $\mu_2$ is of level $c>0$. For $\Phi_s$, assume that $\chi_{2}$ is unramified and $\chi_{1}$, $\chi_{1}|_{\F ^*}$ are both ramified of level $c$. By the condition that $\mu_1\mu_2\chi_{1}\chi_{2}=1$,  $\mu_2\chi_{1}|_{\F ^*}$ is also unramified. 
\begin{prop}\label{prop6-6}
Suppose that $\mu_1$ is unramified and $\mu_2$ is of level $c>0$ for the principal series $\pi(\mu_1,\mu_2)$. Suppose that $\chi_{2}$ is unramified and $\chi_{1}$, $\chi_{1}|_{\F ^*}$ are both ramified of level $c$ for $\Phi_s$. Assume that $\E /\F $ is inert. Pick
$$f=char(\zxz{1+\varpi^c O_F}{O_F}{\varpi^c O_F}{O_F})\times char(1+\varpi^c O_F).$$ 
Pick $\Phi_s$ to be the unique up to constant $K_1(\varpi^c)-$invariant function supported on $B\zxz{1}{0}{1}{1} K_1(\varpi^c)$. Then
\begin{align}
\P (s,w,f,\Phi_s)=\frac{1}{(q-1)^2(q^2-1)q^{4c-4}\chi_{1,s}(\sqrt{D})}\frac{1}{1-q\delta \mu_2\chi_{1,s}}.
\end{align}
Here $\delta=q^{-(w/2+1/4)}$. When $w=1/2$, we have
\begin{equation}\label{jointramL}
\P (s,1/2,f,\Phi_s)=\frac{1}{(q-1)^3(q+1)q^{4c-4}\chi_{1}(\sqrt{D})}\frac{1}{1-\mu_2\chi_{1}q^{-(2s+1/2)}}.
\end{equation}
The denominator is as expected, and
\begin{equation}
 \P^0(s,1/2,f,\Phi_s)=\frac{1}{(q-1)^3(q+1)^2q^{4c-5}\chi_{1}(\sqrt{D})}.
\end{equation}

\end{prop}

First note that our choice for $f$ is only $K_1^1(\varpi^c)-$invariant under the Weil representation, where $K_1^1(\varpi^c)$ is the subgroup of $K$ whose elements are congruent to $\zxz{1}{*}{0}{1} \mod{(\varpi^c)}$.
\begin{lem}
\begin{equation}
\GL_2=\coprod\limits_{0\leq i\leq c, \beta \in (O_F/\varpi^{\min\{i,c-i\}} O_F)^*}B\zxz{1}{0}{\varpi^i}{1}\zxz{\beta}{0}{0}{1}K_1^1(\varpi^c).
\end{equation}
\end{lem}
\begin{proof}First of all, $$K_1(\varpi^c)=\coprod\limits_{\beta\in (O_F/\varpi^c O_F)^*}\zxz{\beta}{0}{0}{1}K_1^1(\varpi^c).$$

We know by Lemma (\ref{Iwasawadecomp}) $$\GL_2=\coprod\limits_{0\leq i\leq c}B\zxz{1}{0}{\varpi^i}{1}K_1(\varpi^c).$$  
We just need to check when $B\zxz{1}{0}{\varpi^i}{1}\zxz{\beta}{0}{0}{1}K_1^1(\varpi^c)=B\zxz{1}{0}{\varpi^i}{1}\zxz{\beta'}{0}{0}{1}K_1^1(\varpi^c)$. 

This is equvalent to that when modulo $\varpi^c$,  $$\zxz{1}{0}{\varpi^i}{1}\zxz{\beta/\beta'}{*}{0}{1}\zxz{1}{0}{-\varpi^i}{1}$$ 
is upper triangular. That is $$\varpi^i\beta/\beta'-\varpi^i-\varpi^{2i}*\equiv 0 \mod{(\varpi^c)}.$$ 
Then the conclusion is clear.
\end{proof}
We pick $\varphi\in \pi(\mu_1^{-1},\mu_2^{-1})$ to be the unique $K_1(\varpi^c)-$ invariant function supported on $BK_1(\varpi^c)$. Then by the above lemma, one could expect the local integral $\P (s,w,f,\Phi_s)$ to be 
\begin{equation}\label{jointramtarget}
\sum\limits_{0\leq i\leq c,\beta}A_{i,\beta}\int W (\zxz{\alpha}{0}{0}{1}\zxz{1}{0}{\varpi^i}{1})|\alpha|^{\frac{w}{2}-\frac{1}{4}}\Phi_s(\alpha)^{-1} I(\alpha, r'(\zxz{1}{0}{\varpi^i}{1}\zxz{\beta}{0}{0}{1})f,\Phi_s)d^*\alpha,
\end{equation}
where $$A_{i,\beta}=\frac{A_i}{\sharp(O_F/\varpi^{\min\{i,c-i\}} O_F)^*}.$$
We shall work out the integral $I(\alpha, r'(\zxz{1}{0}{\varpi^i}{1}\zxz{\beta}{0}{0}{1})f,\Phi_s)$ first.
\begin{equation}
r'(\zxz{\beta}{0}{0}{1})f=char(\zxz{\beta^{-1}+\varpi^c O_F}{O_F}{\varpi^c O_F}{O_F})\times char(\beta+\varpi^c O_F).
\end{equation}
Then by Lemma (\ref{lemofexoticSchwartz}) and the remark after it,
\begin{equation}
r'(\zxz{1}{0}{\varpi^i}{1}\zxz{\beta}{0}{0}{1})f=q^{2(i-c)}\psi(u\varpi^{-i}[(x_1-\beta^{-1})x_4-x_2x_3]) char(\zxz{\beta^{-1}+\varpi^i O_F}{\varpi^{i-c}O_F}{\varpi^i O_F}{\varpi^{i-c}O_F})\times  char(\beta+\varpi^c O_F).
\end{equation}

\begin{lem}\label{jointramI}
\begin{equation*}
\int\Phi_s(\gamma_0g)f(g,\frac{\alpha}{\det g})dg=\frac{1}{(q-1)(q^2-1)q^{3c-3}}\chi_{1,s}(\frac{\alpha}{\sqrt{D}})q^{v(\alpha)} \text{\ \ for\ }v(\alpha)\geq 0,
\end{equation*}

\begin{equation*}
\int\Phi_s(\gamma_0g)r'(\zxz{1}{0}{\varpi^i}{1}\zxz{\beta}{0}{0}{1})f(g,\frac{\alpha}{\det g})dg=0 \text{\ \ for any $i<c$ and $\beta$}.
\end{equation*}
\end{lem}
\begin{proof}
Recall we can write $\GL_2=O_E^*B$. $\Phi_s(\gamma_0g)$ and $r'(\zxz{1}{0}{\varpi^i}{1}\zxz{\beta}{0}{0}{1})f$ are both left invariant under $1+\varpi^c O_E$. Note that $O_E^*/1+\varpi^cO_E\simeq (O_E/\varpi^cO_E)^*$ is of cardinality $(q^2-1)q^{2c-2}$. Then the integral in (1) is 
\begin{align}
\frac{1}{(q^2-1)q^{2c-2}}\sum\limits_{t\in (O_E/\varpi^cO_E)^*}\int& \Omega(t)\Phi_s(\gamma_0 \zxz{a_1}{m}{0}{a_2})r'(\zxz{1}{0}{\varpi^i}{1}\zxz{\beta}{0}{0}{1})\\
&f(t\zxz{a_1}{m}{0}{a_2}, \frac{\alpha}{N(t)a_1a_2})|a_2|^{-1}dmd^*a_1d^*a_2.\notag
\end{align}
If we write $t=b_1+b_2\sqrt{D}=\zxz{b_1}{b_2}{b_2D}{b_1}$, then $N(t)=b_1^2-b_2^2D$. 
$$\Omega(t)=\chi_{1,s}(\bar{t})\chi_{2,s}(t)=\chi_{1,s}(b_1-b_2\sqrt{D}),$$
as $\chi_{2,s}$ is unramified.
To satisfy $$t\zxz{a_1}{m}{0}{a_2} 
\in \zxz{1+\varpi^c O_F}{O_F}{\varpi^c O_F}{O_F},$$ we need $a_1b_1\in 1+\varpi^cO_F$, $a_1b_2D\in \varpi^cO_F$, and $a_2,m\in O_F$. 
If $b_2\notin \varpi^cO_F$, then it's impossible for $a_1$ to satisfies first two conditions. Thus we only need to consider those $t$ with $b_2\in \varpi^cO_F, b_1\in (O_F/\varpi^cO_F)^*$. Then the domain for the integral is $$a_1\equiv b_1^{-1}+\varpi^cO_F\text{\ \ \ } m\in O_F\text{\ \ and\ \  }a_2\in \frac{\alpha}{b_1}(1+\varpi^cO_F),$$ 
as we also need $\frac{\alpha}{N(t)a_1a_2}\in 1+\varpi^cO_F$. 

By Lemma \ref{LevelPhis}, in particular by (1i), $$\Phi_s(\zxz{1}{0}{\sqrt{D}}{1}\zxz{a_1}{m}{0}{a_2})=\chi_{1,s}(\frac{a_2}{\sqrt{D}})\chi_{2,s}(a_1\sqrt{D}),$$ 
when the condition $v(a_1)\leq v(a_2+m\sqrt{D})$ is satisfied. 
In particular over the domain we give above, $$\Phi_s=\begin{cases}
\chi_{1,s}(\frac{\alpha}{b_1\sqrt{D}}),&\text{\ \ if\ }v(\alpha)\geq 0;\\
0,&\text{\ \ otherwise},
\end{cases}$$
as $\chi_{2,s}$ is unramified and $v(a_1\sqrt{D})=0$.
Also $\Omega(t)=\chi_{1,s}(b_1-b_2\sqrt{D})=\chi_{1,s}(b_1)$.
Then for $v(\alpha)\geq 0$ the integral is easily computed to be
\begin{align}
&\text{\ \ \ }\frac{1}{(q^2-1)q^{2c-2}}\sum\limits_{b_1\in (O_F/\varpi^cO_F)^*}\int\limits_{\begin{array}{c}
a_1\equiv b_1^{-1}+\varpi^cO_F\\
m\in O_F\\
a_2\in \frac{\alpha}{b_1}(1+\varpi^cO_F)
\end{array}}
\chi_{1,s}(b_1)\chi_{1,s}(\frac{\alpha}{b_1\sqrt{D}})|a_2|^{-1}dmd^*a_1d^*a_2\\
&=\frac{1}{(q-1)(q^2-1)q^{3c-3}}\chi_{1,s}(\frac{\alpha}{\sqrt{D}})q^{v(\alpha)}.\notag
\end{align}
Now suppose $0<i<c$. For any fixed $t=b_1+b_2\sqrt{D}$ and $\beta$, we can do the integral similarly. For $t\zxz{a_1}{m}{0}{a_2} 
\in \zxz{\beta^{-1}+\varpi^i O_F}{\varpi^{i-c}O_F}{\varpi^i O_F}{\varpi^{i-c}O_F}$, we require $a_1b_1\in \beta^{-1}+\varpi^i O_F$, $a_1b_2D\in \varpi^i O_F$ and $m,a_2\in \varpi^{i-c}O_F$. One would need $b_1\in O_F^*$ and $b_2\in \varpi^iO_F$ for the first two conditions to be satisfied. As a result $a_1\in b_1^{-1}\beta^{-1}+\varpi^i O_F$ and $v(a_1)=0$. Then for $\Phi_s$ to be nonvanishing, we need $v(a_2),v(m)\geq 0$. In particular, we need $v(\alpha)\geq 0$.
The last condition $\frac{\alpha}{N(t)a_1a_2}\in \beta+\varpi^cO_F$ gives the domain of $a_2$. 
Equivalently, the domain of the integral is 
$$v(m)\geq 0,\text{\ \ \ } a_2\in \frac{b_1\alpha}{(b_1^2-b_2^2D)}(1+\varpi^iO_F) \text{\ \ and\ \ } a_1\in \frac{\beta^{-1}\alpha}{(b_1^2-b_2^2D)a_2}(1+\varpi^cO_F).$$ 
Over this domain, we have
\begin{equation}
\psi(u\varpi^{-i}[(x_1-\beta^{-1})x_4-x_2x_3])=\psi(\varpi^{-i}\alpha)\psi(-\frac{\beta^{-1}\alpha\varpi^{-i}}{(b_1^2-b_2^2D)a_1a_2}(a_2b_1+mb_2D))=\psi(\varpi^{-i}\alpha)\psi(-\frac{b_1^2\alpha\varpi^{-i}}{b_1^2-b_2^2D}),
\end{equation}
which is constant over the domain. We have used here $\frac{\alpha}{N(t)a_1a_2}\in \beta+\varpi^cO_F$, $v(m)\geq 0$, and $b_2\in\varpi^iO_F$. So in particular, $$\psi(-\frac{\beta^{-1}\alpha\varpi^{-i}}{(b_1^2-b_2^2D)a_1a_2}mb_2D)=1.$$

Now we do the integral of $\chi_{1,s}(\frac{a_2}{\sqrt{D}})\chi_{2,s}(a_1\sqrt{D})$ 
over the above domain. 
When integrating in $a_1$ first, we are essentailly integrating a constant as $\chi_{2,s}$ is unramified. Then it's clear that the integral in $a_2$ is essentially
\begin{equation}
\int\limits_{a_2\in \frac{b_1\alpha}{(b_1^2-b_2^2D)}(1+\varpi^iO_F)}\chi_{1,s}(a_2)d^*a_2=0.
\end{equation}
When $i=0$, the proof is very similar. We will leave this case to the readers.
\end{proof}
According to this Lemma, we only need to compute one integral for  (\ref{jointramtarget}) and we only care about $W (\zxz{\alpha}{0}{0}{1})$.
\begin{lem}\label{jointramWhit}
Assume that  $\mu_1$ is unramified and $\mu_2$ is of level $c>0$. Suppose that  $\varphi\in \pi(\mu_1^{-1},\mu_2^{-1})$ is the unique $K_1(\varpi^c)-$ invariant function supported on $BK_1(\varpi^c)$. Then
\begin{equation}
W (\zxz{\alpha}{0}{0}{1})=\begin{cases}
q^{-v(\alpha)/2}\mu_1^{-v(\alpha)},&\text{\ \ if\ }v(\alpha)\geq 0;\\
0,&\text{\ \ if\ }v(\alpha)<0.
\end{cases}
\end{equation}
\end{lem}
\begin{proof}
By Lemma \ref{LevelIwaDec}, in particular by part (2ii), we have $$\varphi(\omega\zxz{1}{m}{0}{1}\zxz{\alpha}{0}{0}{1})=\mu_1^{-1}(-\frac{\alpha}{m})\mu_2^{-1}(-m)|\frac{\alpha}{m^2}|^{1/2},$$ when $v(m)\leq v(\alpha)-c$. Recall $W $ is the Whittaker function for $\varphi$ associated to $\psi^- $. Then 
\begin{align}
W (\zxz{\alpha}{0}{0}{1})&=\int\limits_{v(m)\leq v(\alpha)-c}\mu_1^{-1}(-\frac{\alpha}{m})\mu_2^{-1}(-m)|\frac{\alpha}{m^2}|^{1/2}\psi(m)dm\\
&=\begin{cases}
C'q^{-v(\alpha)/2}\mu_1^{-v(\alpha)},&\text{\ \ if\ }v(\alpha)\geq 0;\\
0,&\text{\ \ if\ }v(\alpha)<0.
\end{cases}
\end{align}
where $$C'=\frac{1}{q^c\mu_1^c}\int\limits_{v(m)=-c}\mu_2^{-1}(-m)\psi(m)dm$$ 
is a non-zero constant and will be cancelled after normalization.
\end{proof}
Now we combine Lemma \ref{jointramI} and \ref{jointramWhit} into (\ref{jointramtarget}). One can easily see that 
\begin{align}
\P (s,w,f,\Phi_s)&=A_c\int\limits_{v(\alpha)\geq 0} q^{-v(\alpha)/2}\mu_1^{-v(\alpha)}|\alpha|^{\frac{w}{2}-\frac{1}{4}}\Phi_s(\alpha)^{-1}\frac{1}{(q-1)(q^2-1)q^{3c-3}}\chi_{1,s}(\frac{\alpha}{\sqrt{D}})q^{v(\alpha)}d^*\alpha\\
&=\frac{1}{(q-1)^2(q^2-1)q^{4c-4}\chi_{1,s}(\sqrt{D})}\sum\limits_{v(\alpha)=0}^{+\infty}q^{-(w/2+1/4)v(\alpha)}q^{v(\alpha)}(\mu_1\chi_{2,s})^{-v(\alpha)}\notag\\
&=\frac{1}{(q-1)^2(q^2-1)q^{4c-4}\chi_{1,s}(\sqrt{D})}\frac{1}{1-q\delta \mu_1^{-1}\chi_{2,s}^{-1}}\notag\\
&=\frac{1}{(q-1)^2(q^2-1)q^{4c-4}\chi_{1,s}(\sqrt{D})}\frac{1}{1-q\delta \mu_2\chi_{1,s}}.\notag
\end{align}
Here $\delta=q^{-(w/2+1/4)}$. We have used $\mu_1\mu_2\chi_{1,s}\chi_{2,s}=1$ and $\mu_2\chi_{1,s}$ is unramified.
When $w=1/2$, we have
\begin{equation}\label{jointramL'}
\P (s,1/2,f,\Phi_s)=\frac{1}{(q-1)^2(q^2-1)q^{4c-4}\chi_{1,s}(\sqrt{D})}\frac{1}{1-\mu_2\chi_{1}q^{-(2s+1/2)}}.
\end{equation}

\section{Approach by matrix coefficients}
In this section, we consider a finite place where $\pi$ is supercuspidal, $\chi_1$ and $\chi_2$ are unramified and $\E/\F$ is inert. In general, the calculation in terms of matrix coefficient is also not easy. But this case turns out to be very simple compared with the previous method.

Again we will work mostly in the local settings and omit the subscript $v$ for all notations. Note if the level $c$ of the representation $\pi$ is odd, then the local integral is automatically zero according to Theorem \ref{Tunnell} and Example \ref{ApplyTunnell}. When the level $c$ is even, let $c=2k$. Define $\tilde{K}$ to be the subgroup of $\GL_2(O_E)$ whose elements are congruent to $\zxz{1}{0}{0}{1}\mod{(\varpi^k O_E)}$. Define $\Phi_s$ to be the unique up to constant function from the induced representation such that it's right $\tilde{K}-$invariant and supported on 
$B\zxz{0}{1}{-1}{-\frac{\sqrt{D}}{D}}\tilde{K}. $
It will be normalized such that $\Phi_s(\zxz{0}{1}{-1}{-\frac{\sqrt{D}}{D}})=1$.

As mentioned in the end of Section 4, the local integral of our problem can also be formulated in terms of matrix coefficients:
\begin{equation}\label{Matrixcoeflocaltarget}
\int\limits_{\F ^*\backslash \GL_2(\F )}\Phi_s(\gamma_0 g)< F_{1},\pi (g) F>dg,
\end{equation}
where $F_1\in \hat{\pi}$, $F\in\pi$ and $<\cdot,\cdot>$ is a bilinear and $\GL_2(\F )-$ invariant pairing between $\hat{\pi}$ and $\pi$.

We briefly recall the bilinear and $\GL_2(\F )-$ invariant pairing $<\cdot,\cdot>$ for the supercuspidal representations. 
As elements from the Kirillov moder for $\hat{\pi }$ and $\pi $, $F_{1}$ and $F$ belong to the space of Schwartz functions $S(\F ^*)$. Then define
\begin{equation}
<F_{1},F >=\int\limits_{\F ^*}F_{1}(x)F (-x)d^*x.
\end{equation}
This is indeed bilinear and $\GL_2(\F )-$ invariant (see \cite{JL70} Lemma 2.19.1).

As motivated by Example \ref{GrTest1}, we will pick $F_{1}$ and $F$ to be the unique up to constant elements from respective representations which are invariant under 
$$\{\zxz{a+\varpi^k O_F}{b+\varpi^k O_F}{bD+\varpi^k O_F}{a+\varpi^k O_F}|a+b\sqrt{D}\in O_E^* \} $$
We further assume $F_{1}$ and $F$ are so normalized that 
$$<F_{1},F>=1.$$

\begin{prop}\label{propscinert}
Suppose that $\pi $ is supercupidal of level $c$, $\chi_1$ and $\chi_2$ are both unramified, and $\E /\F $ is inert. Then for the given choices of $F_{1}\in\hat{\pi}$, $F\in\pi $ and $\Phi_s$, we have
$$\int\limits_{\F ^*\backslash \GL_2(\F )}\Phi_s(\gamma_0 g)< F_{1},\pi (g) F>dg=\frac{1}{(q-1)q^{c-1}}.$$
\end{prop}

First of all, we need a lemma about the Gross-Prasad test vector.
\begin{lem}\label{scinertGrtest}
Let $F_{1}\in S(\F ^*)$ be an element in the Kirillov model of $\hat{\pi }$. Suppose that  it's invariant under the action of $R_c^*$ where $R_c$ is the order 
$$\{\zxz{a+\varpi^k O_F}{b+\varpi^k O_F}{bD+\varpi^k O_F}{a+\varpi^k O_F}|a+b\sqrt{D}\in O_E \}.$$ 
Then $F$ is supported at $v(x)=-k$ and consists of characters of level less than or equal to $k$.
\end{lem}
\begin{proof}
To get the statement, we actually only need that $F_{1}$ is invariant under $$\{\zxz{1+\varpi^k O_F}{\varpi^k O_F}{\varpi^k O_F}{1+\varpi^k O_F}\}.$$

We write $F_{1}=\sum\limits_{n\in \Z}\sum\limits_{\nu} a_{\nu,n} \charf_{\nu,n}(x)$.  
By definition, $$\hat{\pi}(\zxz{a}{0}{0}{1})F_{1}(x)=F_{1}(ax)=\sum\limits_{n\in \Z}\sum\limits_{\nu} a_{\nu,n}\nu(a) \charf_{\nu,n}(x)$$
for any $a\in 1+\varpi^kO_F$. To be invariant, we need $a_{\nu,n}=0$ for any $\nu$ of level greater than $k$. From $$\hat{\pi}(\zxz{1}{m}{0}{1})F_{1}(x)=\psi(-mx)F_{1}(x)=F_{1}(x)$$ 
for any $m\in \varpi^kO_F$, we get $a_{\nu,n}=0$ for $n<-k$. Thus we can write
\begin{equation}
F_{1}=\sum\limits_{n\geq -k}\sum\limits_{\nu \text{\ of level}\leq k} a_{\nu,n} \charf_{\nu,n}(x).
\end{equation}
Lastly $F_{1}$ has to be invariant under $\zxz{1}{0}{m'}{1}=\omega^{-1}\zxz{1}{-m'}{0}{1}\omega$ for any $m'\in \varpi^k$. This is equivalent to that $\hat{\pi }(\omega) F_{1}$ is invariant under these $\zxz{1}{-m'}{0}{1}$, which implies $\hat{\pi }(\omega) F_{1}$ is supported at $v(x)\geq -k$. We know from (\ref{singleaction})
\begin{equation}
\hat{\pi}(\omega )F_{1}=\sum\limits_{n\geq -k}\sum\limits_{\nu \text{\ of level}\leq k} a_{\nu,n}C_{\nu w_0^{-1}}z_0^{-n}\charf_{\nu^{-1}w_0,-n+n_{\nu^{-1}}}.
\end{equation}
So we get $-n+n_{\nu^{-1}}\geq -k$ for all $\nu$ of level $\leq k$. By Lemma \ref{powerofmonomials}, $n_{\nu^{-1}}=-c$ for all such characters. Thus $-n-c\geq -k$, that is $n\leq -k$.
\end{proof}

Next we need to know the property for $\Phi_s$.
\begin{lem}\label{Phisscinert}
Let $\Phi_s$ be the unique normalized element from the induced representation which is supported on $B\zxz{0}{1}{-1}{-\frac{\sqrt{D}}{D}}\tilde{K}$. Then 
\begin{equation}
\Phi_s(\gamma_0\zxz{a_1}{m}{0}{1})=\begin{cases}
1,&\text{\ \ if\ } v(m)\geq k\text{\ and\ }a_1\equiv 1\mod{(\varpi^k)};\\
0, &\text{\ \ otherwise\ }.
\end{cases}
\end{equation}
\end{lem}
\begin{proof}
Let's consider when the matrix $\zxz{1}{0}{\sqrt{D}}{1}\zxz{a_1}{m}{0}{1}$ can be in the support $B\zxz{0}{1}{-1}{-\frac{\sqrt{D}}{D}}\tilde{K}$. This is equivalent to say if there exists $k\in \tilde{K}$ such that 
$$\zxz{a_1}{m}{a_1\sqrt{D}}{1+m\sqrt{D}}k\zxz{-\frac{\sqrt{D}}{D}}{-1}{1}{0}$$
is upper triangular. This in turn is equivalent to that
$$-a_1+1+m\sqrt{D}\equiv 0\mod{(\varpi^k)}.$$
Thus one get the conditions for $\Phi_s(\gamma_0\zxz{a_1}{m}{0}{1})$ to be non-zero as in the lemma. When these conditions are satisfied, the rest are easy to check. 
\end{proof}

Now we can prove Proposition \ref{propscinert} easily. As $\chi_1$ and $\chi_2$ are unramified, $\Phi_s(\gamma_0tg)=\Phi_s(\gamma_0g)$ for $t\in O_E^*$. Note that $\F ^*\backslash\GL_2=O_E^*\{\zxz{a_1}{m}{0}{1}\}$. Then the local integral (\ref{Matrixcoeflocaltarget}) becomes
\begin{align}\label{scinertrefinedint}
&\text{\ \ \ }\int\limits_{\F ^*\backslash \GL_2(\F )}\Phi_s(\gamma_0 g)< F_{1},\pi (g) F>dg\\
&=\int\limits_{t\in O_E^*}\int\limits_{a_1,m}\Omega (t)\Phi_s(\gamma_0\zxz{a_1}{m}{0}{1})<\hat{\pi }(t^{-1})F_{1},\pi (\zxz{a_1}{m}{0}{1})F >d^*a_1dm d^*t\notag\\
&=\int\limits_{a_1,m}\Phi_s(\gamma_0\zxz{a_1}{m}{0}{1})<F_{1},\pi (\zxz{a_1}{m}{0}{1})F >d^*a_1dm\notag\\
&=\int\limits_{v(m)\geq k\text{\ and\ }a_1\equiv 1\text{\ mod}{(\varpi^k)}}<F_{1},\pi (\zxz{a_1}{m}{0}{1})F >d^*a_1dm.\notag
\end{align}
Here we have used the fact that $F_1$ is invariant under $O_E^*$ for the second equality, and Lemma \ref{Phisscinert} for the last equality.
Recall for the Kirillov model, 
\begin{equation}\label{scinertF}
\pi (\zxz{a_1}{m}{0}{1})F(x)=\psi(-mx)F(a_1x).
\end{equation}
Now by $F$ consists of characters of level less than or equal to $k$ while $a_1\equiv 1\text{\ mod}(\varpi^k)$, we have
$$F(a_1x)=F(x).$$
By $F$ is supported at $v(x)=-k$ while $v(m)\geq k$, we have
$$\psi(-mx)=1.$$
Thus
\begin{equation}
\int\limits_{\F ^*\backslash \GL_2(\F )}\Phi_s(\gamma_0 g)< F_{1},\pi (g) F>dg=\int\limits_{v(m)\geq k\text{\ and\ }a_1\equiv 1\text{\ mod}{(\varpi^k)}}<F_{1},F >d^*a_1dm=\frac{1}{(q-1)q^{c-1}}.
\end{equation}

\section{Conclusion}
Now we return to the global story. Recall our global integral is
\begin{equation}
\I(E,F,s)=\int\limits_{Z_{\A}\GL_{2}(\F)\backslash \GL_{2}(\A_{\F})} F(g)E(g,s) dg.
\end{equation}
According to Proposition \ref{mainprop}, we either get 0 for $\I(E,F,s)$, or we can fix a non-zero period integral $C$ such that
\begin{equation}\label{primitiveprod}
C\cdot\I(E,F,s)=\prod\limits_v\P_v(s,1/2,f_v,\Phi_{s,v}),
\end{equation}
where \begin{equation}
\P_v(s,w,f_v,\Phi_{s,v})=\int\limits_{ZN\backslash \GL_{2}(\F_v)}\int\limits_{\GL_{2}(\F_v)}W_{\varphi,v}^-(\sigma)\Delta_v(\sigma)^{w-1/2}r'(\sigma)f_v(g,det(g)^{-1})\Phi_{s,v}(\gamma_0 g)d\sigma dg.
\end{equation}
The work in section 5 showed that at unramified places,  
$$\P_v(s,1/2,f_v,\Phi_{s,v})=\frac{L_v(\Pi_v\otimes\Omega_v,1/2)L_v(\pi_v\otimes\chi_{1,v}|_{\F_v^*},2s+1/2)}{L_v(\eta_v,1)L^{\E}_v(\chi_{v},2s+1)
},$$ where $\Pi$ is the base change of $\pi$ to $\E$ and $\chi=\frac{\chi_1}{\chi_2}$. For every place $v$ of $\F$, we normalized the local integral by
\begin{equation}
\P_v^0(s,1/2,f_v,\Phi_{s,v})=\frac{L_v(\eta_v,1)L^{\E}_v(\chi_{v},2s+1)
}{L_v(\Pi_v\otimes\Omega_v,1/2)L_v(\pi_v\otimes\chi_{1,v}|_{\F_v^*},2s+1/2)}\P_v(s,w,f,\Phi_s).
\end{equation}
Then Propositions \ref{propinert} and \ref{propsplit} imply that $\P_v^0(s,1/2,f_v,\Phi_{s,v})=1$ for all unramified places. 
So we can rewrite (\ref{primitiveprod}) as 
\begin{equation}
C\cdot\I(E,F,s)=\frac{L(\Pi\otimes\Omega,1/2)L(\pi\otimes\chi_{1}|_{\A_\F^*},2s+1/2)}{ L(\eta,1)L^{\E}(\chi,2s+1)}\prod\limits_{v\in S}\P_v^0(s,1/2,f_v,\Phi_{s,v})
\end{equation}
Propositions \ref{prop6-1}, \ref{prop6-2}, \ref{propsc}, \ref{prophighlyrampi}, \ref{prop6-5} and \ref{prop6-6} in Section 6 give $\P^0_v(s,1/2,f_v,\Phi_{s,v})$ for some ramified cases.
The following table gives a list of results from Section 6:
\newline
\newline

\begin{tabular}{|c|p{2.1cm}|c|c|c|c|}
\hline
Case	&$\pi_v$	&$\chi_{1,v}$ and $\chi_{2,v}$	&$\E_v/\F_v$	
&$\P_v^0(s,1/2,f,\Phi_s)$\\ \hline
1	&unramified	&unramified	&ramified	&$1$	\\ \hline
2	&unramified special 	&unramified	&split	&$\frac{1}{(q+1)^2(1-\chi_v^{(2)}q^{-(2s+1)})}$\\ \hline
3	&supercupidal or ramified principal 	&unramified	&split	&$\frac{1}{(q+1)^2q^{2c-2}(1-\chi_v^{(2)}q^{-(2s+1)})}$\\ \hline
4	&unramified	&$\chi_{1,v}$ level c	&inert	&$\frac{\P_0}
{1+q^{-1}}$ for $\P_0$ given in (\ref{ramPhiP0})\\ \hline
5	&$\mu_{2,v}$ level c	&$\chi_{1,v}$ level c	&inert	& $\frac{1}{(q-1)^3(q+1)^2q^{4c-5}\chi_{1,v}(\sqrt{D})}$\\ \hline
\end{tabular}
\newline
\newline

Recall the characters not mentioned (that is, $\mu_{1,v}$ and $\chi_{2,v}$) in cases 4 and 5 are all unramified. This implies that $\chi_{1,v}|_{\F_v^*}$ is unramified in case 4 and is of level $c$ in case 5.

We also list the choices for $f_v$ and $\Phi_{s,v}$ for the above results. Recall $\Phi_{s,v}$ is normalized in the sense that if $\Phi_{s,v}$ is $K_1(\varpi^c)-$invariant and supported on $B\zxz{1}{0}{\varpi^i}{1}K_1(\varpi^c)$, then $\Phi_{s,v}(\zxz{1}{0}{\varpi^i}{1})=1$.
\newline
\newline

\begin{tabular}{|c|c|p{5cm}|}
\hline
Case	&$f_v$	&$\Phi_{s,v}$\\ \hline
1	&$char(\zxz{O_F}{O_F}{O_F}{O_F})\times char(O_{F}^{*})$		& $K-$invariant \\ \hline
2	&$char(\zxz{1}{0}{\sqrt{D}}{1}\zxz{O_F}{O_F}{\varpi O_F}{O_F})\times char(O_F^*)$		&$\Phi_{s,v}^{(1)}$ is $K-$ invariant; $\Phi_{s,v}^{(2)}$ is $K_1(\varpi)-$invariant, supported on $BK_1(\varpi)$ \\ \hline
3	&$char(\zxz{1}{0}{\sqrt{D}}{1}\zxz{O_F}{O_F}{\varpi^cO_F}{O_F})\times char(O_F^*)$		&$\Phi_{s,v}^{(1)}$ is $K-$ invariant; $\Phi_{s,v}^{(2)}$ is $K_1(\varpi^c)-$invariant, supported on $BK_1(\varpi^c)$ \\ \hline
4	&$char(\zxz{O_F}{\varpi^cO_F}{\varpi^{-c}O_F}{O_F})\times char(O_F^*)$		&$K_1(\varpi^c)-$invariant, supported on $B\zxz{1}{0}{1}{1} K_1(\varpi^c)$ \\ \hline
5	&$char(\zxz{1+\varpi^c O_F}{O_F}{\varpi^c O_F}{O_F})\times char(1+\varpi^c O_F)$		&$K_1(\varpi^c)-$invariant, supported on $B\zxz{1}{0}{1}{1} K_1(\varpi^c)$\\ \hline
\end{tabular}
\newline
\newline

In cases 2 and 3, we didn't get the expected numerator. Unlike the denominator, the numerator depends on the specific choice of $f_v$ and $\Phi_{s,v}$. We actually have some asymmetry in our choice for $\Phi_{s,v}$. It might be possible that we can get better numerator for a symmetric choice of $\Phi_{s,v}$. But we won't do it here for the sake of concisement.

Also recall in Section 7 we discussed the case when $\pi_v$ is supercupidal, $\chi_{1,v}$ $\chi_{2,v}$ for $\Phi_{s,v}$ are unramified and $\E_v/\F_v$ is inert. We showed that the local integral in terms of the matrix coefficient is
$$\int\limits_{\F ^*\backslash \GL_2(\F )}\Phi_s(\gamma_0 g)< F_{1},\pi (g) F>dg=\frac{1}{(q-1)q^{c-1}}.$$
Here $\Phi_s$ is right $\tilde{K}-$invariant and supported on 
$B\zxz{0}{1}{-1}{-\frac{\sqrt{D}}{D}}\tilde{K} $. $F_1$ and $F$ are the Gross-Prasad test vectors as in Example \ref{GrTest1}.
\begin{theo}\label{maintheo} Suppose
\begin{enumerate}
\item[(i)] $\E$ is a quadratic algebra over $\F$, which is embedded into $M_2(\F)$ by $$a+b\sqrt{D}\mapsto \zxz{a}{b}{bD}{a};$$ 
\item[(ii)] $F$ is a cusp form for $\GL_{2}(\A_{\F})$ in an irreducible cuspidal automorphic representation $\pi$ whose central character is $w_\pi$;
\item[(iii)] $\chi_1$ and $\chi_2$ are two Hecke characters on $\E^*\backslash\E_\A^*$ such that $w_\pi\cdot(\chi_1\chi_2)|_{\A_\F^*}=1$. Define $\chi=\frac{\chi_1}{\chi_2}$. Define $\Omega(t)=\chi_1(\bar{t})\chi_2(t)$ for $t\in\E_\A^*$.
For $\Phi_s\in Ind_B^{\GL_2}(\chi_1,\chi_2,s)$, let $$E(g,s)=\sum\limits_{\gamma \in B(\E)\backslash \GL_{2}(\E)}\Phi_s(\gamma g)$$ be the associated Eisenstein series.
\end{enumerate}

(1)\quad If $\int\limits_{Z_{\A}\E^{*}\backslash \E^{*}_{\A}}F(tg)\Omega(t)dt$ is always zero, then $$\I(E,F,s)=\int\limits_{Z_{\A}\GL_{2}(\F)\backslash \GL_{2}(\A_{\F})} F(g)E(g,s) dg=0.$$

(2)\quad Otherwise we fix a nonzero period integral $C=\int\limits_{Z_{\A}\E^{*}\backslash \E^{*}_{\A}}F_{1}(t_{1})\Omega^{-1}(t_{1})dt_{1}\text{\ \ \ }$  for $F_1\in\hat{\pi}$. Then
\begin{equation}
C\cdot\I(E,F,s)=\frac{L(\Pi\otimes\Omega,1/2)L(\pi\otimes\chi_{1}|_{\A_\F^*},2s+1/2)}{ L(\eta,1)L^{\E}(\chi,2s+1)}\prod\limits_{v\in S}\P_v^0(s,1/2,f_v,\Phi_{s,v}),
\end{equation}
where $\P_v^0(s,1/2,f_v,\Phi_{s,v})$ is partially given in the tables above with the given choices for $f_v$ and $\Phi_{s,v}$.
\end{theo}

\appendix
\section{Some basic facts about the compact subgroups of $\GL_2$ over p-adic field}
Let $\F_v$ be a local p-adic field, and let $B$ be the Borel subgroup of $\GL_2$ and $K=\GL_2(O_F^*)$ be the maximal compact subgroup. Recall we denote by $K_1(\varpi^c)$ (or $K_0(\varpi^c)$) the subgroup of $K=\GL_2(O_F)$ whose elements are congruent to $\zxz{*}{*}{0}{1}$(or $\zxz{*}{*}{0}{*}$) mod ($\varpi^c$) for an integer $c>0$. Most results here are probably already known by experts.
\begin{lem} \label{Iwasawadecomp}
For every positive integer $c$,
$$\GL_2(F)=\coprod\limits_{0\leq i\leq c} B\zxz{1}{0}{\varpi^i}{1}K_1(\varpi^c).$$
\end{lem}
\begin{rem}
When $i=c$, $B\zxz{1}{0}{\varpi^c}{1}K_1(\varpi^c)=BK_1(\varpi^c)$. This lemma can be thought of as a variant of Iwasawa decomposition.
\end{rem}
\begin{proof}
First we show it's a disjoint union. For $0\leq i\neq j\leq c$, suppose 
$$\zxz{a_1}{m}{0}{a_2}\zxz{1}{0}{\varpi^i}{1}=\zxz{1}{0}{\varpi^j}{1}\zxz{k_1}{k_2}{k_3}{k_4}$$
for $\zxz{a_1}{m}{0}{a_2}\in B$ and $\zxz{k_1}{k_2}{k_3}{k_4}\in K_1(\varpi^c)$. Note $k_1,k_4\in O_F^*$ and $v(k_3)\geq c$. By equating the corresponding elements of the matrices, we get 
$$a_1+m\varpi^i=k_1, \text{\ \ \ } m=k_2,\text{\ \ \ }a_2\varpi^i=k_1\varpi^j+k_3,\text{\ \ \ }a_2=k_2\varpi^j+k_4.$$ 
Then we can get a contradiction from the last two equation.
Indeed, $\varpi^i(k_2\varpi^j+k_4)=k_1\varpi^j+k_3$ is impossible for $\zxz{k_1}{k_2}{k_3}{k_4}\in K_1(\varpi^c)$ if $i\neq j$.

Next we show that every matrix of $\GL_2$ belongs to $B\zxz{1}{0}{\varpi^i}{1}K_1(\varpi^c)$ for some $i$. Note that $\GL_2(F)=B\GL_2(O_F)$ by the standard Iwasawa decomposition. As a result of this, we only have to look at matrices of form $\zxz{x_1}{x_2}{x_3}{x_4}\in \GL_2(O_F)$. 

If $i=v(x_3)>0$, then $x_4\in O_F^*$. When $i\geq c$, we have $$\zxz{x_1}{x_2}{x_3}{x_4}=\zxz{1}{0}{0}{x_4}\zxz{x_1}{x_2}{x_3/x_4}{1}.$$ 
When $0< i< c$, we have $$\zxz{x_1}{x_2}{x_3}{x_4}=\zxz{\frac{x_1x_4-x_2x_3}{x_3}\varpi^i}{x_2}{0}{x_4}\zxz{1}{0}{\varpi^i}{1}\zxz{\frac{x_3}{x_4\varpi^i}}{0}{0}{1}.$$ 
When $i=0$ and $x_4\in O_F^*$, we can decompose $\zxz{x_1}{x_2}{x_3}{x_4}$ similarly as in the case $0<i<c$. If $x_4\notin O_F^*$, then $x_2,x_3\in O_F^*$, and $$\zxz{x_1}{x_2}{x_3}{x_4}=\frac{\det x}{x_3} \zxz{1}{\frac{x_1-x_3}{x_3}}{0}{1}\zxz{1}{0}{1}{1}\zxz{\frac{x_3^2}{\det x}}{-1+\frac{x_3x_4}{\det x}}{0}{1}.$$
\end{proof}
\begin{lem}\label{indexofcompactsub}
$[K:K_0(\varpi)]=q+1$, $[K_0(\varpi^i):K_0(\varpi^{i+1})]=q$ for $i>0$.
\end{lem}
\begin{proof}
For $[K:K_0(\varpi)]$, one can check that $$K= 
\coprod\limits_{n\in (O_F/\varpi O_F) 
}\zxz{1}{0}{n}{1}K_0(\varpi)\cup\omega K_0(\varpi).$$

In general for $[K_0(\varpi^i):K_0(\varpi^{i+1})]$, one has
$$K_0(\varpi^i)= 
\coprod\limits_{n\in(\varpi^iO_F/\varpi^{i+1}O_F) 
}\zxz{1}{0}{n}{1}K_0(\varpi^{c+1}).$$
\end{proof}
Now we will focus on the integral of a function $f$ on $\GL_2$. In particular, $f$ will be  right-invariant under $K_1(\varpi^c)$ for some integer $c>0$. By Lemma \ref{Iwasawadecomp}, the integral over $\GL_2$ can be decomposed as the sum of integrals over each $B\zxz{1}{0}{\varpi^i}{1}K_1(\varpi^c)$. Then we use that $f$ is right $K_1(\varpi^c)-$invariant and get
\begin{equation}\label{integraldecomp}
\int\limits_{g\in\GL_2}f(g)dg=\sum\limits_{0\leq i\leq c}A_i\int\limits_{b\in B}f(b\zxz{1}{0}{\varpi^i}{1})db.
\end{equation}
\begin{lem}\label{integraldecompcoeff}
$A_0=\frac{q}{q+1}$, $A_c=\frac{1}{(q+1)q^{c-1}}$, and $A_i=\frac{q-1}{(q+1)q^i}$ for $0<i<c$.
\end{lem}
\begin{proof}
For $0\leq j\leq c$, let $f_j$ be the characteristic function of $K_0(\varpi^j)$. $f_0$ is just the characteristic function of $K$. Clearly they are all right-invariant under $K_1(\varpi^c)$. The integral of these functions just give the volume  of these compact subgroups. Suppose that  the Haar measure on $\GL_2$ are so normalized that the volumes of $K$ and $B(O_F)=B\cap K$ are 1. By Lemma \ref{indexofcompactsub}, the volume of $K_0(\varpi^j)$ is $\frac{1}{(q+1)q^{j-1}}$ for $j>0$. On the other hand, we can evaulate the integral by the right hand side of (\ref{integraldecomp}). $$f_j(b\zxz{1}{0}{\varpi^i}{1})=\begin{cases}1,&\text{\ if\ }b\in B(O_F)\text{\ and}  j\leq i;\\ 0, &\text{\ otherwise}.\end{cases}$$
So $$\frac{1}{(q+1)q^{j-1}}=\int\limits_{g\in\GL_2}f_j(g)dg=\sum\limits_{0\leq i\leq c}A_i\int\limits_{b\in B}f_j(b\zxz{1}{0}{\varpi^i}{1})db=\sum\limits_{j\leq i\leq c}A_i$$
for $0<j\leq c$. When j=0, we get $$1=\sum\limits_{0\leq i\leq c}A_i.$$
Then it's easy to see that the values of the coefficients $A_i$ in the lemma are the only choice.
\end{proof}
\begin{rem}
When we integrate a $K_1(\varpi^c)-$invariant function over $ZN\backslash\GL_2$, we have a similar formula as (\ref{integraldecomp}) with the same coefficients.
\end{rem}
\section{Kirillov model for the supercuspidal representation and its newform}
 Recall $$\charf_{\nu,n}(x)=\begin{cases}
                        \nu(u), &\text{if\ } x=u\varpi^n\text{\  for\ } u\in O_F^*;\\
						0,&\text{otherwise},
                       \end{cases}
 $$
and
\begin{equation}
\hat{\pi}(\omega )\charf_{\nu,n}=C_{\nu w_0^{-1}}z_0^{-n}\charf_{\nu^{-1}w_0,-n+n_{\nu^{-1}}}.
\end{equation}
In [Ca], it is proved that for any supercuspidal representation $\pi$ with central character $w_\pi$, there is a unique explicit element $\varphi$ in it such that
\begin{equation}
\pi(\zxz{k_1}{k_2}{k_3}{k_4})\varphi=w_\pi(k_1)\varphi
\end{equation}
for any $\zxz{k_1}{k_2}{k_3}{k_4}\in K_0(\varpi^{-n_1})$. One can however easily imitate the proof and show that there is such $\varphi$ that 
\begin{equation}
\pi(\zxz{k_1}{k_2}{k_3}{k_4})\varphi=w_\pi(k_4)\varphi
\end{equation}
for $\zxz{k_1}{k_2}{k_3}{k_4}\in K_0(\varpi^{-n_{w_0^{-1}}})$. (Actually $\varphi$ is just $\charf_{1,0}$.) Note $n_{w_0^{-1}}=n_1$.  
Then we restrict the above equation to $K_1(\varpi^{-n_1})$, so $k_4\equiv 1 (mod \varpi^{-n_1})$. If one can show that the level of central charater $c(w_\pi)$ is less than or equal to $-n_1$, then $w(k_4)=1$, and $\varphi$ is $K_1(\varpi^{-n_1})-$invariant.

\begin{lem}
$n_{1}\leq \min\{-2,-c(w_\pi)-1\}$, or equivalently, $c=-n_1\geq \max\{2,c(w_\pi)+1\}$.
\end{lem}
\begin{proof}
Recall $n_\nu\leq -2$ for any character $\nu$. Suppose $-2\geq n_1>-c(w_\pi)-1$. We have relation 
\begin{equation}\label{2ndrelation}
 \omega \zxz{1}{-1}{0}{1}\omega =\zxz{1}{1}{0}{1}\omega \zxz{1}{1}{0}{1}.
\end{equation}
We will test each side on $\charf_{1,0}$. We will keep track of the levels and the supports after each action, while ignoring the coefficients as long as they are not zero.

For the left hand side, action of $\omega $ will change $\charf_{1,0}$ into a multiple of $\charf_{w_0,n_1}$, which is supported at $v(x)=n_1$. After the action of $\zxz{1}{-1}{0}{1}$, we get a finite linear combination of $\charf_{w_0 \nu,n_1}$ for all $\nu$ of level $-n_1$, as the Mellin transform of $\psi(-x)$ for $v(x)=n_1$ consist of all such $\nu$. It should always be understood that if $-n_1=1$, then $\nu$ is of level 1 or 0.
Finally we will get a finite linear combination of all those $\nu$ of level $-n_1$ as the last action of $\omega $ cancels the $w_0$ part. 

For the right hand side, the first action of $\zxz{1}{1}{0}{1}$ is trivial as $\psi$ is unramified. Then the action of $\zxz{1}{1}{0}{1}\omega $ is just like the first two steps of  the left hand side. So we get a linear combination of $\charf_{w_0 \nu,n_1}$ for all $\nu$ of level $-n_1$ for the right hand side. Recall our 
assumption that $n_1>-c(w_\pi)-1$. If $n_1>-c(w_\pi)$, then the level of $w_0 \nu$ is $c(w_\pi)$, contradicting the left hand side; If $n_1=-c(w_\pi)$, $\nu=w_0^{-1}$ is level $-n_1$ and $w_0 \nu$ is the trivial character, still contradicting the left hand side.
\end{proof}
\begin{rem}\label{lvofcelvofrep}
With more careful arguments, one would expect that $c(w_\pi)\leq \frac{c}{2}$.
\end{rem}
For simplicity, we focus on the case when $w$ is unramified or level 1. Correspondingly $w_0=w_\pi|_{O_F^*}$ is trivial or level 1. Let $\F_v$ be the local field and $p$ be the characteristic of its residue field.
\begin{prop}\label{powerofmonomials}
Suppose that $c=-n_1\geq 2$ is the level of a supercuspidal representation $\pi$ whose central character is unramified or level 1. If $p\neq 2$ and $\nu$ is a level $i$ character, then we have $$n_\nu=\min\{-c,-2i\}.$$ 
For $p=2$ we have the same statement,except when $n_1\leq -4$ is an even integer and $i=-n_1/2$. In that case, we only claim $n_\nu\geq -c$.
\end{prop}
\begin{proof}
We will just apply (\ref{2ndrelation}) to different test functions and compare levels or supports of each sides. First of all, consider $\charf_{1,n}$ for $n\geq 0$. Then the right hand side of (\ref{2ndrelation}) will give a linear combination of $\charf_{\nu,-n+n_1}$ for all $\nu$ of level $-n_1+n$. The left hand side has an additional action of $\omega $, which has to maintain the right hand side. In particular, $n_\nu=2n_1-2n$ for all $\nu$ of level $-n_1+n$.

Then we consider those $\nu$ of level from 1 to $-n_1-1$. Suppose $1\leq i< c/2$. 

First we test $\charf_{1,-i}$. The situation for the left hand side is similar. The action of $\omega $ will change $\charf_{1,-i}$ into a multiple of $\charf_{w_0,i-c}$. Then $\psi(-x)\charf_{w_0,i-c}$ will be a linear combination of $\charf_{w_0\nu,i-c}$ for all $\nu$ of level $c-i$. After another action of $\omega $, what we get just consists of all level $c-i$ characters. 

On the right hand side, $\psi\charf_{1,-i}$ consists of all $\charf_{\mu, -i}$ for $\mu$ of level $i$. Again if $i=1$, it should be understood that $\mu$ is of level 1 or 0. 
If $n_\mu>-c$ for some $\mu$ of level $i$, then the action of $\omega $ will change $\charf_{\mu^{-1}, -i}$ into $\charf_{\mu w_0^{-1},i+n_\mu}$, so we know $\pi(\omega \zxz{1}{1}{0}{1})\charf_{1,-i}$ has level $i<c-i$ components at $v(x)=i+n_\mu>i-c$. Note that $\mu$ and $\psi(x)$ at $v(x)=i+n_\mu$ are both of level $<c-i$. Then multiplying another $\psi(x)$ will never give level $c-i$ components there, contradiction. So $n_\nu\leq -c$.

As a direct result of this, we also get $n_\nu\leq -2(c-i)$ for all $\nu$ of level $c-i$. This is because in our argument for $n_\mu\leq -c$, the right hand side will be  supported at $v(x)\leq i-c$. Then on the left hand side, the action of $\omega $ has to change all level $c-i$ components at $v(x)=i-c$ back to $v(x)\leq i-c$.

Now if $n_{\mu'}=n_{\mu'^{-1}w_0^{-1}}<-c$ for a $\mu'$ of level $i$, then we can test on $\charf_{\mu'^{-1},-i}$. The left hand side will give purely level $-n_{\mu'}-i>c-i$ components, supported at $v(x)\leq i+n_{\mu'}$ by what we just showed above. On the right hand side we still have level 0 component in $\psi\charf_{\mu',-i}$ as $\psi(x)$ at $v(x)=-i$ has $\mu'^{-1}$ component. The action of $\zxz{1}{1}{0}{1}\omega $ on this level 0 component will give a nonzero part for $v(x)=i+n_1=i-c$. This contradicts the support of the left hand side. So $n_\mu\geq -c$ for all $\mu$ of level $i$.

Combine the two arguments above, we can conclude that $n_\mu=-c$ for all $\mu$ of level $i$. One can also get $n_\nu=-2(c-i)$ for all $\nu$ of level $c-i$.

Now we only have to consider the "middle" level characters. Suppose that  $c$ is even and $\nu$ is a character of level $c/2$. The expected value of $n_\nu$ is still $-c$, as suggested by the statement of the lemma. While the argument above to prove $n_\nu\geq -c$ still works, the argument about $n_\nu\leq -c$ will fail. Indeed $\pi(\omega \zxz{1}{1}{0}{1})\charf_{1,-i}$ has level $i=c-i$ components.  Note that for $c=2$, $n_\nu\leq -2$ is automatic, so we can assume $c\geq 4$. 
We need a more accurate description of each side and try to compare the support in addition to the levels.

We know $n_\nu\geq -c$ for all $\nu$ of level $c/2\geq 2$. Suppose that  $n_{\nu'}>-c$ is the largest among all characters of level $c/2$. (Recall $n_\nu\leq -2$ for any $\nu$.) Choose the test function to be $\charf_{\nu'^{-1},-c/2}$. The action of $\omega $ will change it into a multiple of $\charf_{\nu'w_0,c/2+n_{\nu'}}$. Then $\psi(-x)\charf_{\nu'w_0,c/2+n_{\nu'}}$ consists of level $c/2$ characters that differ from $\nu'w_0$ by smaller level characters which are components of $\psi(-x)$ for $v(x)=c/2+n_{\nu'}>-c/2$. That is, it's a linear combination of $\charf_{\nu'w_0\nu,c/2+n_{\nu'}}$ for $\nu$ of level less than $c/2$ introduced by $\psi(-x)$.  When changed back by another $\omega $ action, what we get is a linear combination of characters $\nu'^{-1}\nu^{-1}$. 
They are supported on $v(x)\leq -c/2$, as we have assumed $n_{\nu'}$ is the largest.

On the right hand side, $\psi(x)\charf_{\nu'^{-1},-c/2}$ consists of all characters of level less than or equal to $c/2$, except those differ from $\nu'^{-1}$ by a lower level characters. After the action of $\omega $, they  will be supported on $v(x)\geq -c/2$, as we already know. Multiplying with another $\psi$ won't change the support. By comparing the supports on both sides, we can get $n_{\nu'\nu}=n_{\nu'}$ for all $\nu$ for left hand side introduced by $\psi(-x)$ at $v(x)=c/2+n_{\nu'}$. On the right hand side, we get $n_\mu=-c$ for all $\mu$ of level $c/2$ not differing from $\nu'$ by lower level characters. But we know $n_{\nu'}=n_{\nu'^{-1}w_0^{-1}}>-c$. To avoid contradiction, $\nu'^{-1}w_0^{-1}$ has to differ from $\nu'$ by lower level characters. This implies $\nu'^2$ itself is of lower level. But we will show this is impossible if $p\neq 2$ in the following lemma. Thus the proposition is proved.

\end{proof}
\begin{lem}
Suppose that  $p\neq 2$, and $\nu$ is a character of $O_F^*$ of level $n>1$. Then $\nu^2$ is still level $n$.
\end{lem}
\begin{proof}
Let $\varpi$ be a local uniformizer. $\nu$ being of level $n$ implies that there exist $b\in O_F^*$ such that $\nu(1+\varpi^{n-1}b)\neq 1$. If $p\neq 2$, then 2 is a unit. So
$$\nu^2(1+\varpi^{n-1}b/2)=\nu(1+\varpi^{n-1}b+\varpi^{2n-2}b^2/4)=\nu(1+\varpi^{n-1}b)\neq 1.$$
This means $\nu^2$ is still level $n$.
\end{proof}
\begin{rem}
When p=2, we can easily give a counterexample. Just consider the local field being $\Q_2$. Then $\Z_2^*$ has a unique level 2 character, whose square is the trivial character. 
\end{rem}
\begin{rem}
In general, we have $c(w_\pi)\leq \frac{c}{2}$ as suggested by Remark \ref{lvofcelvofrep}. We still expect $n_\nu=\min\{-c,-2i\}$ to hold for most cases, except when $i=c(w_\pi)=\frac{c}{2}$. Then we only expect $n_\nu\geq -c$.
\end{rem}
\begin{rem}
 In \cite{Yo77}, (\ref{singleaction}) was formulated (using (\ref{connectiontoepsilon}))  as
 $$\pi(\omega)\charf_{\lambda_0,n}=\epsilon(\pi\otimes\lambda^{-1},\psi)\charf_{w_0\lambda_0^{-1},-c(\pi\otimes\lambda^{-1})-n},$$
 where $\lambda_0=\lambda|_{O_F^*}$ and $c(\pi\otimes\lambda^{-1})$ is the level of $\pi\otimes\lambda^{-1}$. That is, $n_{\lambda_0}=-c(\pi\otimes\lambda^{-1})$. Let $\lambda$ be a character of $\F_v^*$ of level i. It was proved in \cite{Yo77} that if $i\leq c(\pi)$, then $c(\pi\otimes\lambda)\leq c(\pi)$; If $i>c(\pi)/2$, then $c(\pi\otimes\lambda)=2i$.
\end{rem}
\begin{cor}\label{cortoPropB3}
 Suppose a supercuspidal representation $\pi$ is of level $c$ with central character being unramified or level 1. Let $\lambda$ be a character of $\F_v^*$ of level $i$. Then $c(\pi\otimes\lambda)=\text{max}\{c,2i\}$.
\end{cor}


\end{document}